\documentclass[12pt, letterpaper, twoside]{amsart}
\usepackage[utf8]{inputenc}
\usepackage{amsmath}
\usepackage{amssymb}
\usepackage[usenames, dvipsnames]{xcolor}
\usepackage{amsfonts}
\usepackage{hyperref}
\usepackage{mathrsfs,mathtools}
\usepackage{tikz}
\usepackage{soul}

\tikzset{
    every node/.style={scale=.7}
}
\usetikzlibrary{positioning, arrows}
\usepackage[boxsize=1.25em]{ytableau} 

\newcommand{\gap}{\hspace{1in} \\ \vspace{-.2in}}

\newcommand*{\Scale}[2][4]{\scalebox{#1}{$#2$}}%

\title[Unique rectification in $d$-complete posets]{Unique rectification in $d$-complete posets: towards the $K$-theory of Kac-Moody flag varieties}
\author{Rahul Ilango}
\address{Deparment of Mathematics, Rutgers University, Piscataway, NJ 08854}
\email{rahul.ilango@rutgers.edu}

\author{Oliver Pechenik}
\address{Department of Mathematics, University of Michigan, Ann Arbor, MI, 48109}
\email{pechenik@umich.edu}

\author{Michael Zlatin}
\address{Department of Mathematics, Rutgers University, Piscataway, NJ 08854}
\email{mikhaelzlatin@gmail.com}
\date{\today}

\usepackage{amsthm}
\theoremstyle{plain}
\newtheorem{theorem}{Theorem}[section]
\newtheorem{lemma}[theorem]{Lemma}
\newtheorem{corollary}[theorem]{Corollary}
\newtheorem{conjecture}[theorem]{Conjecture}
\newtheorem{proposition}[theorem]{Proposition}
\newtheorem*{claim*}{Claim}
\newtheorem{claim}{Claim}
\numberwithin{claim}{theorem}

\numberwithin{equation}{section}

\theoremstyle{definition}
\newtheorem{definition}[theorem]{Definition}
\newtheorem{remark}[theorem]{Remark}

\newtheorem{examplex}[theorem]{Example}
\newenvironment{example}
  {\pushQED{\qed}\examplex}
  {\popQED\endexamplex}

\newcommand{\slantsum}[1]{\, _{#1}/ \, }
\newcommand{\range}{\ensuremath{\mathrm{Range}}}
\newcommand{\domain}{\ensuremath{\mathrm{Dom}}}
\newcommand{\ic}{\ensuremath{\mathrm{IC}}}

\begin{document}

\begin{abstract}
The jeu-de-taquin-based Littlewood-Richardson rule of H.~Thomas and A.~Yong (2009) for minuscule varieties has been extended in two orthogonal directions, either enriching the cohomology theory or else expanding the family of varieties considered. In one direction, A.~Buch and M.~Samuel (2016) developed a combinatorial theory of `unique rectification targets' in minuscule posets to extend the Thomas-Yong rule from ordinary cohomology to $K$-theory. Separately, P.-E.~Chaput and N.~Perrin (2012) used the combinatorics of R.~Proctor's `$d$-complete posets' to  extend the Thomas-Yong rule from minuscule varieties to a broader class of Kac-Moody structure constants.  We begin to address the unification of these theories. Our main result is the existence of unique rectification targets in a large class of $d$-complete posets. 
From this result, we obtain conjectural positive combinatorial formulas for certain $K$-theoretic Schubert structure constants in the Kac-Moody setting.
\end{abstract}

\maketitle

\section{Introduction}
The 1970s saw a major advance in the combinatorial approach to enumerative geometry when M.-P.~Sch\"{u}tzenberger proved the Littlewood-Richardson rule for describing the cohomology rings of Grassmannians. Since then, the modern Schubert calculus has turned to extending this understanding in two different directions: on the one hand to replace the Grassmannian with a more complicated homogeneous space, and on the other hand to replace ordinary cohomology with a richer generalized cohomology theory. Along these lines, the goal of this paper is to begin unraveling the $K$-theoretic Schubert calculus of Kac-Moody homogeneous spaces. Our results are purely combinatorial in nature, but allow us to conjecture explicit Littlewood-Richardson-style rules in this geometric context.

Let $G$ be a complex Kac-Moody group with Borel and opposite Borel subgroups $B_+$ and $B_-$, respectively.  Let $B_+ \subseteq P \subset G$ be a parabolic subgroup. The homogeneous space $X = G/P$ is a {\bf Kac-Moody flag variety}. The Zariski closures of the \mbox{$B_-$-orbits} are the {\bf Schubert varieties} $\{ X_w \}_{w \in W^P}$ and give a cell decomposition of $X$; here, $W^P$ denotes the set of minimal-length representatives of the quotient $W/W_P$, where $W$ is the Weyl group of $G$ and $W_P$ is the parabolic Weyl group for $P$.  The cohomology ring $H^\star(G/P)$ thereby has a distinguished Schubert basis $\{ \sigma_w \}_{w \in W^P}$, where $\sigma_w$ is Poincar\'e dual to $X_w$. Thus, to determine multiplication in $H^\star(X)$, it suffices to determine the {\bf Schubert structure constants} $c_{u,v}^w$ defined by 
\begin{equation}\label{eq:LR}
\sigma_u \cdot \sigma_v = \sum_{w \in W^P} c_{u,v}^w \sigma_w.
\end{equation}

In the case that $X = {\rm Gr}_k(\mathbb{C}^n)$ is a Grassmannian, the parameter space of $k$-dimensional linear subspaces of $\mathbb{C}^n$, this problem is solved in a positive combinatorial manner by any of the various Littlewood-Richardson rules (e.g., \cite{Littlewood.Richardson,Sch77,Vakil}). For a general Kac-Moody flag variety $X$, these $c_{u,v}^w$ are also non-negative integers, but it is generally a major open problem to give an analogous Littlewood-Richardson-style rule to determine them. 

For $X = {\rm Gr}_k(\mathbb{C}^n)$, M.-P.~Sch\"{u}tzenberger's Littlewood-Richardson rule is stated in terms of the \emph{jeu de taquin} for standard Young tableaux \cite{Sch77} fitting inside a $k \times (n-k)$ rectangle. One may realize this rectangle as a subposet of positive roots for ${\rm GL}_n(\mathbb{C})$ in such a way that the inversion set of $w \in W^P$ is an order ideal in this subposet. One may further realize standard Young tableaux as linear extensions of intervals in this poset. Using this perspective, H.~Thomas and A.~Yong \cite{Thomas.Yong:minuscule} gave a uniform extension of Sch\"{u}tzenberger's rule to compute all cohomological Schubert structure constants for the larger family of \emph{minuscule varieties}. This was further extended by P.-E.~Chaput and N.~Perrin \cite{CP12} to a positive combinatorial formula for computing certain \emph{$\Lambda$-minuscule} Schubert structure constants for general Kac-Moody $X$. In the Chaput-Perrin rule, the role of the $k \times (n-k)$ rectangle is played by the \emph{$d$-complete posets} introduced by R.~Proctor \cite{Proctor:JACO,Proctor:algebra}; $d$-complete posets are exactly those posets encoding the containment relations among $\Lambda$-minuscule Schubert varieties.

Much work in the modern Schubert calculus has been devoted to studying homogeneous spaces through richer cohomology theories. In these theories, there are Schubert bases analogous to the cohomological $\sigma_w$ and the structure constants defined analogously to Equation~(\ref{eq:LR}) enjoy various positivity properties. Hence, it makes sense to attempt to develop positive combinatorial formulas for these structure constants in the style of the classical Littlewood-Richardson rules. In the Grassmannian case, one has, for example: the equivariant cohomology rule of A.~Knutson and T.~Tao \cite{Knutson.Tao}, the $K$-theory rule of A.~Buch \cite{Buch:K}; the equivariant $K$-theory rule of O.~Pechenik and A.~Yong \cite{Pechenik.Yong:KT}; the quantum cohomology rule of A.~Buch, A.~Kresch, K.~Purbhoo, and H.~Tamvakis \cite{Buch.Kresch.Purbhoo.Tamvakis}; and the equivariant quantum cohomology rule of A.~Buch \cite{Buch:quantum}. Our interest is in the ordinary $K$-theory ring $K(X)$ of the Kac-Moody flag variety $X$, where the $K$-theoretic Schubert classes $\{[\mathcal{O}_{X_w}]\}_{w \in W^P}$ are represented by the structure sheaves of the Schubert varieties. Specifically, we are interested in the structure constants $K_{u,v}^w$ of $K(X)$ defined by
\begin{equation}\label{eq:KLR}
[\mathcal{O}_{X_u}] \cdot [\mathcal{O}_{X_v}] = \sum_{w \in W^P} K_{u,v}^w [\mathcal{O}_{X_w}].
\end{equation}

For Grassmannians, various alternatives to Buch's original rule \cite{Buch:K} for $K_{u,v}^w$ are now known \cite{Vakil,Thomas.Yong:K,Pechenik.Yong:genomic}. However, only the rule of H.~Thomas and A.~Yong \cite{Thomas.Yong:K} is currently known to extend to all of the minuscule varieties \cite{Buch.Ravikumar,Clifford.Thomas.Yong,BS16}. This Thomas-Yong rule is based on a jeu de taquin theory for \emph{increasing tableaux}. This combinatorial theory displays a number of additional subtleties when compared to Sch\"{u}tzenberger's jeu de taquin for standard tableaux. In particular, a key ingredient is the need to identify increasing tableaux with the \emph{unique rectification target} property. (These combinatorial notions are reviewed in Section~\ref{Section posets, skew shapes, and rectifications}.)

In \cite[Problem 9.1]{Thomas.Yong:K} and \cite[Remark 3.24]{BS16}, the authors ask to what extent their combinatorial theory extends to the case of $d$-complete posets. The missing ingredient is that it is not currently known whether general $d$-complete posets have ``enough'' unique rectification targets. We conjecture, however, that they do.
\begin{conjecture}\label{conj:URTs}
Let $\mathcal{P}$ be a $d$-complete poset and let $\lambda \subseteq \mathcal{P}$ be an order ideal. Then there is an (explicitly-defined) unique rectification target supported on $\lambda$.
\end{conjecture}
We initiate a study of the existence and structure of unique rectification targets in the $d$-complete posets. As shown by R.~Proctor \cite{Proctor:JACO}, every $d$-complete poset can be constructed by gluing together (in prescribed ways) certain irreducible $d$-complete posets. These irreducible pieces are classified in \cite{Proctor:JACO} and include all of the minuscule posets (i.e., the posets describing the Schubert stratification of minuscule varieties).
The informal version of our main result is the following special case of Conjecture~\ref{conj:URTs}.
\begin{theorem}\label{thm:main}
Conjecture~\ref{conj:URTs} holds in the case that $\mathcal{P}$ is built from minuscule posets.
\end{theorem}
\begin{figure}[ht]
\begin{tikzpicture}[every node/.append style={circle, draw=black, inner sep=0pt, minimum size=10pt}, every draw/.append style={black, thick},anchor=base,baseline,node distance=.8cm]
    \node (Shape1a1) at (0,0) {};
    \node [above left of=Shape1a1] (Shape1a2) {};
    \node [above left of=Shape1a2] (Shape1a3) {};
    \node [above left of=Shape1a3] (Shape1a4) {};
    \node [above right of=Shape1a1](Shape1d1) {};
    \node [above left of=Shape1d1] (Shape1d2) {};
    \node [above left of=Shape1d2] (Shape1d3) {};
    \node [above left of=Shape1d3] (Shape1d4) {};
    \draw (Shape1a1)--(Shape1a2);
    \draw (Shape1a2)--(Shape1a3);
    \draw (Shape1a3)--(Shape1a4);
    
    \draw (Shape1d1)--(Shape1d2);
    \draw (Shape1d2)--(Shape1d3);
    \draw (Shape1d3)--(Shape1d4);
    \draw (Shape1a1)--(Shape1d1);
    \draw (Shape1a2)--(Shape1d2);
    \draw (Shape1a3)--(Shape1d3);
    \draw (Shape1a4)--(Shape1d4);

    \node [above left = 1cm and 2.8cm of Shape1a4](Bat1a1) {};
    \node [above left of=Bat1a1](Bat1a2)  {};
    \node [above left of=Bat1a2](Bat1a3)  {};
    \node [above left of=Bat1a3](Bat1a4)  {};
    \node [above left of=Bat1a4](Bat1a5)  {};
    \node [above left of=Bat1a5](Bat1a6)  {};
    \node [above right of=Bat1a4](Bat1b1)  {};
    \node [above left of=Bat1b1](Bat1b2)  {};
    \node [above left of=Bat1b2](Bat1b3)  {};
    \node [above right of=Bat1b2](Bat1c1)  {};
    \node [above left of=Bat1c1](Bat1c2)  {};
    \node [above left of=Bat1c2](Bat1c3)  {};
    \node [above right of=Bat1c1](Bat1d1)  {};
    \node [above left of=Bat1d1](Bat1d2)  {};
    \node [above left of=Bat1d2](Bat1d3)  {};
    \node [above left of=Bat1d3](Bat1d4)  {};
    \node [above left of=Bat1d4](Bat1d5)  {};
    \node [above right of=Bat1d1](Bat1e1)  {};
    \node [above left of=Bat1e1](Bat1e2)  {};
    \node [above left of=Bat1e2](Bat1e3)  {};
    \node [above left of=Bat1e3](Bat1e4)  {};
    \node [above left of=Bat1e4](Bat1e5)  {};
    \node [above right of=Bat1e4](Bat1f1)  {};
    \node [above left of=Bat1f1](Bat1f2)  {};
    \node [above of=Bat1f2](Bat1f3)  {};
    \node [above of=Bat1f3](Bat1f4)  {};
    \node [above of=Bat1f4](Bat1f5)  {};
    \draw (Shape1a4)--(Bat1a1);
    \draw (Bat1a1)--(Bat1a2);
    \draw (Bat1a2)--(Bat1a3);
    \draw (Bat1a3)--(Bat1a4);
    \draw (Bat1a4)--(Bat1a5);
    \draw (Bat1a5)--(Bat1a6);
    \draw (Bat1b1)--(Bat1b2);
    \draw (Bat1b2)--(Bat1b3);
    \draw (Bat1c1)--(Bat1c2);
    \draw (Bat1c2)--(Bat1c3);
    \draw (Bat1d1)--(Bat1d2);
    \draw (Bat1d2)--(Bat1d3);
    \draw (Bat1d3)--(Bat1d4);
    \draw (Bat1d4)--(Bat1d5);
    \draw (Bat1e1)--(Bat1e2);
    \draw (Bat1e2)--(Bat1e3);
    \draw (Bat1e3)--(Bat1e4);
    \draw (Bat1e4)--(Bat1e5);
    \draw (Bat1f1)--(Bat1f2);
    \draw (Bat1a4)--(Bat1b1);
    \draw (Bat1a5)--(Bat1b2);
    \draw (Bat1a6)--(Bat1b3);
    \draw (Bat1c1)--(Bat1b2);
    \draw (Bat1c2)--(Bat1b3);
    \draw (Bat1c1)--(Bat1d1);
    \draw (Bat1c2)--(Bat1d2);
    \draw (Bat1c3)--(Bat1d3);
    \draw (Bat1e1)--(Bat1d1);
    \draw (Bat1e2)--(Bat1d2);
    \draw (Bat1e3)--(Bat1d3);
    \draw (Bat1e4)--(Bat1d4);
    \draw (Bat1e5)--(Bat1d5);
    \draw (Bat1e4)--(Bat1f1);
    \draw (Bat1e5)--(Bat1f2);
    \draw (Bat1f3)--(Bat1f2);
    \draw (Bat1f3)--(Bat1f4);
    \draw (Bat1f5)--(Bat1f4);

	\node [above = 1cm of Shape1a4] (Diamond1b2)  {};
    \node [circle,above of=Diamond1b2] (Diamond1b0) {};
    \node [circle,above of=Diamond1b0] (Diamond1bottom) {};
    \node [circle,above left of=Diamond1bottom] (Diamond1left)  {};
    \node [circle,above right of=Diamond1bottom] (Diamond1right) {};
    \node [circle, above right of=Diamond1left] (Diamond1top) {};
    \node [circle,above of =Diamond1top] (Diamond1t0)  {};
    \node [circle, above of =Diamond1t0] (Diamond1t2)  {};
    \draw (Diamond1t2) -- (Diamond1t0);
    \draw (Diamond1top) -- (Diamond1t0);
    \draw (Diamond1top) -- (Diamond1left);
    \draw (Diamond1top) -- (Diamond1right);
    \draw  (Diamond1right) -- (Diamond1bottom);
    \draw (Diamond1left) -- (Diamond1bottom);
    \draw  (Diamond1b0) -- (Diamond1bottom);
    \draw  (Diamond1b2) -- (Diamond1b0);
    \draw (Diamond1b2)--(Shape1a4);
    
    \node [above left= 1cm of Diamond1left] (ShiftedShape1a1) {};
  \node [above left of=ShiftedShape1a1] (ShiftedShape1b1) {};
    \node [above left of=ShiftedShape1b1] (ShiftedShape1c1) {};
    \node [above left of=ShiftedShape1c1] (ShiftedShape1d1) {};
    \node [above left of=ShiftedShape1d1] (ShiftedShape1e1) {};
    \node [above left of=ShiftedShape1e1] (ShiftedShape1f1) {};
    \node [above right of=ShiftedShape1b1] (ShiftedShape1b2) {};
    \node [above right of=ShiftedShape1c1] (ShiftedShape1c2) {};
    \node [above right of=ShiftedShape1c2] (ShiftedShape1c3) {};
    \node [above right of=ShiftedShape1d1] (ShiftedShape1d2) {};
    \node [above right of=ShiftedShape1d2] (ShiftedShape1d3) {};
    \node [above right of=ShiftedShape1d3] (ShiftedShape1d4) {};
    \node [above right of=ShiftedShape1e1] (ShiftedShape1e2) {};
    \node [above right of=ShiftedShape1e2] (ShiftedShape1e3) {};
    \node [above right of=ShiftedShape1e3] (ShiftedShape1e4) {};
    \node [above right of=ShiftedShape1e4] (ShiftedShape1e5) {};
    \node [above right of=ShiftedShape1f1] (ShiftedShape1f2) {};
    \node [above right of=ShiftedShape1f2] (ShiftedShape1f3) {};
    \node [above right of=ShiftedShape1f3] (ShiftedShape1f4) {};
    \node [above right of=ShiftedShape1f4] (ShiftedShape1f5) {};

    \draw (ShiftedShape1a1) -- (ShiftedShape1b1);
    \draw (ShiftedShape1b1) -- (ShiftedShape1c1);
    \draw (ShiftedShape1c1) -- (ShiftedShape1d1);
    \draw (ShiftedShape1d1) -- (ShiftedShape1e1);
    \draw (ShiftedShape1e1) -- (ShiftedShape1f1);
    \draw (ShiftedShape1b2) -- (ShiftedShape1c2);
    \draw (ShiftedShape1c2) -- (ShiftedShape1d2);
    \draw (ShiftedShape1d2) -- (ShiftedShape1e2);
    \draw (ShiftedShape1e2) -- (ShiftedShape1f2);
    \draw (ShiftedShape1c3) -- (ShiftedShape1d3);
    \draw (ShiftedShape1d3) -- (ShiftedShape1e3);
    \draw (ShiftedShape1d4) -- (ShiftedShape1e4);
    \draw (ShiftedShape1e3) -- (ShiftedShape1f3);
    \draw (ShiftedShape1e4) -- (ShiftedShape1f4);
    \draw (ShiftedShape1e5) -- (ShiftedShape1f5);
    \draw (ShiftedShape1b1) -- (ShiftedShape1b2);
    \draw (ShiftedShape1c1) -- (ShiftedShape1c2);
    \draw (ShiftedShape1c3) -- (ShiftedShape1c2);
    \draw (ShiftedShape1d1) -- (ShiftedShape1d2);
    \draw (ShiftedShape1d3) -- (ShiftedShape1d2);
    \draw (ShiftedShape1d3) -- (ShiftedShape1d4);
    \draw (ShiftedShape1e1) -- (ShiftedShape1e2);
    \draw (ShiftedShape1e3) -- (ShiftedShape1e2);
    \draw (ShiftedShape1e3) -- (ShiftedShape1e4);
    \draw (ShiftedShape1e5) -- (ShiftedShape1e4);
    \draw (ShiftedShape1f1) -- (ShiftedShape1f2);
    \draw (ShiftedShape1f3) -- (ShiftedShape1f2);
    \draw (ShiftedShape1f3) -- (ShiftedShape1f4);
    \draw (ShiftedShape1f5) -- (ShiftedShape1f4);
    \draw (Diamond1left) -- (ShiftedShape1a1);
    
    \node [above right= 1cm of Diamond1right](Cayley1a1){};
    \node [above right of=Cayley1a1](Cayley1a2)  {};
    \node [above right of=Cayley1a2](Cayley1a3)  {};
    \node [above right of=Cayley1a3](Cayley1a4)  {};
    \node [above right of=Cayley1a4](Cayley1a5)  {};
    \node [above left of=Cayley1a3](Cayley1b1)  {};
    \node [above right of=Cayley1b1](Cayley1b2)  {};
    \node [above right of=Cayley1b2](Cayley1b3)  {};
    \node [above left of=Cayley1b2](Cayley1c1)  {};
    \node [above right of=Cayley1c1](Cayley1c2)  {};
    \node [above right of=Cayley1c2](Cayley1c3)  {};
    \node [above left of=Cayley1c1](Cayley1d1)  {};
    \node [above right of=Cayley1d1](Cayley1d2)  {};
    \node [above right of=Cayley1d2](Cayley1d3)  {};
    \draw (Cayley1a1)--(Cayley1a2);
    \draw (Cayley1a2)--(Cayley1a3);
    \draw (Cayley1a3)--(Cayley1a4);
    \draw (Cayley1a4)--(Cayley1a5);
    \draw (Cayley1b1)--(Cayley1b2);
    \draw (Cayley1b2)--(Cayley1b3);
    \draw (Cayley1c1)--(Cayley1c2);
    \draw (Cayley1c2)--(Cayley1c3);
    \draw (Cayley1d1)--(Cayley1d2);
    \draw (Cayley1d2)--(Cayley1d3);
    \draw (Cayley1a3)--(Cayley1b1);
    \draw (Cayley1a4)--(Cayley1b2);
    \draw (Cayley1a5)--(Cayley1b3);
    \draw (Cayley1c1)--(Cayley1b2);
    \draw (Cayley1c2)--(Cayley1b3);
    \draw (Cayley1c1)--(Cayley1d1);
    \draw (Cayley1c2)--(Cayley1d2);
    \draw (Cayley1c3)--(Cayley1d3);
	\draw (Diamond1right) -- (Cayley1a1);
    
    \node [above right =1cm and .8cm of Shape1d1](Cayley2a1) {};
    \node [above of=Cayley2a1](Cayley2a2)  {};
    \node [above of=Cayley2a2](Cayley2a3)  {};
    \node [above left of=Cayley2a3](Cayley2a4)  {};
    \node [above left of=Cayley2a4](Cayley2a5)  {};
    \node [above right of=Cayley2a3](Cayley2b1)  {};
    \node [above left of=Cayley2b1](Cayley2b2)  {};
    \node [above left of=Cayley2b2](Cayley2b3)  {};
    \node [above right of=Cayley2b2](Cayley2c1)  {};
    \node [above left of=Cayley2c1](Cayley2c2)  {};
    \node [above left of=Cayley2c2](Cayley2c3)  {};
    \node [above right of=Cayley2c1](Cayley2d1)  {};
    \node [above left of=Cayley2d1](Cayley2d2)  {};
    \node [above left of=Cayley2d2](Cayley2d3)  {};
    \node [above right of=Cayley2d3](Cayley2d4)  {};
    \draw (Cayley2a1)--(Cayley2a2);
    \draw (Cayley2a2)--(Cayley2a3);
    \draw (Cayley2a3)--(Cayley2a4);
    \draw (Cayley2a4)--(Cayley2a5);
    \draw (Cayley2b1)--(Cayley2b2);
    \draw (Cayley2b2)--(Cayley2b3);
    \draw (Cayley2c1)--(Cayley2c2);
    \draw (Cayley2c2)--(Cayley2c3);
    \draw (Cayley2d1)--(Cayley2d2);
    \draw (Cayley2d2)--(Cayley2d3);
    \draw (Cayley2d3)--(Cayley2d4);
    \draw (Cayley2a3)--(Cayley2b1);
    \draw (Cayley2a4)--(Cayley2b2);
    \draw (Cayley2a5)--(Cayley2b3);
    \draw (Cayley2c1)--(Cayley2b2);
    \draw (Cayley2c2)--(Cayley2b3);
    \draw (Cayley2c1)--(Cayley2d1);
    \draw (Cayley2c2)--(Cayley2d2);
    \draw (Cayley2c3)--(Cayley2d3);
    \draw (Shape1d1)--(Cayley2a1);
    
    \node [above right =1cm and 2.1cm of Shape1d1](ShiftedShape2a1)  {};
    \node [above right of=ShiftedShape2a1] (ShiftedShape2b1) {};
    \node [above right of=ShiftedShape2b1] (ShiftedShape2c1) {};
    \node [above right of=ShiftedShape2c1] (ShiftedShape2d1) {};
    \node [above right of=ShiftedShape2d1] (ShiftedShape2e1) {};
    \node [above right of=ShiftedShape2e1] (ShiftedShape2f1) {};
    \node [above left of=ShiftedShape2b1] (ShiftedShape2b2) {};
    \node [above left of=ShiftedShape2c1] (ShiftedShape2c2) {};
    \node [above left of=ShiftedShape2c2] (ShiftedShape2c3) {};
    \node [above left of=ShiftedShape2d1] (ShiftedShape2d2) {};
    \node [above left of=ShiftedShape2d2] (ShiftedShape2d3) {};
    \node [above left of=ShiftedShape2d3] (ShiftedShape2d4) {};
    \node [above left of=ShiftedShape2e1] (ShiftedShape2e2) {};
    \node [above left of=ShiftedShape2e2] (ShiftedShape2e3) {};
    \node [above left of=ShiftedShape2e3,] (ShiftedShape2e4) {};
    \node [above left of=ShiftedShape2e4] (ShiftedShape2e5) {};
    \node [above left of=ShiftedShape2f1] (ShiftedShape2f2) {};
    \node [above left of=ShiftedShape2f2] (ShiftedShape2f3) {};
    \node [above left of=ShiftedShape2f3] (ShiftedShape2f4) {};
    \node [above left of=ShiftedShape2f4] (ShiftedShape2f5) {};
    \node [above left of=ShiftedShape2f5] (ShiftedShape2f6) {};
    \draw (ShiftedShape2a1) -- (ShiftedShape2b1);
    \draw (ShiftedShape2b1) -- (ShiftedShape2c1);
    \draw (ShiftedShape2c1) -- (ShiftedShape2d1);
    \draw (ShiftedShape2d1) -- (ShiftedShape2e1);
    \draw (ShiftedShape2e1) -- (ShiftedShape2f1);
    \draw (ShiftedShape2b2) -- (ShiftedShape2c2);
    \draw (ShiftedShape2c2) -- (ShiftedShape2d2);
    \draw (ShiftedShape2d2) -- (ShiftedShape2e2);
    \draw (ShiftedShape2e2) -- (ShiftedShape2f2);
    \draw (ShiftedShape2c3) -- (ShiftedShape2d3);
    \draw (ShiftedShape2d3) -- (ShiftedShape2e3);
    \draw (ShiftedShape2e3) -- (ShiftedShape2f3);
    \draw (ShiftedShape2e5) -- (ShiftedShape2f5);
    \draw (ShiftedShape2b1) -- (ShiftedShape2b2);
    \draw (ShiftedShape2c1) -- (ShiftedShape2c2);
    \draw (ShiftedShape2c3) -- (ShiftedShape2c2);
    \draw (ShiftedShape2d1) -- (ShiftedShape2d2);
    \draw (ShiftedShape2d3) -- (ShiftedShape2d2);
    \draw (ShiftedShape2d3) -- (ShiftedShape2d4);
    \draw (ShiftedShape2e1) -- (ShiftedShape2e2);
    \draw (ShiftedShape2e3) -- (ShiftedShape2e2);
    \draw (ShiftedShape2e3) -- (ShiftedShape2e4);
    \draw (ShiftedShape2e5) -- (ShiftedShape2e4);
    \draw (ShiftedShape2f1) -- (ShiftedShape2f2);
    \draw (ShiftedShape2f3) -- (ShiftedShape2f2);
    \draw (ShiftedShape2f3) -- (ShiftedShape2f4);
    \draw (ShiftedShape2f5) -- (ShiftedShape2f4);
    \draw (ShiftedShape2f5) -- (ShiftedShape2f6);
    \draw (ShiftedShape2d4) -- (ShiftedShape2e4);
    \draw (ShiftedShape2e4) -- (ShiftedShape2f4);
    \draw (Shape1d1)--(ShiftedShape2a1);

	\node [above = 0.9cm of ShiftedShape2f1](Bat2a1){};
    \node [above left of=Bat2a1](Bat2a2)  {};
    \node [above left of=Bat2a2](Bat2a3)  {};
    \node [above left of=Bat2a3](Bat2a4)  {};
    \node [above left of=Bat2a4](Bat2a5)  {};
    \node [above left of=Bat2a5](Bat2a6)  {};
    \node [above right of=Bat2a4](Bat2b1)  {};
    \node [above left of=Bat2b1](Bat2b2)  {};
    \node [above left of=Bat2b2](Bat2b3)  {};
    \node [above right of=Bat2b2](Bat2c1)  {};
    \node [above left of=Bat2c1](Bat2c2)  {};
    \node [above left of=Bat2c2](Bat2c3)  {};
    \node [above right of=Bat2c1](Bat2d1)  {};
    \node [above left of=Bat2d1](Bat2d2)  {};
    \node [above left of=Bat2d2](Bat2d3)  {};
    \node [above left of=Bat2d3](Bat2d4)  {};
    \node [above left of=Bat2d4](Bat2d5)  {};
    \node [above right of=Bat2d1](Bat2e1)  {};
    \node [above left of=Bat2e1](Bat2e2)  {};
    \node [above left of=Bat2e2](Bat2e3)  {};
    \node [above left of=Bat2e3](Bat2e4)  {};
    \node [above left of=Bat2e4](Bat2e5)  {};
    \node [above right of=Bat2e4](Bat2f1)  {};
    \node [above left of=Bat2f1](Bat2f2)  {};
    \draw (Bat2a1)--(Bat2a2);
    \draw (Bat2a2)--(Bat2a3);
    \draw (Bat2a3)--(Bat2a4);
    \draw (Bat2a4)--(Bat2a5);
    \draw (Bat2a5)--(Bat2a6);
    \draw (Bat2b1)--(Bat2b2);
    \draw (Bat2b2)--(Bat2b3);
    \draw (Bat2c1)--(Bat2c2);
    \draw (Bat2c2)--(Bat2c3);
    \draw (Bat2d1)--(Bat2d2);
    \draw (Bat2d2)--(Bat2d3);
    \draw (Bat2d3)--(Bat2d4);
    \draw (Bat2d4)--(Bat2d5);
    \draw (Bat2e1)--(Bat2e2);
    \draw (Bat2e2)--(Bat2e3);
    \draw (Bat2e3)--(Bat2e4);
    \draw (Bat2e4)--(Bat2e5);
    \draw (Bat2f1)--(Bat2f2);
    \draw (Bat2a4)--(Bat2b1);
    \draw (Bat2a5)--(Bat2b2);
    \draw (Bat2a6)--(Bat2b3);
    \draw (Bat2c1)--(Bat2b2);
    \draw (Bat2c2)--(Bat2b3);
    \draw (Bat2c1)--(Bat2d1);
    \draw (Bat2c2)--(Bat2d2);
    \draw (Bat2c3)--(Bat2d3);
    \draw (Bat2e1)--(Bat2d1);
    \draw (Bat2e2)--(Bat2d2);
    \draw (Bat2e3)--(Bat2d3);
    \draw (Bat2e4)--(Bat2d4);
    \draw (Bat2e5)--(Bat2d5);
    \draw (Bat2e4)--(Bat2f1);
    \draw (Bat2e5)--(Bat2f2);
    \draw (ShiftedShape2f1)--(Bat2a1);
    
    \node [above right= 1cm and 1cm of ShiftedShape2f1](Shape2a1) {};
    \node [above left of=Shape2a1] (Shape2a2) {};
    \node [above left of=Shape2a2] (Shape2a3) {};
    \node [above left of=Shape2a3] (Shape2a4) {};
    \node [above right of=Shape2a1](Shape2b1) {};
    \node [above left of=Shape2b1] (Shape2b2) {};
    \node [above left of=Shape2b2] (Shape2b3) {};
    \node [above left of=Shape2b3] (Shape2b4) {};
    \node [above right of=Shape2b1](Shape2c1) {};
    \node [above left of=Shape2c1] (Shape2c2) {};
    \node [above left of=Shape2c2] (Shape2c3) {};
    \node [above left of=Shape2c3] (Shape2c4) {};
    \node [above right of=Shape2c1](Shape2d1) {};
    \node [above left of=Shape2d1] (Shape2d2) {};
    \node [above left of=Shape2d2] (Shape2d3) {};
    \node [above left of=Shape2d3] (Shape2d4) {};
    \draw (Shape2a1)--(Shape2a2);
    \draw (Shape2a2)--(Shape2a3);
    \draw (Shape2a3)--(Shape2a4);
    \draw (Shape2b1)--(Shape2b2);
    \draw (Shape2b2)--(Shape2b3);
    \draw (Shape2b3)--(Shape2b4);
    \draw (Shape2c1)--(Shape2c2);
    \draw (Shape2c2)--(Shape2c3);
    \draw (Shape2c3)--(Shape2c4);
    \draw (Shape2d1)--(Shape2d2);
    \draw (Shape2d2)--(Shape2d3);
    \draw (Shape2d3)--(Shape2d4);
    \draw (Shape2a1)--(Shape2b1);
    \draw (Shape2a2)--(Shape2b2);
    \draw (Shape2a4)--(Shape2b4);
    \draw (Shape2c1)--(Shape2b1);
    \draw (Shape2c2)--(Shape2b2);
    \draw (Shape2c4)--(Shape2b4);
    \draw (Shape2c1)--(Shape2d1);
    \draw (Shape2c2)--(Shape2d2);
    \draw (Shape2c4)--(Shape2d4);
    \draw (Shape2a3)--(Shape2b3);
    \draw (Shape2c3)--(Shape2b3);
    \draw (Shape2c3)--(Shape2d3);
    \draw (ShiftedShape2f1)--(Shape2a1);
\end{tikzpicture}
\caption{The Hasse diagram of a representative $d$-complete poset $\mathcal{P}$ that is ``built from minuscule posets'' in the sense of Theorem~\ref{thm:main}. In $\mathcal{P}$, every order ideal $\lambda \subseteq \mathcal{P}$ has a unique rectification target, provided by Theorem~\ref{thm:main}.}
\label{fig:big_poset}
\end{figure}
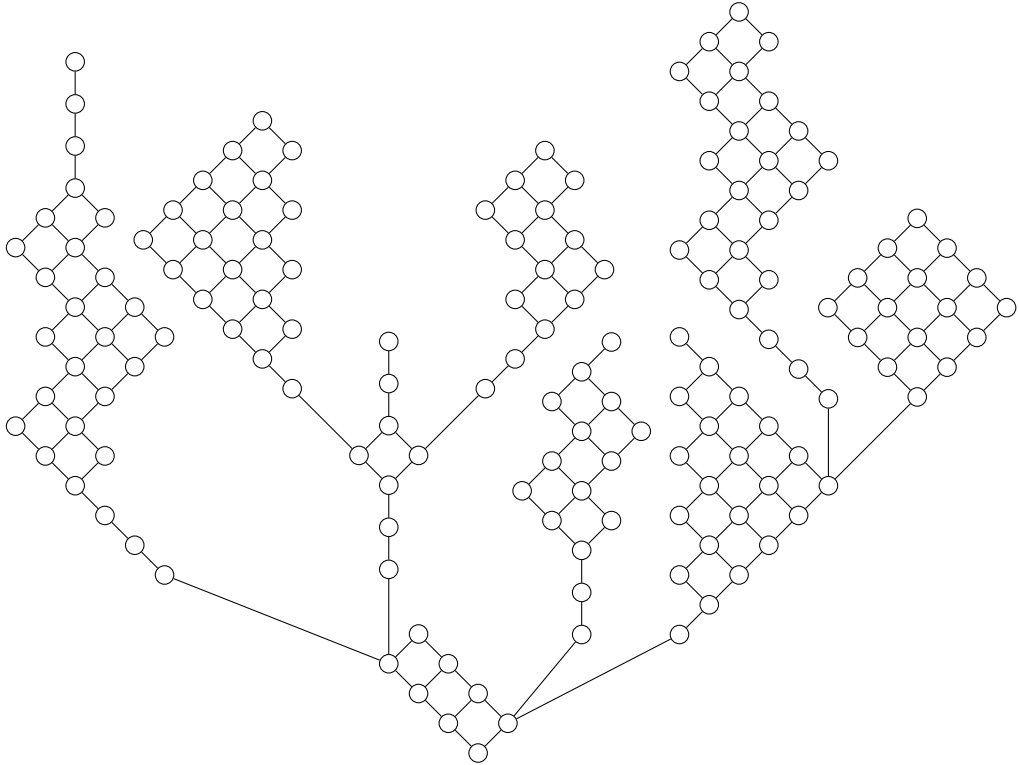

For an example of such a poset covered by Theorem~\ref{thm:main}, see Figure~\ref{fig:big_poset}.
We also demonstrate the extent to which Conjecture~\ref{conj:URTs} is sensitive to the poset $\mathcal{P}$ being $d$-complete. We establish general results on the failure of Conjecture~\ref{conj:URTs} to extend to posets that are slight deformations of $d$-complete posets.

For any $\mathcal{P}$ satisfying Conjecture~\ref{conj:URTs}, one obtains (as in \cite[\textsection 3.5]{BS16}) a corresponding combinatorially-defined associative commutative unital algebra $K(\mathcal{P})$ with a basis $\{ \lambda \}$ indexed by order ideals of $\mathcal{P}$. The structure constants $t_{\lambda,\mu}^\nu$ of $K(\mathcal{P})$ are defined in such a way as to transparently alternate in degree. (This construction is discussed in Section~\ref{Section $d$-complete posets}.) 
For $w \in W^P$ $\Lambda$-minuscule, the interval $[{\rm id}, w]$ in Bruhat order is isomorphic to the poset of order ideals of a certain $d$-complete poset $\mathcal{P}_w$ constructed from $w$.
Building on Conjecture~\ref{conj:URTs}, we propose the following.

\begin{conjecture}\label{conj:geometry}
Let $X=G/P$ be a Kac-Moody flag variety and let $m \in W^P$ be $\Lambda$-minuscule. 
Then, for $u,v,w \leq m$ in Bruhat order, we have the equality of structure constants \[
K_{u,v}^w  = t_{\lambda,\mu}^\nu, 
\]
between the rings $K(X)$ and $K(\mathcal{P}_m)$,
where the order ideals $\lambda, \mu, \nu \subseteq \mathcal{P}_m$ correspond to the the Weyl group elements $u,v,w \in W^P$, respectively.
\end{conjecture}

In Section~\ref{Section $d$-complete posets}, we give the precise versions of Conjecture~\ref{conj:URTs} and Theorem~\ref{thm:main}, as well as the details necessary for a precise understanding of Conjecture~\ref{conj:geometry}.
In light of Conjecture~\ref{conj:geometry}, Theorem~\ref{thm:main} should be understood as giving a conjectural positive combinatorial rule for certain $K$-theoretic Schubert structure constants of Kac-Moody flag varieties.
Several cases of Conjecture~\ref{conj:geometry} are known to be true or have been previously conjectured. If the flag variety $X$ is minuscule, then Conjecture~\ref{conj:geometry} reduces to the main theorem of \cite{BS16}. If, on the other hand, $X$ is general but $|\nu| - |\lambda| - |\mu| = 0$, then Conjecture~\ref{conj:geometry} reduces to \cite[Conjecture~1.1]{CP12}, many cases of which are proved in \cite[Theorem~1.3]{CP12}. Assuming one followed the general structure utilized by \cite{CP12,BS16}, the main ingredients one would need in a proof of Conjecture~\ref{conj:geometry} are 
\begin{itemize}
\item[(1.)] a proof of the remaining cases of Conjecture~\ref{conj:URTs} and
\item[(2.)] {\it ad hoc} geometric verifications of Conjecture~\ref{conj:geometry} for special $u$ lying in a generating set of classes.
\end{itemize}
For a large class of such Schubert problems, Theorem~\ref{thm:main} provides the necessary first ingredient, so it only remains to establish the second in those cases.

Another potential application of Theorem~\ref{thm:main} (or more generally Conjecture~\ref{conj:URTs}) is to establishing plane partition identities. In \cite{HPPW}, the authors use the existence of unique rectification targets in minuscule posets to give bijective proofs of the equinumerosity of various classes of plane partitions, in particular resolving a 1983 question of R.~Proctor \cite{Proctor:trap}. The main technology of \cite{HPPW} applies equally to any $d$-complete poset satisfying Conjecture~\ref{conj:URTs}; hence, we expect Theorem~\ref{thm:main} to yield analogous identities. Further discussion may appear elsewhere.

This paper is organized as follows. In Section~\ref{Section posets, skew shapes, and rectifications}, we fix notation for posets and describe the Thomas-Yong theory of jeu de taquin for increasing tableaux. We then recall the definition of unique rectification targets (URTs).
Section~\ref{Section adding to minimum and maximal elements} studies the behavior of URTs when two posets are combined via Proctor's \emph{slant sum} operation. Section ~\ref{Section slant sum} builds on Section~\ref{Section adding to minimum and maximal elements} by introducing the notion of a \emph{$p$-chain URT}, a stronger version of a URT that we will later need. Section ~\ref{Section doubled tailed diamonds} establishes the necessary technical fact that all increasing tableaux of straight shape in a double-tailed diamond poset are $p$-chain URTs. In Section~\ref{Section $d$-complete posets}, we first recall background on $d$-complete posets. We also recall needed notions to make Conjectures~\ref{conj:URTs} and~\ref{conj:geometry} precise. We then apply the results from Section~\ref{Section slant sum} and Section ~\ref{Section doubled tailed diamonds} to the study of $d$-complete posets and prove Theorem~\ref{thm:main}, our main result.

\section{posets, skew shapes, and rectifications}
\label{Section posets, skew shapes, and rectifications}

All posets in this paper will be finite, nonempty, and connected. These assumptions are made for convenience and clarity only; many of our results do not fundamentally rely on these properties, although the statements and proofs become messier without them. Moreover, the original definition of $d$-complete posets (which we follow in this paper) requires finiteness. Although there is now a more general notion of infinite $d$-complete posets (see the ``Added Notes'' at the very end of \cite{Proctor.Scoppetta} for discussion), we will not explicitly consider such objects.  In this section, $\mathcal{P}$ will denote an otherwise arbitrary poset.

We begin by fixing necessary terminology regarding posets. For $x,y \in \mathcal{P}$, we say that $z$ \textbf{covers} $x$ (written $x \lessdot z$), if $x<z$ and there does not exist a $y \in \mathcal{P}$ with $x < y < z$.
Let $x,y \in \mathcal{P}$. If $x < y$, we say that $x$ is an \textbf{ancestor} of $y$ and that $y$ is a \textbf{descendant} of $x$. If $x \lessdot y$, we say that $x$ is a \textbf{parent} of $y$ and that $y$ is a \textbf{child} of $x$. 
Adding ``\textbf{weak}'' to any of these terms also allows for equality, e.g.\ $x$ is a \textbf{weak descendant} of $y$ if $x \geq y$.
We denote the minimum element (if it exists) of the poset $\mathcal{P}$ by $\hat{0}_\mathcal{P}$. We say a poset $\mathcal{P}$ has a $\hat{0}_\mathcal{P}$ to mean that it has a minimum, which is $\hat{0}_\mathcal{P}$.

We often visualize posets using Hasse diagrams, where each element is represented by a circle, and $a \lessdot b$ if there is a line that goes up from $a$ to $b$.

\begin{example}\label{ex:Q}
Let $\mathcal{Q}$ be the poset on the elements 
$$\{(1,1),(1,2),(1,3),(2,1),(2,2),(3,1)\}$$
of $\mathbb{Z}^2$ under the natural order $(a,b) \leq (c,d)$ if both $a \leq c$ and $b \leq d$.
As a Hasse diagram, we have
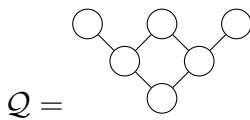
\[
\mathcal{Q} = 
\begin{tikzpicture} [every node/.append style={circle, draw=black, inner sep=0pt, minimum size=16pt}, every draw/.append style={black, thick}]
\node (11) at (0,0) {};
\node [above right of=11] (21)  {};
\node [above right of=21] (31)  {};
\node [above left of=11] (12)  {};
\node [above left of=12] (13)  {};
\node [above left of=21] (22)  {};
 
\draw (11) -- (12);
\draw (11) -- (21);
\draw (13) -- (12);
\draw (22) -- (12);
\draw (22) -- (21);
\draw (31) -- (21);
\end{tikzpicture}
\]

We will use this $\mathcal{Q}$ as a running example throughout this section.
\end{example}

A \textbf{shape} $\nu$ of $\mathcal{P}$ is any subset of $\mathcal{P}$. The shape $\nu$ has a natural poset structure given by restricting that of $\mathcal{P}$.
A shape $\nu$ of $\mathcal{P}$ is called an \textbf{order ideal} of $\mathcal{P}$ if it is closed downwards, i.e.\ if $y \in \nu$ and $x < y$ together imply $x \in \nu$. Similarly, an \textbf{order filter} of $\mathcal{P}$ is a subset that is closed upwards. For historical reasons, we will also refer to the order ideals of $\mathcal{P}$ as \textbf{straight shapes}.

\begin{example}\label{ex:straight}
The following are all the straight shapes of the poset $\mathcal{Q}$ from Example~\ref{ex:Q}. For greater visual context, we represent elements not in the straight shape with solid black circles.
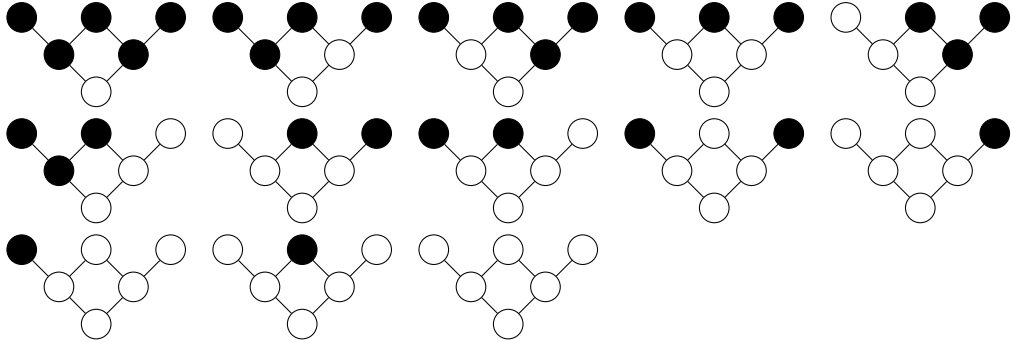
\[
\begin{array}{lllll}
\begin{tikzpicture} [every node/.append style={circle, draw=black, inner sep=0pt, minimum size=16pt}, every draw/.append style={black, thick}]
\node (11) at (0,0) {};
\node [fill=black, above right of=11] (21)  {};
\node [fill=black, above right of=21] (31)  {};
\node [fill=black, above left of=11] (12)  {};
\node [fill=black, above left of=12] (13)  {};
\node [fill=black, above left of=21] (22)  {};
 
\draw (11) -- (12);
\draw (11) -- (21);
\draw (13) -- (12);
\draw (22) -- (12);
\draw (22) -- (21);
\draw (31) -- (21);
\end{tikzpicture}&
\begin{tikzpicture} [every node/.append style={circle, draw=black, inner sep=0pt, minimum size=16pt}, every draw/.append style={black, thick}]
\node (11) at (0,0) {};
\node [ above right of=11] (21)  {};
\node [fill=black, above right of=21] (31)  {};
\node [fill=black, above left of=11] (12)  {};
\node [fill=black, above left of=12] (13)  {};
\node [fill=black, above left of=21] (22)  {};
 
\draw (11) -- (12);
\draw (11) -- (21);
\draw (13) -- (12);
\draw (22) -- (12);
\draw (22) -- (21);
\draw (31) -- (21);
\end{tikzpicture}&
\begin{tikzpicture} [every node/.append style={circle, draw=black, inner sep=0pt, minimum size=16pt}, every draw/.append style={black, thick}]
\node (11) at (0,0) {};
\node [fill=black, above right of=11] (21)  {};
\node [fill=black, above right of=21] (31)  {};
\node [ above left of=11] (12)  {};
\node [fill=black, above left of=12] (13)  {};
\node [fill=black, above left of=21] (22)  {};
 
\draw (11) -- (12);
\draw (11) -- (21);
\draw (13) -- (12);
\draw (22) -- (12);
\draw (22) -- (21);
\draw (31) -- (21);
\end{tikzpicture}&
\begin{tikzpicture} [every node/.append style={circle, draw=black, inner sep=0pt, minimum size=16pt}, every draw/.append style={black, thick}]
\node (11) at (0,0) {};
\node [ above right of=11] (21)  {};
\node [fill=black, above right of=21] (31)  {};
\node [ above left of=11] (12)  {};
\node [fill=black, above left of=12] (13)  {};
\node [fill=black, above left of=21] (22)  {};
 
\draw (11) -- (12);
\draw (11) -- (21);
\draw (13) -- (12);
\draw (22) -- (12);
\draw (22) -- (21);
\draw (31) -- (21);
\end{tikzpicture}&
\begin{tikzpicture} [every node/.append style={circle, draw=black, inner sep=0pt, minimum size=16pt}, every draw/.append style={black, thick}]
\node (11) at (0,0) {};
\node [fill=black, above right of=11] (21)  {};
\node [fill=black, above right of=21] (31)  {};
\node [ above left of=11] (12)  {};
\node [ above left of=12] (13)  {};
\node [fill=black, above left of=21] (22)  {};
 
\draw (11) -- (12);
\draw (11) -- (21);
\draw (13) -- (12);
\draw (22) -- (12);
\draw (22) -- (21);
\draw (31) -- (21);
\end{tikzpicture}
\\
\begin{tikzpicture} [every node/.append style={circle, draw=black, inner sep=0pt, minimum size=16pt}, every draw/.append style={black, thick}]
\node (11) at (0,0) {};
\node [ above right of=11] (21)  {};
\node [above right of=21] (31)  {};
\node [fill=black, above left of=11] (12)  {};
\node [fill=black, above left of=12] (13)  {};
\node [fill=black, above left of=21] (22)  {};
 
\draw (11) -- (12);
\draw (11) -- (21);
\draw (13) -- (12);
\draw (22) -- (12);
\draw (22) -- (21);
\draw (31) -- (21);
\end{tikzpicture}&
\begin{tikzpicture} [every node/.append style={circle, draw=black, inner sep=0pt, minimum size=16pt}, every draw/.append style={black, thick}]
\node (11) at (0,0) {};
\node [ above right of=11] (21)  {};
\node [fill=black, above right of=21] (31)  {};
\node [ above left of=11] (12)  {};
\node [above left of=12] (13)  {};
\node [fill=black, above left of=21] (22)  {};
 
\draw (11) -- (12);
\draw (11) -- (21);
\draw (13) -- (12);
\draw (22) -- (12);
\draw (22) -- (21);
\draw (31) -- (21);
\end{tikzpicture}&
\begin{tikzpicture} [every node/.append style={circle, draw=black, inner sep=0pt, minimum size=16pt}, every draw/.append style={black, thick}]
\node (11) at (0,0) {};
\node [ above right of=11] (21)  {};
\node [above right of=21] (31)  {};
\node [ above left of=11] (12)  {};
\node [fill=black, above left of=12] (13)  {};
\node [fill=black, above left of=21] (22)  {};
 
\draw (11) -- (12);
\draw (11) -- (21);
\draw (13) -- (12);
\draw (22) -- (12);
\draw (22) -- (21);
\draw (31) -- (21);
\end{tikzpicture}&
\begin{tikzpicture} [every node/.append style={circle, draw=black, inner sep=0pt, minimum size=16pt}, every draw/.append style={black, thick}]
\node (11) at (0,0) {};
\node [ above right of=11] (21)  {};
\node [fill=black, above right of=21] (31)  {};
\node [ above left of=11] (12)  {};
\node [fill=black, above left of=12] (13)  {};
\node [above left of=21] (22)  {};
 
\draw (11) -- (12);
\draw (11) -- (21);
\draw (13) -- (12);
\draw (22) -- (12);
\draw (22) -- (21);
\draw (31) -- (21);
\end{tikzpicture}&
\begin{tikzpicture} [every node/.append style={circle, draw=black, inner sep=0pt, minimum size=16pt}, every draw/.append style={black, thick}]
\node (11) at (0,0) {};
\node [ above right of=11] (21)  {};
\node [fill=black, above right of=21] (31)  {};
\node [ above left of=11] (12)  {};
\node [above left of=12] (13)  {};
\node [above left of=21] (22)  {};
 
\draw (11) -- (12);
\draw (11) -- (21);
\draw (13) -- (12);
\draw (22) -- (12);
\draw (22) -- (21);
\draw (31) -- (21);
\end{tikzpicture}
\\
\begin{tikzpicture} [every node/.append style={circle, draw=black, inner sep=0pt, minimum size=16pt}, every draw/.append style={black, thick}]
\node (11) at (0,0) {};
\node [ above right of=11] (21)  {};
\node [above right of=21] (31)  {};
\node [ above left of=11] (12)  {};
\node [fill=black, above left of=12] (13)  {};
\node [above left of=21] (22)  {};
 
\draw (11) -- (12);
\draw (11) -- (21);
\draw (13) -- (12);
\draw (22) -- (12);
\draw (22) -- (21);
\draw (31) -- (21);
\end{tikzpicture}&
\begin{tikzpicture} [every node/.append style={circle, draw=black, inner sep=0pt, minimum size=16pt}, every draw/.append style={black, thick}]
\node (11) at (0,0) {};
\node [ above right of=11] (21)  {};
\node [above right of=21] (31)  {};
\node [ above left of=11] (12)  {};
\node [ above left of=12] (13)  {};
\node [fill=black,above left of=21] (22)  {};
 
\draw (11) -- (12);
\draw (11) -- (21);
\draw (13) -- (12);
\draw (22) -- (12);
\draw (22) -- (21);
\draw (31) -- (21);
\end{tikzpicture}&
\begin{tikzpicture} [every node/.append style={circle, draw=black, inner sep=0pt, minimum size=16pt}, every draw/.append style={black, thick}]
\node (11) at (0,0) {};
\node [ above right of=11] (21)  {};
\node [above right of=21] (31)  {};
\node [ above left of=11] (12)  {};
\node [ above left of=12] (13)  {};
\node [above left of=21] (22)  {};
 
\draw (11) -- (12);
\draw (11) -- (21);
\draw (13) -- (12);
\draw (22) -- (12);
\draw (22) -- (21);
\draw (31) -- (21);
\end{tikzpicture}&
\end{array}
\]
\end{example}

If $\lambda \subseteq \nu$ are straight shapes of $\mathcal{P}$, then the shape $\nu \setminus \lambda$ is called a \textbf{skew shape} of $\mathcal{P}$ and is denoted $\nu/\lambda$. Note that every straight shape $\lambda$ can also be realized as the skew shape $\lambda / \emptyset$.

An element $x \in \lambda$ is called an \textbf{inner corner} of the skew shape $\nu / \lambda$ if $x$ is maximal in $\lambda$. We write $\ic(\nu / \lambda)$ to denote the set of inner corners of $\nu / \lambda$.

Clearly, we have the following.
\begin{lemma}
\label{Corollary inner corner ancestors}
Let $\nu/\lambda$ be a skew shape of the poset $\mathcal{P}$. If $c \in \ic(\nu / \lambda)$ and $b < c$, then $b \notin \ic(\nu / \lambda)$. \qed 
\end{lemma}

   For a skew shape $\nu/\lambda$ of $\mathcal{P}$, a function $T: \nu/\lambda \to \mathbb{Z}_{>0}$ is called a \textbf{skew increasing $\mathcal{P}$-tableau} of shape $\nu/\lambda$ if $T$ is a strictly order preserving map, i.e.\ if $x < y$ implies $T(x) < T(y)$.
If, in addition, $T$ is a bijection onto an initial segment of $\mathbb{Z}_{>0}$, we say $T$ is a \textbf{skew standard $\mathcal{P}$-tableau}. In both cases, if $\nu/\lambda$ is a straight shape, we drop the word ``skew.''

We depict a skew increasing $\mathcal{P}$-tableau $T$ using Hasse diagrams with labels. For an element $x \in \mathcal{P}$, we put the value of $T(x)$ in the circle of the Hasse diagram corresponding to $x$. Also, to make clear what the ambient poset is we represent skewed out elements (the elements in $\lambda$) with unlabeled hollow circles.
\begin{example}
If $\mathcal{P}$ is the numbers $1,2,3,4$ with the usual order and $\nu/\lambda = \mathcal{P}$ and $T(x) = x+5$, then $T$ can be visualized as
\[
\begin{tikzpicture} [every node/.append style={circle, draw=black, inner sep=0pt, minimum size=16pt}, every draw/.append style={black, thick}]
\node (1) at (0,0) {6};
\node [ above of=1] (2)  {7};
\node [ above of=2] (3)  {8};
\node [ above of=3] (4)  {9};
 
\draw (1) -- (2);
\draw (2) -- (3);
\draw (3) -- (4);
\end{tikzpicture}
\]
\end{example}
\begin{example}
The following is an example of a skew increasing $\mathcal{Q}$-tableau  $T$ of shape $\nu/\lambda$
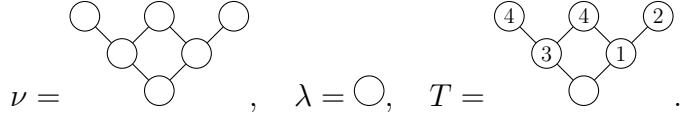
\[\nu = 
\begin{tikzpicture} [every node/.append style={circle, draw=black, inner sep=0pt, minimum size=16pt}, every draw/.append style={black, thick}]
\node (11) at (0,0) {};
\node [above right of=11] (21)  {};
\node [above right of=21] (31)  {};
\node [above left of=11] (12)  {};
\node [above left of=12] (13)  {};
\node [above left of=21] (22)  {};
 
\draw (11) -- (12);
\draw (11) -- (21);
\draw (13) -- (12);
\draw (22) -- (12);
\draw (22) -- (21);
\draw (31) -- (21);
\end{tikzpicture}, \quad 
\lambda = \begin{tikzpicture} [every node/.append style={circle, draw=black, inner sep=0pt, minimum size=16pt}, every draw/.append style={black, thick}]
\node (11) at (0,0) {};
\end{tikzpicture}
, 
\quad T = \begin{tikzpicture} [every node/.append style={circle, draw=black, inner sep=0pt, minimum size=16pt}, every draw/.append style={black, thick}]
\node  (11) at (0,0) {};
\node [above right of=11] (21)  {1};
\node [above right of=21] (31)  {2};
\node [above left of=11] (12)  {3};
\node [above left of=12] (13)  {4};
\node [above left of=21] (22)  {4};
 
\draw (11) -- (12);
\draw (11) -- (21);
\draw (13) -- (12);
\draw (22) -- (12);
\draw (22) -- (21);
\draw (31) -- (21);
\end{tikzpicture}.\]
\end{example}

For a skew shape $\nu/\lambda$ of $\mathcal{P}$, we say a function $T: \nu/\lambda \to \mathbb{Z}_{>0} \cup \{\bullet\}$ is a  \textbf{skew dotted increasing $\mathcal{P}$-tableau} of shape $\nu/\lambda$  if there is a rational number $q$ such that $T$ becomes a strictly order preserving map ($\nu/\lambda \to \mathbb{Q}$) when we replace each $\bullet$ with that fixed $q$.
\begin{example}
The following is a skew dotted increasing $\mathcal{Q}$-tableau
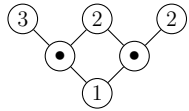
\[
\begin{tikzpicture} [every node/.append style={circle, draw=black, inner sep=0pt, minimum size=16pt}, every draw/.append style={black, thick}]
\node  (11) at (0,0) {1};
\node [above right of=11] (21)  {$\bullet$};
\node [above right of=21] (31)  {2};
\node [above left of=11] (12)  {$\bullet$};
\node [above left of=12] (13)  {3};
\node [above left of=21] (22)  {2};
 
\draw (11) -- (12);
\draw (11) -- (21);
\draw (13) -- (12);
\draw (22) -- (12);
\draw (22) -- (21);
\draw (31) -- (21);
\end{tikzpicture}
\]
and the following is not (because one cannot replace the $\bullet$ with one fixed $q$ and make it an order preserving map)
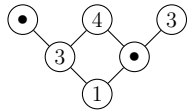
\[
\begin{tikzpicture} [every node/.append style={circle, draw=black, inner sep=0pt, minimum size=16pt}, every draw/.append style={black, thick}]
\node  (11) at (0,0) {1};
\node [above right of=11] (21)  {$\bullet$};
\node [above right of=21] (31)  {3};
\node [above left of=11] (12)  {3};
\node [above left of=12] (13)  {$\bullet$};
\node [above left of=21] (22)  {4};
 
\draw (11) -- (12);
\draw (11) -- (21);
\draw (13) -- (12);
\draw (22) -- (12);
\draw (22) -- (21);
\draw (31) -- (21);
\end{tikzpicture}
\]
\end{example}

\begin{definition}
Let $T$ be a skew increasing $\mathcal{P}$-tableau of shape $\nu/\lambda$. If $\gamma$ is a nonempty set of inner (or outer) corners of $\nu/\lambda$, then $\mathsf{AddDots}_\gamma(T)$ is the skew increasing $\mathcal{P}$-tableau $S$ of shape $\nu/\lambda \cup \gamma$ defined by
\[S(x) = 
			\begin{cases} 
              T(x),  & \text{ if } x \in \nu/\lambda; \\
              \bullet, & \text{ if } x \in \gamma.
          \end{cases}
\]
\end{definition}
\begin{example}
For 
\[T = 
\begin{tikzpicture} [every node/.append style={circle, draw=black, inner sep=0pt, minimum size=16pt}, every draw/.append style={black, thick}]
\node  (11) at (0,0) {};
\node [fill=SkyBlue,above right of=11] (21)  {};
\node [above right of=21] (31)  {3};
\node [fill=SkyBlue,above left of=11] (12)  {};
\node [above left of=12] (13)  {1};
\node [above left of=21] (22)  {2};
 
\draw (11) -- (12);
\draw (11) -- (21);
\draw (13) -- (12);
\draw (22) -- (12);
\draw (22) -- (21);
\draw (31) -- (21);
\end{tikzpicture}
\]
let $\gamma$ be the set of blue shaded inner corners. Then, we have
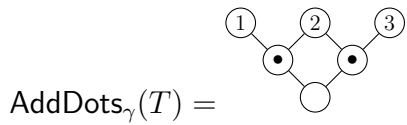
\pushQED{\qed}
\[ \mathsf{AddDots}_\gamma(T) =
\begin{tikzpicture} [every node/.append style={circle, draw=black, inner sep=0pt, minimum size=16pt}, every draw/.append style={black, thick}]
\node  (11) at (0,0) {};
\node [above right of=11] (21)  {$\bullet$};
\node [above right of=21] (31)  {3};
\node [above left of=11] (12)  {$\bullet$};
\node [above left of=12] (13)  {1};
\node [above left of=21] (22)  {2};
 
\draw (11) -- (12);
\draw (11) -- (21);
\draw (13) -- (12);
\draw (22) -- (12);
\draw (22) -- (21);
\draw (31) -- (21);
\end{tikzpicture} \qedhere \popQED \] 
\let\qed\relax
\end{example}

\begin{definition}
Let $T$  be a skew dotted increasing $\mathcal{P}$-tableau. For $n \in \mathbb{Z}_{>0}$, $\mathsf{Swap}_{\bullet,n}(T)$ is the skew dotted increasing $\mathcal{P}$-tableau $S$ defined by
\[S(x) = 
     	\begin{cases} 
            n, & \text{ if } T(x) = \bullet \text{ and } T(y) = n \text{ for some $y \lessdot x$}; \\
            \bullet, & \text{ if } T(x) = n \text{ and } T(y) = \bullet \text{ for some $y \gtrdot x$}; \\
            T(x), & \text{ otherwise.} 
        \end{cases}
\]
\end{definition}
\begin{example}
We have
\ytableausetup{centertableaux}
\begin{align*}
&\mathsf{Swap}_{\bullet,1}\left( \;
\begin{tikzpicture} [every node/.append style={circle, draw=black, inner sep=0pt, minimum size=16pt}, every draw/.append style={black, thick}, baseline={([yshift=-.5ex]current bounding box.center)}]
\node  (11) at (0,0) {};
\node [above right of=11] (21)  {$\bullet$};
\node [above right of=21] (31)  {3};
\node [above left of=11] (12)  {$\bullet$};
\node [above left of=12] (13)  {1};
\node [above left of=21] (22)  {2};
\draw (11) -- (12);
\draw (11) -- (21);
\draw (13) -- (12);
\draw (22) -- (12);
\draw (22) -- (21);
\draw (31) -- (21);
\end{tikzpicture} \; \right) = 
\begin{tikzpicture} [every node/.append style={circle, draw=black, inner sep=0pt, minimum size=16pt}, every draw/.append style={black, thick}, baseline={([yshift=-.5ex]current bounding box.center)}]
\node  (11) at (0,0) {};
\node [above right of=11] (21)  {$\bullet$};
\node [above right of=21] (31)  {3};
\node [above left of=11] (12)  {1};
\node [above left of=12] (13)  {$\bullet$};
\node [above left of=21] (22)  {2};
\draw (11) -- (12);
\draw (11) -- (21);
\draw (13) -- (12);
\draw (22) -- (12);
\draw (22) -- (21);
\draw (31) -- (21);
\end{tikzpicture}
&\mathsf{Swap}_{\bullet,2}\left( \;
\begin{tikzpicture} [every node/.append style={circle, draw=black, inner sep=0pt, minimum size=16pt}, every draw/.append style={black, thick}, baseline={([yshift=-.5ex]current bounding box.center)}]
\node  (11) at (0,0) {};
\node [above right of=11] (21)  {1};
\node [above right of=21] (31)  {2};
\node [above left of=11] (12)  {$\bullet$};
\node [above left of=12] (13)  {2};
\node [above left of=21] (22)  {2};
\draw (11) -- (12);
\draw (11) -- (21);
\draw (13) -- (12);
\draw (22) -- (12);
\draw (22) -- (21);
\draw (31) -- (21);
\end{tikzpicture}
\; \right) =
\begin{tikzpicture} [every node/.append style={circle, draw=black, inner sep=0pt, minimum size=16pt}, every draw/.append style={black, thick}, baseline={([yshift=-.5ex]current bounding box.center)}]
\node  (11) at (0,0) {};
\node [above right of=11] (21)  {1};
\node [above right of=21] (31)  {2};
\node [above left of=11] (12)  {2};
\node [above left of=12] (13)  {$\bullet$};
\node [above left of=21] (22)  {$\bullet$};
\draw (11) -- (12);
\draw (11) -- (21);
\draw (13) -- (12);
\draw (22) -- (12);
\draw (22) -- (21);
\draw (31) -- (21);
\end{tikzpicture}\\
&\mathsf{Swap}_{\bullet,1}\left( \; 
\begin{tikzpicture} [every node/.append style={circle, draw=black, inner sep=0pt, minimum size=16pt}, every draw/.append style={black, thick}, baseline={([yshift=-.5ex]current bounding box.center)}]
\node  (11) at (0,0) {};
\node [above right of=11] (21)  {$\bullet$};
\node [above right of=21] (31)  {2};
\node [above left of=11] (12)  {$\bullet$};
\node [above left of=12] (13)  {1};
\node [above left of=21] (22)  {1};
\draw (11) -- (12);
\draw (11) -- (21);
\draw (13) -- (12);
\draw (22) -- (12);
\draw (22) -- (21);
\draw (31) -- (21);
\end{tikzpicture}
\right) = 
\begin{tikzpicture} [every node/.append style={circle, draw=black, inner sep=0pt, minimum size=16pt}, every draw/.append style={black, thick}, baseline={([yshift=-.5ex]current bounding box.center)}]
\node  (11) at (0,0) {};
\node [above right of=11] (21)  {1};
\node [above right of=21] (31)  {2};
\node [above left of=11] (12)  {1};
\node [above left of=12] (13)  {$\bullet$};
\node [above left of=21] (22)  {$\bullet$};
\draw (11) -- (12);
\draw (11) -- (21);
\draw (13) -- (12);
\draw (22) -- (12);
\draw (22) -- (21);
\draw (31) -- (21);
\end{tikzpicture}
&\mathsf{Swap}_{\bullet,1}\left( \; 
\begin{tikzpicture} [every node/.append style={circle, draw=black, inner sep=0pt, minimum size=16pt}, every draw/.append style={black, thick}, baseline={([yshift=-.5ex]current bounding box.center)}]
\node  (11) at (0,0) {};
\node [above right of=11] (21)  {$\bullet$};
\node [above right of=21] (31)  {1};
\node [above left of=11] (12)  {$\bullet$};
\node [above left of=12] (13)  {1};
\node [above left of=21] (22)  {1};
\draw (11) -- (12);
\draw (11) -- (21);
\draw (13) -- (12);
\draw (22) -- (12);
\draw (22) -- (21);
\draw (31) -- (21);
\end{tikzpicture}
\; \right) = 
\begin{tikzpicture} [every node/.append style={circle, draw=black, inner sep=0pt, minimum size=16pt}, every draw/.append style={black, thick}, baseline={([yshift=-.5ex]current bounding box.center)}]
\node  (11) at (0,0) {};
\node [above right of=11] (21)  {1};
\node [above right of=21] (31)  {$\bullet$};
\node [above left of=11] (12)  {1};
\node [above left of=12] (13)  {$\bullet$};
\node [above left of=21] (22)  {$\bullet$};
\draw (11) -- (12);
\draw (11) -- (21);
\draw (13) -- (12);
\draw (22) -- (12);
\draw (22) -- (21);
\draw (31) -- (21);
\end{tikzpicture}\\
&\mathsf{Swap}_{\bullet,1}\left( \;
\begin{tikzpicture} [every node/.append style={circle, draw=black, inner sep=0pt, minimum size=16pt}, every draw/.append style={black, thick}, baseline={([yshift=-.5ex]current bounding box.center)}]
\node  (11) at (0,0) {};
\node [above right of=11] (21)  {$\bullet$};
\node [above right of=21] (31)  {3};
\node [above left of=11] (12)  {1};
\node [above left of=12] (13)  {2};
\node [above left of=21] (22)  {3};
\draw (11) -- (12);
\draw (11) -- (21);
\draw (13) -- (12);
\draw (22) -- (12);
\draw (22) -- (21);
\draw (31) -- (21);
\end{tikzpicture}
\; \right) = 
\begin{tikzpicture} [every node/.append style={circle, draw=black, inner sep=0pt, minimum size=16pt}, every draw/.append style={black, thick}, baseline={([yshift=-.5ex]current bounding box.center)}]
\node  (11) at (0,0) {};
\node [above right of=11] (21)  {$\bullet$};
\node [above right of=21] (31)  {3};
\node [above left of=11] (12)  {1};
\node [above left of=12] (13)  {2};
\node [above left of=21] (22)  {3};
\draw (11) -- (12);
\draw (11) -- (21);
\draw (13) -- (12);
\draw (22) -- (12);
\draw (22) -- (21);
\draw (31) -- (21);
\end{tikzpicture}& \\
\end{align*}
\ytableausetup{nocentertableaux}
\end{example}

\begin{definition}
Let $T$  be a skew dotted increasing $\mathcal{P}$-tableau. Let $\mathcal{Q}$ be the subset of $\mathcal{P}$ which $T$ maps to an integer, i.e.\
\[\mathcal{Q} = \{x : T(x) \in \mathbb{Z}_{>0}  \}.\]
Then, we define $\mathsf{RemoveDots}(T) = T|_\mathcal{Q}$.
\end{definition}
\begin{example}
For example,
\[\mathsf{RemoveDots} \left( \; 
\begin{tikzpicture} [every node/.append style={circle, draw=black, inner sep=0pt, minimum size=16pt}, every draw/.append style={black, thick}, baseline={([yshift=-.5ex]current bounding box.center)}]
\node  (11) at (0,0) {1};
\node [above right of=11] (21)  {2};
\node [above right of=21] (31)  {$\bullet$};
\node [above left of=11] (12)  {2};
\node [above left of=12] (13)  {$\bullet$};
\node [above left of=21] (22)  {$\bullet$};
\draw (11) -- (12);
\draw (11) -- (21);
\draw (13) -- (12);
\draw (22) -- (12);
\draw (22) -- (21);
\draw (31) -- (21);
\end{tikzpicture}
\; \right) = \begin{tikzpicture} [every node/.append style={circle, draw=black, inner sep=0pt, minimum size=16pt}, every draw/.append style={black, thick}, baseline={([yshift=-.5ex]current bounding box.center)}]
\node  (11) at (0,0) {1};
\node [above right of=11] (21)  {2};
\node [above left of=11] (12)  {2};
\draw (11) -- (12);
\draw (11) -- (21);
\end{tikzpicture}.\]
\end{example}

\begin{definition}
Let $T$ be a skew increasing $\mathcal{P}$-tableau of shape $\nu/\lambda$ and let $\gamma \subseteq \ic(\nu/\lambda)$. Let $n = \max(\range(T))$. The \textbf{slide} of $\gamma$ in $T$ is the skew increasing $\mathcal{P}$-tableau
\[\mathsf{Slide}_\gamma(T) = \mathsf{RemoveDots} \circ \mathsf{Swap}_{\bullet,n} \circ \cdots \circ \mathsf{Swap}_{\bullet,1} \circ \mathsf{AddDots}_\gamma(T).\]
We will also use the notation $\mathsf{Slide}_{\gamma_1, \ldots , \gamma_n}$ to denote iterated slides, i.e.\
\[\mathsf{Slide}_{\gamma_1, \ldots, \gamma_n}(T) = \mathsf{Slide}_{\gamma_n} \circ \cdots \circ \mathsf{Slide}_{\gamma_1}(T).\]
\end{definition}

\begin{example}
As an example, let $T$ be an increasing $\mathcal{P}$ tableau
\[T = \begin{tikzpicture} [every node/.append style={circle, draw=black, inner sep=0pt, minimum size=16pt}, every draw/.append style={black, thick}]
\node  (11) at (0,0) {};
\node [above left of=11] (12)  {};
\node [above left of=12] (13)  {1};
\node [above left of=13] (14)  {2};
\node [above right of=11] (21)  {};
\node [above left of=21] (22)  {2};
\node [above left of=22] (23)  {3};
\node [above right of=21] (31)  {2};
\node [above left of=31] (32)  {3};
\node [above right of=31] (41)  {4};

\draw (11) -- (12);
\draw (12) -- (13);
\draw (13) -- (14);
\draw (21) -- (22);
\draw (22) -- (23);
\draw (31) -- (32);
\draw (11) -- (21);
\draw (21) -- (31);
\draw (31) -- (41);
\draw (12) -- (22);
\draw (22) -- (32);
\draw (13) -- (23);

\end{tikzpicture}, \;
\mathcal{P} = \begin{tikzpicture} [every node/.append style={circle, draw=black, inner sep=0pt, minimum size=16pt}, every draw/.append style={black, thick}]
\node  (11) at (0,0) {};
\node [above left of=11] (12)  {a};
\node [above left of=12] (13)  {};
\node [above left of=13] (14)  {};
\node [above right of=11] (21)  {b};
\node [above left of=21] (22)  {};
\node [above left of=22] (23)  {};
\node [above right of=21] (31)  {};
\node [above left of=31] (32)  {};
\node [above right of=31] (41)  {};

\draw (11) -- (12);
\draw (12) -- (13);
\draw (13) -- (14);
\draw (21) -- (22);
\draw (22) -- (23);
\draw (31) -- (32);
\draw (11) -- (21);
\draw (21) -- (31);
\draw (31) -- (41);
\draw (12) -- (22);
\draw (22) -- (32);
\draw (13) -- (23);

\end{tikzpicture}\]
where $a,b$ are the two inner corners for $T$.
We compute $\mathsf{Slide_\gamma}$ for various values of $\gamma$, beginning after $\mathsf{addDots}$, showing the intermediate swaps, and ending after $\mathsf{removeDots}$:
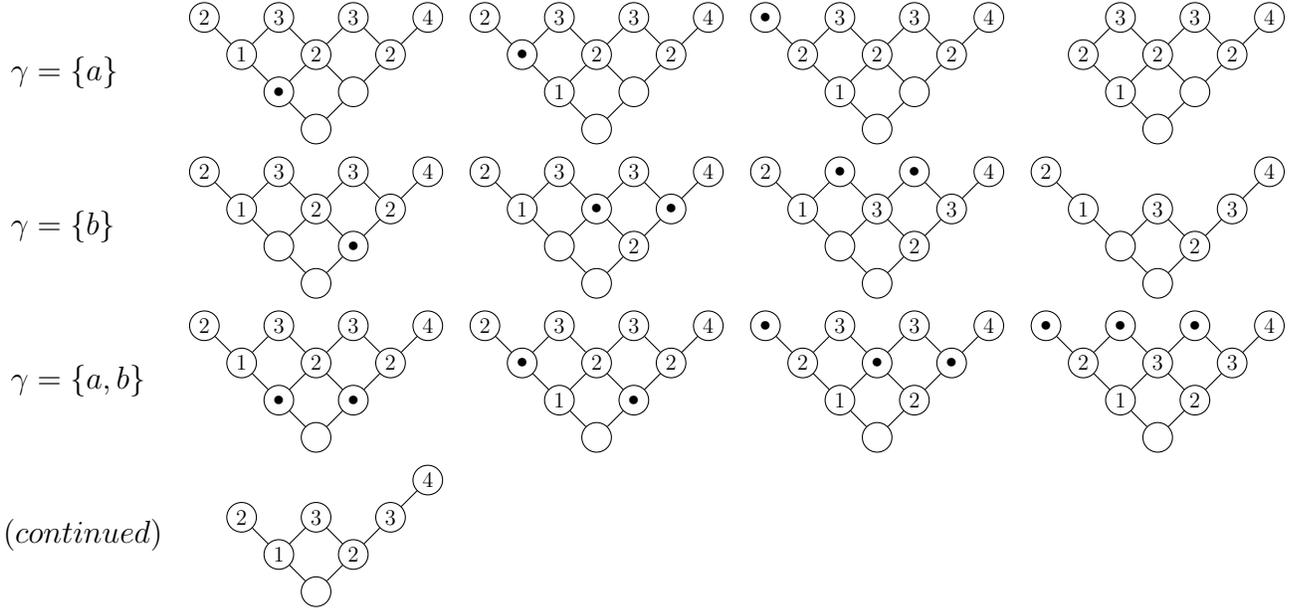
\begin{align*}
\gamma &= \{a\} &
\begin{tikzpicture} [every node/.append style={circle, draw=black, inner sep=0pt, minimum size=16pt}, every draw/.append style={black, thick}, baseline={([yshift=-.5ex]current bounding box.center)}]
\node  (11) at (0,0) {};
\node [above left of=11] (12)  {$\bullet$};
\node [above left of=12] (13)  {1};
\node [above left of=13] (14)  {2};
\node [above right of=11] (21)  {};
\node [above left of=21] (22)  {2};
\node [above left of=22] (23)  {3};
\node [above right of=21] (31)  {2};
\node [above left of=31] (32)  {3};
\node [above right of=31] (41)  {4};
\draw (11) -- (12);
\draw (12) -- (13);
\draw (13) -- (14);
\draw (21) -- (22);
\draw (22) -- (23);
\draw (31) -- (32);
\draw (11) -- (21);
\draw (21) -- (31);
\draw (31) -- (41);
\draw (12) -- (22);
\draw (22) -- (32);
\draw (13) -- (23);
\end{tikzpicture} && 
\begin{tikzpicture} [every node/.append style={circle, draw=black, inner sep=0pt, minimum size=16pt}, every draw/.append style={black, thick}, baseline={([yshift=-.5ex]current bounding box.center)}]
\node  (11) at (0,0) {};
\node [above left of=11] (12)  {1};
\node [above left of=12] (13)  {$\bullet$};
\node [above left of=13] (14)  {2};
\node [above right of=11] (21)  {};
\node [above left of=21] (22)  {2};
\node [above left of=22] (23)  {3};
\node [above right of=21] (31)  {2};
\node [above left of=31] (32)  {3};
\node [above right of=31] (41)  {4};
\draw (11) -- (12);
\draw (12) -- (13);
\draw (13) -- (14);
\draw (21) -- (22);
\draw (22) -- (23);
\draw (31) -- (32);
\draw (11) -- (21);
\draw (21) -- (31);
\draw (31) -- (41);
\draw (12) -- (22);
\draw (22) -- (32);
\draw (13) -- (23);
\end{tikzpicture}&& 
\begin{tikzpicture} [every node/.append style={circle, draw=black, inner sep=0pt, minimum size=16pt}, every draw/.append style={black, thick}, baseline={([yshift=-.5ex]current bounding box.center)}]
\node  (11) at (0,0) {};
\node [above left of=11] (12)  {1};
\node [above left of=12] (13)  {2};
\node [above left of=13] (14)  {$\bullet$};
\node [above right of=11] (21)  {};
\node [above left of=21] (22)  {2};
\node [above left of=22] (23)  {3};
\node [above right of=21] (31)  {2};
\node [above left of=31] (32)  {3};
\node [above right of=31] (41)  {4};
\draw (11) -- (12);
\draw (12) -- (13);
\draw (13) -- (14);
\draw (21) -- (22);
\draw (22) -- (23);
\draw (31) -- (32);
\draw (11) -- (21);
\draw (21) -- (31);
\draw (31) -- (41);
\draw (12) -- (22);
\draw (22) -- (32);
\draw (13) -- (23);
\end{tikzpicture}&&
\begin{tikzpicture} [every node/.append style={circle, draw=black, inner sep=0pt, minimum size=16pt}, every draw/.append style={black, thick}, baseline={([yshift=-.5ex]current bounding box.center)}]
\node  (11) at (0,0) {};
\node [above left of=11] (12)  {1};
\node [above left of=12] (13)  {2};
\node [above right of=11] (21)  {};
\node [above left of=21] (22)  {2};
\node [above left of=22] (23)  {3};
\node [above right of=21] (31)  {2};
\node [above left of=31] (32)  {3};
\node [above right of=31] (41)  {4};
\draw (11) -- (12);
\draw (12) -- (13);
\draw (21) -- (22);
\draw (22) -- (23);
\draw (31) -- (32);
\draw (11) -- (21);
\draw (21) -- (31);
\draw (31) -- (41);
\draw (12) -- (22);
\draw (22) -- (32);
\draw (13) -- (23);
\end{tikzpicture} \\
\gamma &= \{b\} &
\begin{tikzpicture} [every node/.append style={circle, draw=black, inner sep=0pt, minimum size=16pt}, every draw/.append style={black, thick}, baseline={([yshift=-.5ex]current bounding box.center)}]
\node  (11) at (0,0) {};
\node [above left of=11] (12)  {};
\node [above left of=12] (13)  {1};
\node [above left of=13] (14)  {2};
\node [above right of=11] (21)  {$\bullet$};
\node [above left of=21] (22)  {2};
\node [above left of=22] (23)  {3};
\node [above right of=21] (31)  {2};
\node [above left of=31] (32)  {3};
\node [above right of=31] (41)  {4};
\draw (11) -- (12);
\draw (12) -- (13);
\draw (13) -- (14);
\draw (21) -- (22);
\draw (22) -- (23);
\draw (31) -- (32);
\draw (11) -- (21);
\draw (21) -- (31);
\draw (31) -- (41);
\draw (12) -- (22);
\draw (22) -- (32);
\draw (13) -- (23);
\end{tikzpicture}&&
\begin{tikzpicture} [every node/.append style={circle, draw=black, inner sep=0pt, minimum size=16pt}, every draw/.append style={black, thick}, baseline={([yshift=-.5ex]current bounding box.center)}]
\node  (11) at (0,0) {};
\node [above left of=11] (12)  {};
\node [above left of=12] (13)  {1};
\node [above left of=13] (14)  {2};
\node [above right of=11] (21)  {2};
\node [above left of=21] (22)  {$\bullet$};
\node [above left of=22] (23)  {3};
\node [above right of=21] (31)  {$\bullet$};
\node [above left of=31] (32)  {3};
\node [above right of=31] (41)  {4};
\draw (11) -- (12);
\draw (12) -- (13);
\draw (13) -- (14);
\draw (21) -- (22);
\draw (22) -- (23);
\draw (31) -- (32);
\draw (11) -- (21);
\draw (21) -- (31);
\draw (31) -- (41);
\draw (12) -- (22);
\draw (22) -- (32);
\draw (13) -- (23);
\end{tikzpicture}&&
\begin{tikzpicture} [every node/.append style={circle, draw=black, inner sep=0pt, minimum size=16pt}, every draw/.append style={black, thick}, baseline={([yshift=-.5ex]current bounding box.center)}]
\node  (11) at (0,0) {};
\node [above left of=11] (12)  {};
\node [above left of=12] (13)  {1};
\node [above left of=13] (14)  {2};
\node [above right of=11] (21)  {2};
\node [above left of=21] (22)  {3};
\node [above left of=22] (23)  {$\bullet$};
\node [above right of=21] (31)  {3};
\node [above left of=31] (32)  {$\bullet$};
\node [above right of=31] (41)  {4};
\draw (11) -- (12);
\draw (12) -- (13);
\draw (13) -- (14);
\draw (21) -- (22);
\draw (22) -- (23);
\draw (31) -- (32);
\draw (11) -- (21);
\draw (21) -- (31);
\draw (31) -- (41);
\draw (12) -- (22);
\draw (22) -- (32);
\draw (13) -- (23);
\end{tikzpicture}&&
\begin{tikzpicture} [every node/.append style={circle, draw=black, inner sep=0pt, minimum size=16pt}, every draw/.append style={black, thick}, baseline={([yshift=-.5ex]current bounding box.center)}]
\node  (11) at (0,0) {};
\node [above left of=11] (12)  {};
\node [above left of=12] (13)  {1};
\node [above left of=13] (14)  {2};
\node [above right of=11] (21)  {2};
\node [above left of=21] (22)  {3};
\node [above right of=21] (31)  {3};
\node [above right of=31] (41)  {4};
\draw (11) -- (12);
\draw (12) -- (13);
\draw (13) -- (14);
\draw (21) -- (22);
\draw (11) -- (21);
\draw (21) -- (31);
\draw (31) -- (41);
\draw (12) -- (22);
\end{tikzpicture}\\
\gamma &= \{a,b\} &
\begin{tikzpicture} [every node/.append style={circle, draw=black, inner sep=0pt, minimum size=16pt}, every draw/.append style={black, thick}, baseline={([yshift=-.5ex]current bounding box.center)}]
\node  (11) at (0,0) {};
\node [above left of=11] (12)  {$\bullet$};
\node [above left of=12] (13)  {1};
\node [above left of=13] (14)  {2};
\node [above right of=11] (21)  {$\bullet$};
\node [above left of=21] (22)  {2};
\node [above left of=22] (23)  {3};
\node [above right of=21] (31)  {2};
\node [above left of=31] (32)  {3};
\node [above right of=31] (41)  {4};
\draw (11) -- (12);
\draw (12) -- (13);
\draw (13) -- (14);
\draw (21) -- (22);
\draw (22) -- (23);
\draw (31) -- (32);
\draw (11) -- (21);
\draw (21) -- (31);
\draw (31) -- (41);
\draw (12) -- (22);
\draw (22) -- (32);
\draw (13) -- (23);
\end{tikzpicture}&&
\begin{tikzpicture} [every node/.append style={circle, draw=black, inner sep=0pt, minimum size=16pt}, every draw/.append style={black, thick}, baseline={([yshift=-.5ex]current bounding box.center)}]
\node  (11) at (0,0) {};
\node [above left of=11] (12)  {1};
\node [above left of=12] (13)  {$\bullet$};
\node [above left of=13] (14)  {2};
\node [above right of=11] (21)  {$\bullet$};
\node [above left of=21] (22)  {2};
\node [above left of=22] (23)  {3};
\node [above right of=21] (31)  {2};
\node [above left of=31] (32)  {3};
\node [above right of=31] (41)  {4};
\draw (11) -- (12);
\draw (12) -- (13);
\draw (13) -- (14);
\draw (21) -- (22);
\draw (22) -- (23);
\draw (31) -- (32);
\draw (11) -- (21);
\draw (21) -- (31);
\draw (31) -- (41);
\draw (12) -- (22);
\draw (22) -- (32);
\draw (13) -- (23);
\end{tikzpicture}&&
\begin{tikzpicture} [every node/.append style={circle, draw=black, inner sep=0pt, minimum size=16pt}, every draw/.append style={black, thick}, baseline={([yshift=-.5ex]current bounding box.center)}]
\node  (11) at (0,0) {};
\node [above left of=11] (12)  {1};
\node [above left of=12] (13)  {2};
\node [above left of=13] (14)  {$\bullet$};
\node [above right of=11] (21)  {2};
\node [above left of=21] (22)  {$\bullet$};
\node [above left of=22] (23)  {3};
\node [above right of=21] (31)  {$\bullet$};
\node [above left of=31] (32)  {3};
\node [above right of=31] (41)  {4};
\draw (11) -- (12);
\draw (12) -- (13);
\draw (13) -- (14);
\draw (21) -- (22);
\draw (22) -- (23);
\draw (31) -- (32);
\draw (11) -- (21);
\draw (21) -- (31);
\draw (31) -- (41);
\draw (12) -- (22);
\draw (22) -- (32);
\draw (13) -- (23);
\end{tikzpicture}&&
\begin{tikzpicture} [every node/.append style={circle, draw=black, inner sep=0pt, minimum size=16pt}, every draw/.append style={black, thick}, baseline={([yshift=-.5ex]current bounding box.center)}]
\node  (11) at (0,0) {};
\node [above left of=11] (12)  {1};
\node [above left of=12] (13)  {2};
\node [above left of=13] (14)  {$\bullet$};
\node [above right of=11] (21)  {2};
\node [above left of=21] (22)  {3};
\node [above left of=22] (23)  {$\bullet$};
\node [above right of=21] (31)  {3};
\node [above left of=31] (32)  {$\bullet$};
\node [above right of=31] (41)  {4};
\draw (11) -- (12);
\draw (12) -- (13);
\draw (13) -- (14);
\draw (21) -- (22);
\draw (22) -- (23);
\draw (31) -- (32);
\draw (11) -- (21);
\draw (21) -- (31);
\draw (31) -- (41);
\draw (12) -- (22);
\draw (22) -- (32);
\draw (13) -- (23);
\end{tikzpicture}\\
(c&ontinued) &
\begin{tikzpicture} [every node/.append style={circle, draw=black, inner sep=0pt, minimum size=16pt}, every draw/.append style={black, thick}, baseline={([yshift=-.5ex]current bounding box.center)}]
\node  (11) at (0,0) {};
\node [above left of=11] (12)  {1};
\node [above left of=12] (13)  {2};
\node [above right of=11] (21)  {2};
\node [above left of=21] (22)  {3};
\node [above right of=21] (31)  {3};
\node [above right of=31] (41)  {4};
\draw (11) -- (12);
\draw (12) -- (13);
\draw (21) -- (22);
\draw (11) -- (21);
\draw (21) -- (31);
\draw (31) -- (41);
\draw (12) -- (22);
\end{tikzpicture}
\end{align*}
\end{example}

For a tableau $T$ of shape $\nu / \lambda$, we use the notation $\ic(T)$ to mean $\ic(\nu / \lambda)$.

\begin{definition}
Let $T$ be a skew increasing $\mathcal{P}$-tableau. 
We define its rectification step sets, $S_i$, recursively. 
First, $S_0 = \{T\}$. Next, 
\[S_{n+1} = \{\mathsf{Slide}_\gamma(S) : S \in S_{n} \text{ and } \emptyset \neq \gamma \subseteq \ic(S) \}.\]
The \textbf{rectifications} of $T$ are the elements of the \textbf{rectification set}
\[\mathsf{rects}(T) =\{U : U  \in S_n \text{ for some } n \in \mathbb{Z}_{\geq 0} \text{ and } U \text{ is of straight shape}\}.\]
To denote that $U$ is a rectification of $T$ given by sliding the sequence of sets of inner corners $(\gamma_1, \ldots, \gamma_n)$, we write $T \xrightarrow{\gamma_1, \ldots, \gamma_n} U$.
\end{definition}
\begin{example}
We do an example which has two rectifications.
Let 
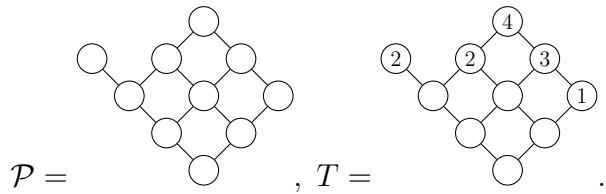
\[\mathcal{P} = 
\begin{tikzpicture} [every node/.append style={circle, draw=black, inner sep=0pt, minimum size=16pt}, every draw/.append style={black, thick}]
\node  (11) at (0,0) {};
\node [above left of=11] (12)  {};
\node [above left of=12] (13)  {};
\node [above left of=13] (14)  {};
\node [above right of=11] (21) {};
\node [above left of=21] (22)  {};
\node [above left of=22] (23)  {};
\node [above right of=21] (31) {};
\node [above left of=31] (32)  {};
\node [above left of=32] (33)  {};
\draw (11) -- (12);
\draw (12) -- (13);
\draw (13) -- (14);
\draw (21) -- (22);
\draw (22) -- (23);
\draw (31) -- (32);
\draw (32) -- (33);
\draw (11) -- (21);
\draw (21) -- (31);
\draw (12) -- (22);
\draw (22) -- (32);
\draw (23) -- (33);
\draw (13) -- (23);
\end{tikzpicture}, \;
T = 
\begin{tikzpicture} [every node/.append style={circle, draw=black, inner sep=0pt, minimum size=16pt}, every draw/.append style={black, thick}]
\node  (11) at (0,0) {};
\node [above left of=11] (12)  {};
\node [above left of=12] (13)  {};
\node [above left of=13] (14)  {2};
\node [above right of=11] (21) {};
\node [above left of=21] (22)  {};
\node [above left of=22] (23)  {2};
\node [above right of=21] (31) {1};
\node [above left of=31] (32)  {3};
\node [above left of=32] (33)  {4};
\draw (11) -- (12);
\draw (12) -- (13);
\draw (13) -- (14);
\draw (21) -- (22);
\draw (22) -- (23);
\draw (31) -- (32);
\draw (32) -- (33);
\draw (11) -- (21);
\draw (21) -- (31);
\draw (12) -- (22);
\draw (22) -- (32);
\draw (23) -- (33);
\draw (13) -- (23);
\end{tikzpicture}.
\]
One possible rectification follows this path
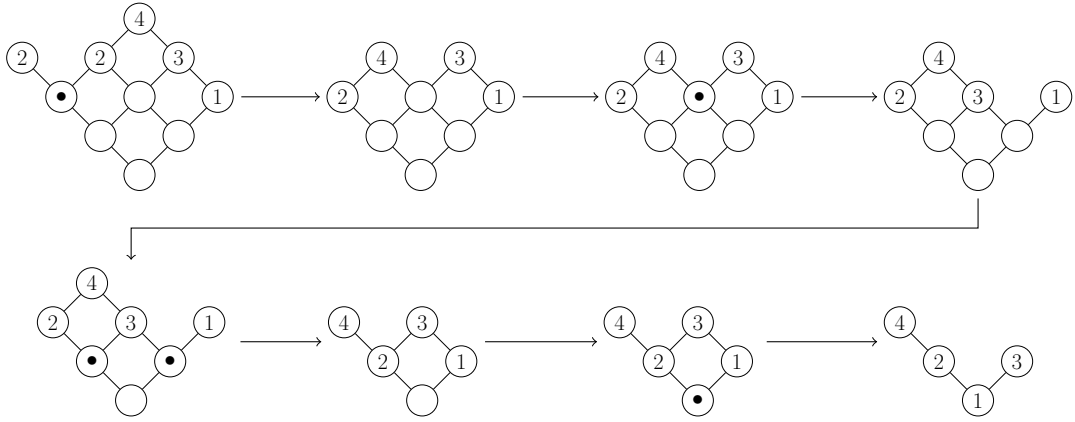
\[
\begin{tikzpicture} [arrow/.style = {thick,-stealth}]
\node  (T1) at (0,0){
  \begin{tikzpicture} [every node/.append style={circle, draw=black, inner sep=0pt, minimum size=24pt, font=\Large}, every draw/.append style={black, thick}, node distance = 1.5cm,anchor=base,baseline,]
  \node  (11) at (0,0) {};
  \node [above left of=11] (12)  {};
  \node [above left of=12] (13)  {$\bullet$};
  \node [above left of=13] (14)  {2};
  \node [above right of=11] (21) {};
  \node [above left of=21] (22)  {};
  \node [above left of=22] (23)  {2};
  \node [above right of=21] (31) {1};
  \node [above left of=31] (32)  {3};
  \node [above left of=32] (33)  {4};
  \draw (11) -- (12);
  \draw (12) -- (13);
  \draw (13) -- (14);
  \draw (21) -- (22);
  \draw (22) -- (23);
  \draw (31) -- (32);
  \draw (32) -- (33);
  \draw (11) -- (21);
  \draw (21) -- (31);
  \draw (12) -- (22);
  \draw (22) -- (32);
  \draw (23) -- (33);
  \draw (13) -- (23);
  \end{tikzpicture}
};
\node[right= -.2cm of T1.east](T1r){};
\node [right = 4cm of T1.south, anchor=south] (T2) {
     \begin{tikzpicture} [every node/.append style={circle, draw=black, inner sep=0pt, minimum size=24pt, font=\Large}, every draw/.append style={black, thick}, node distance = 1.5cm,anchor=base,baseline,]
    \node  (11) at (0,0) {};
    \node [above left of=11] (12)  {};
    \node [above left of=12] (13)  {2};
    \node [above right of=11] (21) {};
    \node [above left of=21] (22)  {};
    \node [above left of=22] (23)  {4};
    \node [above right of=21] (31) {1};
    \node [above left of=31] (32)  {3};
    \draw (11) -- (12);
    \draw (12) -- (13);
    \draw (21) -- (22);
    \draw (22) -- (23);
    \draw (31) -- (32);
    \draw (11) -- (21);
    \draw (21) -- (31);
    \draw (12) -- (22);
    \draw (22) -- (32);
    \draw (13) -- (23);
    \end{tikzpicture}
};
\node [right = .01cm of  T1r -| T2.west](T2l) {};
\node [right = -.2cm of T2l -| T2.east] (T2r) {};
\draw [->] (T1r) -- (T2l);
\node [right= 1cm of T2] (T3){
  \begin{tikzpicture} [every node/.append style={circle, draw=black, inner sep=0pt, minimum size=24pt, font=\Large}, every draw/.append style={black, thick}, node distance = 1.5cm,anchor=base,baseline,]
  \node  (11) at (0,0) {};
  \node [above left of=11] (12)  {};
  \node [above left of=12] (13)  {2};
  \node [above right of=11] (21) {};
  \node [above left of=21] (22)  {$\bullet$};
  \node [above left of=22] (23)  {4};
  \node [above right of=21] (31) {1};
  \node [above left of=31] (32)  {3};
  \draw (11) -- (12);
  \draw (12) -- (13);
  \draw (21) -- (22);
  \draw (22) -- (23);
  \draw (31) -- (32);
  \draw (11) -- (21);
  \draw (21) -- (31);
  \draw (12) -- (22);
  \draw (22) -- (32);
  \draw (13) -- (23);
  \end{tikzpicture}
};
\node [right = .01cm of  T2r -| T3.west](T3l) {};
\node [right = -.2cm of T3l -| T3.east] (T3r) {};
\draw [->] (T2r) -- (T3l);

\node [right=1cm of T3] (T4){
  \begin{tikzpicture} [every node/.append style={circle, draw=black, inner sep=0pt, minimum size=24pt, font=\Large}, every draw/.append style={black, thick}, node distance = 1.5cm,anchor=base,baseline,]
  \node  (11) at (0,0) {};
  \node [above left of=11] (12)  {};
  \node [above left of=12] (13)  {2};
  \node [above right of=11] (21) {};
  \node [above left of=21] (22)  {3};
  \node [above left of=22] (23)  {4};
  \node [above right of=21] (31) {1};
  \draw (11) -- (12);
  \draw (12) -- (13);
  \draw (21) -- (22);
  \draw (22) -- (23);
  \draw (11) -- (21);
  \draw (21) -- (31);
  \draw (12) -- (22);
  \draw (13) -- (23);
  \end{tikzpicture}
};
\node [right = .01cm of  T3r -| T4.west](T4l) {};
\draw [->] (T3r) -- (T4l);

\node [below right = 3cm and .15cm of T1.south, anchor=south] (S1){
  \begin{tikzpicture} [every node/.append style={circle, draw=black, inner sep=0pt, minimum size=24pt, font=\Large}, every draw/.append style={black, thick}, node distance = 1.5cm,anchor=base,baseline,]
  \node  (11) at (0,0) {};
  \node [above left of=11] (12)  {$\bullet$};
  \node [above left of=12] (13)  {2};
  \node [above right of=11] (21) {$\bullet$};
  \node [above left of=21] (22)  {3};
  \node [above left of=22] (23)  {4};
  \node [above right of=21] (31) {1};
  \draw (11) -- (12);
  \draw (12) -- (13);
  \draw (21) -- (22);
  \draw (22) -- (23);
  \draw (11) -- (21);
  \draw (21) -- (31);
  \draw (12) -- (22);
  \draw (13) -- (23);
  \end{tikzpicture}
};
\node [above = .32cm of S1.north] (step){};
\node [above = -.1cm of S1.east] (S1r) {};
\node [below left = 3cm and .24cm of T2.south, anchor=south](S2){
  \begin{tikzpicture} [every node/.append style={circle, draw=black, inner sep=0pt, minimum size=24pt, font=\Large}, every draw/.append style={black, thick}, node distance = 1.5cm,anchor=base,baseline,]
  \node  (11) at (0,0) {};
  \node [above left of=11] (12)  {2};
  \node [above left of=12] (13)  {4};
  \node [above right of=11] (21) {1};
  \node [above left of=21] (22)  {3};
  \draw (11) -- (12);
  \draw (12) -- (13);
  \draw (21) -- (22);
  \draw (11) -- (21);
  \draw (12) -- (22);
  \end{tikzpicture}
};
\node [right = .01cm of  S1r -| S2.west](S2l) {};
\node [right = -.2cm of S2l -| S2.east] (S2r) {};
\draw [->] (S1r) -- (S2l);
\node [below left = 3cm and .24cm of T3.south, anchor=south] (S3){
  \begin{tikzpicture} [every node/.append style={circle, draw=black, inner sep=0pt, minimum size=24pt, font=\Large}, every draw/.append style={black, thick}, node distance = 1.5cm,,anchor=base,baseline,]
  \node  (11) at (0,0) {$\bullet$};
  \node [above left of=11] (12)  {2};
  \node [above left of=12] (13)  {4};
  \node [above right of=11] (21) {1};
  \node [above left of=21] (22)  {3};
  \draw (11) -- (12);
  \draw (12) -- (13);
  \draw (21) -- (22);
  \draw (11) -- (21);
  \draw (12) -- (22);
  \end{tikzpicture}\;
};
\node [right = .01cm of  S2r -| S3.west](S3l) {};
\node [right = -.2cm of S3l -| S3.east] (S3r) {};
\draw [->] (S2r) -- (S3l);
\node [below left = 3cm and .26cm of T4.south, anchor=south] (S4){
  \begin{tikzpicture} [every node/.append style={circle, draw=black, inner sep=0pt, minimum size=24pt, font=\Large}, every draw/.append style={black, thick}, node distance = 1.5cm,anchor=base,baseline,]
  \node  (11) at (0,0) {1};
  \node [above left of=11] (12)  {2};
  \node [above left of=12] (13)  {4};
  \node [above right of=11] (21) {3};
  \draw (11) -- (12);
  \draw (12) -- (13);
  \draw (11) -- (21);
  \end{tikzpicture}
};
\node [right = .01cm of  S3r -| S4.west](S4l) {};
\node [right = -.2cm of S4l -| S4.east] (S4r) {};
\draw [->] (S3r) -- (S4l);

 \draw[->] (T4.south) |- (step) -| (S1.north);

\end{tikzpicture}
\]
so one rectification of $T$ is \begin{tikzpicture} [every node/.append style={circle, draw=black, inner sep=0pt, minimum size=16pt}, every draw/.append style={black, thick}, baseline={([yshift=-.5ex]current bounding box.center)}]
\node  (11) at (0,0) {1};
\node [above left of=11] (12)  {2};
\node [above left of=12] (13)  {4};
\node [above right of=11] (21) {3};
\draw (11) -- (12);
\draw (12) -- (13);
\draw (11) -- (21);
\end{tikzpicture}.

Another rectification follows this path
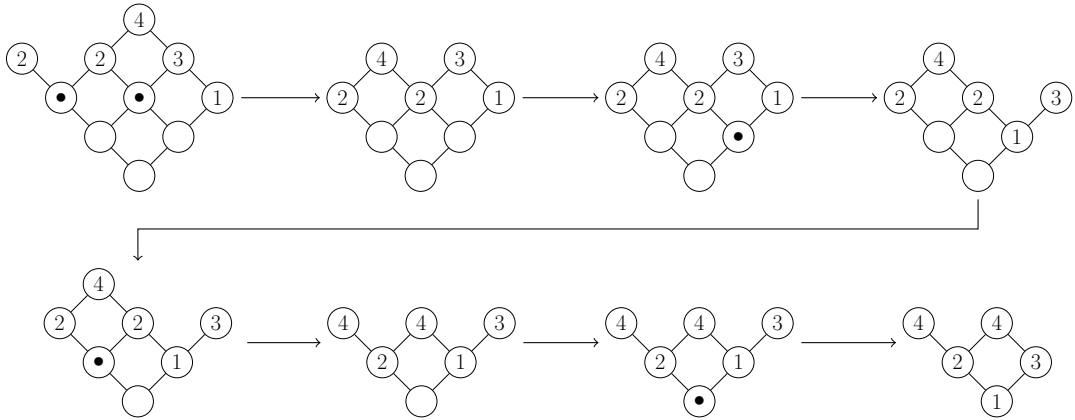
\[
\begin{tikzpicture} [arrow/.style = {thick,-stealth}]
\node (T1) at (0,0) {
  \begin{tikzpicture} [every node/.append style={circle, draw=black, inner sep=0pt, minimum size=24pt, font=\Large}, every draw/.append style={black, thick}, node distance = 1.5cm,anchor=base,baseline,]
  \node  (11) at (0,0) {};
  \node [above left of=11] (12)  {};
  \node [above left of=12] (13)  {$\bullet$};
  \node [above left of=13] (14)  {2};
  \node [above right of=11] (21) {};
  \node [above left of=21] (22)  {$\bullet$};
  \node [above left of=22] (23)  {2};
  \node [above right of=21] (31) {1};
  \node [above left of=31] (32)  {3};
  \node [above left of=32] (33)  {4};
  \draw (11) -- (12);
  \draw (12) -- (13);
  \draw (13) -- (14);
  \draw (21) -- (22);
  \draw (22) -- (23);
  \draw (31) -- (32);
  \draw (32) -- (33);
  \draw (11) -- (21);
  \draw (21) -- (31);
  \draw (12) -- (22);
  \draw (22) -- (32);
  \draw (23) -- (33);
  \draw (13) -- (23);
  \end{tikzpicture}
};
\node[right= -.2cm of T1.east](T1r){};

\node [right = 4cm of T1.south, anchor=south] (T2) {
  \begin{tikzpicture} [every node/.append style={circle, draw=black, inner sep=0pt, minimum size=24pt, font=\Large}, every draw/.append style={black, thick}, node distance = 1.5cm,anchor=base,baseline,]
  \node  (11) at (0,0) {};
  \node [above left of=11] (12)  {};
  \node [above left of=12] (13)  {2};
  \node [above right of=11] (21) {};
  \node [above left of=21] (22)  {2};
  \node [above left of=22] (23)  {4};
  \node [above right of=21] (31) {1};
  \node [above left of=31] (32)  {3};
  \draw (11) -- (12);
  \draw (12) -- (13);
  \draw (21) -- (22);
  \draw (23) -- (22);
  \draw (23) -- (13);
  \draw (31) -- (32);
  \draw (11) -- (21);
  \draw (21) -- (31);
  \draw (12) -- (22);
  \draw (22) -- (32);
  \end{tikzpicture}
};
\node [right = .01cm of  T1r -| T2.west](T2l) {};
\node [right = -.2cm of T2l -| T2.east] (T2r) {};
\draw [->] (T1r) -- (T2l);

\node [right = 1cm of T2] (T3) {
  \begin{tikzpicture} [every node/.append style={circle, draw=black, inner sep=0pt, minimum size=24pt, font=\Large}, every draw/.append style={black, thick}, node distance = 1.5cm,anchor=base,baseline,]
  \node  (11) at (0,0) {};
  \node [above left of=11] (12)  {};
  \node [above left of=12] (13)  {2};
  \node [above right of=11] (21) {$\bullet$};
  \node [above left of=21] (22)  {2};
  \node [above left of=22] (23)  {4};
  \node [above right of=21] (31) {1};
  \node [above left of=31] (32)  {3};
  \draw (11) -- (12);
  \draw (12) -- (13);
  \draw (21) -- (22);
  \draw (23) -- (22);
  \draw (23) -- (13);
  \draw (31) -- (32);
  \draw (11) -- (21);
  \draw (21) -- (31);
  \draw (12) -- (22);
  \draw (22) -- (32);
  \end{tikzpicture}
};
\node [right = .01cm of  T2r -| T3.west](T3l) {};
\node [right = -.2cm of T3l -| T3.east] (T3r) {};
\draw [->] (T2r) -- (T3l);

\node [right = 1cm of T3] (T4) {
  \begin{tikzpicture} [every node/.append style={circle, draw=black, inner sep=0pt, minimum size=24pt, font=\Large}, every draw/.append style={black, thick}, node distance = 1.5cm,anchor=base,baseline,]
  \node  (11) at (0,0) {};
  \node [above left of=11] (12)  {};
  \node [above left of=12] (13)  {2};
  \node [above right of=11] (21) {1};
  \node [above left of=21] (22)  {2};
  \node [above left of=22] (23)  {4};
  \node [above right of=21] (31) {3};
  \draw (11) -- (12);
  \draw (12) -- (13);
  \draw (21) -- (22);
  \draw (23) -- (22);
  \draw (23) -- (13);
  \draw (11) -- (21);
  \draw (21) -- (31);
  \draw (12) -- (22);
  \end{tikzpicture}
};
\node [right = .01cm of  T3r -| T4.west](T4l) {};
\node [right = -.2cm of T4l -| T4.east] (T4r) {};
\draw [->] (T3r) -- (T4l);

\node [below right = 3cm and .24cm of T1.south, anchor=south] (S1){
  \begin{tikzpicture} [every node/.append style={circle, draw=black, inner sep=0pt, minimum size=24pt, font=\Large}, every draw/.append style={black, thick}, node distance = 1.5cm,anchor=base,baseline,]
  \node  (11) at (0,0) {};
  \node [above left of=11] (12)  {$\bullet$};
  \node [above left of=12] (13)  {2};
  \node [above right of=11] (21) {1};
  \node [above left of=21] (22)  {2};
  \node [above left of=22] (23)  {4};
  \node [above right of=21] (31) {3};
  \draw (11) -- (12);
  \draw (12) -- (13);
  \draw (21) -- (22);
  \draw (22) -- (23);
  \draw (11) -- (21);
  \draw (21) -- (31);
  \draw (12) -- (22);
  \draw (13) -- (23);
  \end{tikzpicture}
};
\node [above = .32cm of S1.north] (step){};
\node [above = -.1cm of S1.east] (S1r) {};
\node [below right = 3cm and .01cm of T2.south, anchor=south](S2){
  \begin{tikzpicture} [every node/.append style={circle, draw=black, inner sep=0pt, minimum size=24pt, font=\Large}, every draw/.append style={black, thick}, node distance = 1.5cm,anchor=base,baseline,]
  \node  (11) at (0,0) {};
  \node [above left of=11] (12)  {2};
  \node [above left of=12] (13)  {4};
  \node [above right of=11] (21) {1};
  \node [above left of=21] (22)  {4};
  \node [above right of=21] (31) {3};
  \draw (11) -- (12);
  \draw (12) -- (13);
  \draw (21) -- (22);
  \draw (11) -- (21);
  \draw (21) -- (31);
  \draw (12) -- (22);
  \end{tikzpicture}
};
\node [right = .01cm of  S1r -| S2.west](S2l) {};
\node [right = -.2cm of S2l -| S2.east] (S2r) {};
\draw [->] (S1r) -- (S2l);

\node [below right = 3cm and .001cm of T3.south, anchor=south] (S3){
  \begin{tikzpicture} [every node/.append style={circle, draw=black, inner sep=0pt, minimum size=24pt, font=\Large}, every draw/.append style={black, thick}, node distance = 1.5cm,,anchor=base,baseline,]
\node  (11) at (0,0) {$\bullet$};
\node [above left of=11] (12)  {2};
\node [above left of=12] (13)  {4};
\node [above right of=11] (21) {1};
\node [above left of=21] (22)  {4};
\node [above right of=21] (31) {3};
\draw (11) -- (12);
\draw (12) -- (13);
\draw (21) -- (22);
\draw (11) -- (21);
\draw (21) -- (31);
\draw (12) -- (22);
\end{tikzpicture}
};
\node [right = .01cm of  S2r -| S3.west](S3l) {};
\node [right = -.2cm of S3l -| S3.east] (S3r) {};
\draw [->] (S2r) -- (S3l);
\node [below left = 3cm and .01cm of T4.south, anchor=south] (S4){
  \begin{tikzpicture} [every node/.append style={circle, draw=black, inner sep=0pt, minimum size=24pt, font=\Large}, every draw/.append style={black, thick}, node distance = 1.5cm,,anchor=base,baseline,]
\node  (11) at (0,0) {1};
\node [above left of=11] (12)  {2};
\node [above left of=12] (13)  {4};
\node [above right of=11] (21) {3};
\node [above left of=21] (22)  {4};
\draw (11) -- (12);
\draw (12) -- (13);
\draw (21) -- (22);
\draw (11) -- (21);
\draw (12) -- (22);
\end{tikzpicture}
};
\node [right = .01cm of  S3r -| S4.west](S4l) {};
\node [right = -.2cm of S4l -| S4.east] (S4r) {};
\draw [->] (S3r) -- (S4l);

 \draw[->] (T4.south) |- (step) -| (S1.north);
\end{tikzpicture}
\]
so another rectification of $T$ is \begin{tikzpicture} [every node/.append style={circle, draw=black, inner sep=0pt, minimum size=16pt}, every draw/.append style={black, thick}, baseline={([yshift=-.5ex]current bounding box.center)}]
\node  (11) at (0,0) {1};
\node [above left of=11] (12)  {2};
\node [above left of=12] (13)  {4};
\node [above right of=11] (21) {3};
\node [above left of=21] (22)  {4};
\draw (11) -- (12);
\draw (12) -- (13);
\draw (21) -- (22);
\draw (11) -- (21);
\draw (12) -- (22);
\end{tikzpicture}. One can check that these two are the only rectifications of $T$.
\end{example}

  We say a skew increasing $\mathcal{P}$-tableau \textbf{rectifies uniquely} if it has exactly one rectification. We say an increasing $\mathcal{P}$-tableau $T$ of straight shape is a \textbf{unique rectification target (URT)} if every skew increasing $\mathcal{P}$-tableau which rectifies to $T$ rectifies uniquely.

\section{Unique rectification in slant sums}
\label{Section adding to minimum and maximal elements}
In this section, we explore the structure of URTs in posets that are built out of smaller posets by a slant sum operation.

\begin{definition}
\cite{Proctor:JACO}
Let $\mathcal{P}, \mathcal{Q}$ be disjoint posets. Assume $\mathcal{Q}$ has a minimum element $\hat{0}_\mathcal{Q}$. Let $p \in \mathcal{P}$. 
The \textbf{slant sum} of $\mathcal{Q}$ to $\mathcal{P}$ at $p$, denoted $\mathcal{P} \, _{p}/^{\hat{0}_\mathcal{Q}} \, \mathcal{Q}$ is the poset on $\mathcal{P} \sqcup \mathcal{Q}$ induced by imposing $\hat{0}_\mathcal{Q} \gtrdot p$ together with the orders on $\mathcal{P}$ and $\mathcal{Q}$. Because a poset's minimum is unique, we will usually drop the $\hat{0}_\mathcal{Q}$ and denote the slant sum as $\mathcal{P} \slantsum{p} \mathcal{Q}$. If $\mathcal{Q}_1, \dots , \mathcal{Q}_n$ are pairwise disjoint posets which each have minima, then $\mathcal{P} \slantsum{p} (\mathcal{Q}_1, \dots ,\mathcal{Q}_n)$ denotes the iterated slant sum of posets at $p$. (Clearly, the order of the $\mathcal{Q}_i$s does not matter.)
Finally, given $p_1, \dots, p_m \in \mathcal{P}$ and pairwise disjoint posets $\mathcal{Q}_i^j$ with minima, we write
\[\mathcal{P} \slantsum{p_1} (\mathcal{Q}_1^1, \dots, \mathcal{Q}_{r_1}^1) \slantsum{p_2} \cdots \slantsum{p_m} (\mathcal{Q}_1^m, \dots, \mathcal{Q}_{r_m}^m)\] to denote
the result of slant summing each $\mathcal{Q}_i^j$ onto $p_j$ (in any order).
\end{definition}

We say that a poset $\mathcal{P}$ is a \textbf{chain} if all pairs of elements of $\mathcal{P}$ are comparable, that is, if $\mathcal{P}$ is a total order. The size of a chain $\mathcal{P}$ is the number of its elements.

\begin{proposition}\label{thm:adding min}
Let $\mathcal{P}$ be a poset with $\hat{0}_\mathcal{P}$ and let $\mathcal{C} = \{ c\}$ be a chain of size $1$. Let $\mathcal{R} = \mathcal{C} \slantsum{c} \mathcal{P}$.

If $T$ is a a skew increasing $\mathcal{P}$-tableau that rectifies uniquely (in $\mathcal{P}$), then the skew increasing $\mathcal{R}$-tableau $T_\mathcal{R}$ of shape $\nu \sqcup c /\lambda \sqcup c$ defined by $T_\mathcal{R}(x) = T(x)$ rectifies uniquely (in $\mathcal{R}$).
\end{proposition}
\begin{proof}
Suppose $U$ is the unique rectification of $T$. Let $p$ be the minimum of $\mathcal{P}$. 

When we rectify $T_\mathcal{R}$ in $\mathcal{R}$, since $c$ is covered only by $p$, the last step of rectification is to slide $c$. Just before this final step, one has a tableau $S$ whose shape is an order ideal in $\mathcal{P} \subset \mathcal{R}$. 
Clearly, a process of slides producing $S$ from $T_\mathcal{R}$ (in $\mathcal{R}$) corresponds to a sequence of slides rectifying $T$ (in $\mathcal{P}$).
Thus, since $T$ rectifies uniquely to $U$, we have $S(x) = U(x)$ for all $x \in \mathcal{P}$. 

The final slide is necessarily $\mathsf{Slide}_{\{c\}}$, since $c$ is the only inner corner of $S$. Hence, there are no further choices to make and thus $T_\mathcal{R}$ rectifies uniquely. 
\end{proof}

\begin{remark}
\label{Remark counterexample tail}
The converse of Proposition~\ref{thm:adding min} is false. 
Let 
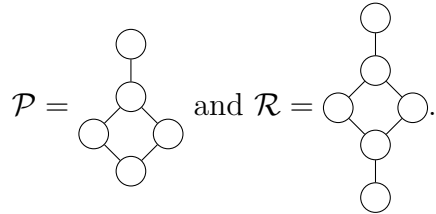
\[\text {$\mathcal{P} = \begin{tikzpicture}[every node/.append style={circle, draw=black, inner sep=0pt, minimum size=16pt}, every draw/.append style={black, thick}, baseline={([yshift=-.5ex]current bounding box.center)}]
\node  (b) at (0,0) {};
\node [above left of=b] (l)  {};
\node [above right of=b] (r)  {};
\node [above right of=l] (t) {};
\node [above of =t] (t1) {};
\draw (b) -- (l);
\draw (b) -- (r);
\draw (t) -- (r);
\draw (t) -- (l);
\draw (t) -- (t1);
\end{tikzpicture}$ and $\mathcal{R} = \begin{tikzpicture}[every node/.append style={circle, draw=black, inner sep=0pt, minimum size=16pt}, every draw/.append style={black, thick}, baseline={([yshift=-.5ex]current bounding box.center)}]
\node  (b1) at (0,0) {};
\node [above of=b1] (b) {};
\node [above left of=b] (l)  {};
\node [above right of=b] (r)  {};
\node [above right of=l] (t) {};
\node [above of =t] (t1) {};
\draw (b1) -- (b);
\draw (b) -- (l);
\draw (b) -- (r);
\draw (t) -- (r);
\draw (t) -- (l);
\draw (t) -- (t1);
\end{tikzpicture}$.}\]
For the skew increasing $\mathcal{P}$-tableaux $T_\mathcal{P} = \begin{tikzpicture}[every node/.append style={circle, draw=black, inner sep=0pt, minimum size=16pt}, every draw/.append style={black, thick}, baseline={([yshift=-.5ex]current bounding box.center)}]
\node  (b) at (0,0) {};
\node [above left of=b] (l)  {};
\node [above right of=b] (r)  {};
\node [above right of=l] (t) {1};
\node [above of =t] (t1) {2};
\draw (b) -- (l);
\draw (b) -- (r);
\draw (t) -- (r);
\draw (t) -- (l);
\draw (t) -- (t1);
\end{tikzpicture}$, we have 
\[
\mathsf{rects}(T_\mathcal{P}) = \left\{ \text{$\begin{tikzpicture}[every node/.append style={circle, draw=black, inner sep=0pt, minimum size=16pt}, every draw/.append style={black, thick}, baseline={([yshift=-.5ex]current bounding box.center)}]
\node  (b) at (0,0) {1};
\node [above left of=b] (l)  {2};
\node [above right of=b] (r)  {};
\node [above right of=l] (t) {};
\node [above of =t] (t1) {};
\draw (b) -- (l);
\draw (b) -- (r);
\draw (t) -- (r);
\draw (t) -- (l);
\draw (t) -- (t1);
\end{tikzpicture}$, $\begin{tikzpicture}[every node/.append style={circle, draw=black, inner sep=0pt, minimum size=16pt}, every draw/.append style={black, thick}, baseline={([yshift=-.5ex]current bounding box.center)}]
\node  (b) at (0,0) {1};
\node [above left of=b] (l)  {};
\node [above right of=b] (r)  {2};
\node [above right of=l] (t) {};
\node [above of =t] (t1) {};
\draw (b) -- (l);
\draw (b) -- (r);
\draw (t) -- (r);
\draw (t) -- (l);
\draw (t) -- (t1);
\end{tikzpicture}$, and $\begin{tikzpicture}[every node/.append style={circle, draw=black, inner sep=0pt, minimum size=16pt}, every draw/.append style={black, thick}, baseline={([yshift=-.5ex]current bounding box.center)}]
\node  (b) at (0,0) {1};
\node [above left of=b] (l)  {2};
\node [above right of=b] (r)  {2};
\node [above right of=l] (t) {};
\node [above of =t] (t1) {};
\draw (b) -- (l);
\draw (b) -- (r);
\draw (t) -- (r);
\draw (t) -- (l);
\draw (t) -- (t1);
\end{tikzpicture}$} \right\}.
\]
However, the skew $\mathcal{R}$-tableaux $T_\mathcal{R} = \begin{tikzpicture}[every node/.append style={circle, draw=black, inner sep=0pt, minimum size=16pt}, every draw/.append style={black, thick}, baseline={([yshift=-.5ex]current bounding box.center)}]
\node  (b1) at (0,0) {};
\node [above of=b1] (b) {};
\node [above left of=b] (l)  {};
\node [above right of=b] (r)  {};
\node [above right of=l] (t) {1};
\node [above of =t] (t1) {2};
\draw (b1) -- (b);
\draw (b) -- (l);
\draw (b) -- (r);
\draw (t) -- (r);
\draw (t) -- (l);
\draw (t) -- (t1);
\end{tikzpicture}$ rectifies uniquely to $\begin{tikzpicture}[every node/.append style={circle, draw=black, inner sep=0pt, minimum size=16pt}, every draw/.append style={black, thick}, baseline={([yshift=-.5ex]current bounding box.center)}]
\node  (b1) at (0,0) {1};
\node [above of=b1] (b) {2};
\node [above left of=b] (l)  {};
\node [above right of=b] (r)  {};
\node [above right of=l] (t) {};
\node [above of =t] (t1) {};
\draw (b1) -- (b);
\draw (b) -- (l);
\draw (b) -- (r);
\draw (t) -- (r);
\draw (t) -- (l);
\draw (t) -- (t1);
\end{tikzpicture}$.
\end{remark}

\begin{remark}\label{rem:URT_not_pchain}
Extending a poset by a new maximum element does not preserve unique rectification targets. For example, it is easy to check that
$\begin{tikzpicture}[every node/.append style={circle, draw=black, inner sep=0pt, minimum size=16pt}, every draw/.append style={black, thick}, baseline={([yshift=-.5ex]current bounding box.center)}]
\node  (b) at (0,0) {1};
\node [above left of=b] (l)  {2};
\node [above right of=b] (r)  {};
\node [above right of=l] (t) {};
\draw (b) -- (l);
\draw (b) -- (r);
\draw (t) -- (r);
\draw (t) -- (l);
\end{tikzpicture}$
is a unique rectification target in the poset $\begin{tikzpicture}[every node/.append style={circle, draw=black, inner sep=0pt, minimum size=16pt}, every draw/.append style={black, thick}, baseline={([yshift=-.5ex]current bounding box.center)}]
\node  (b) at (0,0) {};
\node [above left of=b] (l)  {};
\node [above right of=b] (r)  {};
\node [above right of=l] (t) {};
\draw (b) -- (l);
\draw (b) -- (r);
\draw (t) -- (r);
\draw (t) -- (l);
\end{tikzpicture}$.
 However, as shown in Remark~\ref{Remark counterexample tail}, $\begin{tikzpicture}[every node/.append style={circle, draw=black, inner sep=0pt, minimum size=16pt}, every draw/.append style={black, thick}, baseline={([yshift=-.5ex]current bounding box.center)}]
\node  (b) at (0,0) {1};
\node [above left of=b] (l)  {2};
\node [above right of=b] (r)  {};
\node [above right of=l] (t) {};
\node [above of =t] (t1) {};
\draw (b) -- (l);
\draw (b) -- (r);
\draw (t) -- (r);
\draw (t) -- (l);
\draw (t) -- (t1);
\end{tikzpicture}$ is not a unique rectification target in the poset $\begin{tikzpicture}[every node/.append style={circle, draw=black, inner sep=0pt, minimum size=16pt}, every draw/.append style={black, thick}, baseline={([yshift=-.5ex]current bounding box.center)}]
\node  (b) at (0,0) {};
\node [above left of=b] (l)  {};
\node [above right of=b] (r)  {};
\node [above right of=l] (t) {};
\node [above of =t] (t1) {};
\draw (b) -- (l);
\draw (b) -- (r);
\draw (t) -- (r);
\draw (t) -- (l);
\draw (t) -- (t1);
\end{tikzpicture}$.
\end{remark}

We now introduce some useful notation and terminology that we will need. Suppose $\mathcal{Q}$ is a shape of $\mathcal{P}$, and $T$ is a skew increasing $\mathcal{P}$-tableau. Then, the \textbf{restriction} of $T$ to $\mathcal{Q}$, denoted $T|_\mathcal{Q}$ is the increasing $\mathcal{Q}$-tableau given by restricting the domain of $T$ to those elements that are in $\mathcal{Q}$.

\begin{definition}
Suppose $\mathcal{F}$ is an order filter of a poset $\mathcal{P}$ with a $\hat{0}_\mathcal{F}$. We say $\mathcal{F}$ is a \textbf{funnel} if, for all $p \in \mathcal{P} \setminus \mathcal{F}$ with $p<f$ for some $f \in \mathcal{F}$, we have $p<\hat{0}_\mathcal{F}$.
\end{definition}

Note that, in particular, the embedded copy of $\mathcal{Q}$ in any slant sum $\mathcal{P} \, _{p}/^{\hat{0}_\mathcal{Q}} \, \mathcal{Q}$ forms a funnel.

\begin{definition}
\label{definition corresponding}
Let $\mathcal{F}$ be a funnel of a poset $\mathcal{P}$. Suppose $T$ is a skew increasing $\mathcal{P}$-tableau of shape $\nu/\lambda$ with a rectification $U$.
Then, the \textbf{corresponding} skew increasing $\mathcal{F}$-tableau of $T$ for $U$, denoted as $(T \rightarrow U)|_\mathcal{F},$ is the increasing $\mathcal{F}$-tableau defined by 
\[(T \rightarrow U)|_\mathcal{F} = T|_\mathcal{E},\]
where
\[\mathcal{E} \coloneqq
		\begin{cases} 
            {\{x \in \nu/\lambda \cap \mathcal{F} : T(x) \geq U(\hat{0}_\mathcal{F})\}}, & \text{ if } \hat{0}_\mathcal{F} \in \nu/\lambda; \\
            \emptyset, & \text{ if } \hat{0}_\mathcal{F} \notin \nu/\lambda.
        \end{cases}
        \]
\end{definition}

\begin{proposition}
\label{Theorem corresponding respects slides}
Let $\mathcal{F}$ be a funnel of a poset $\mathcal{P}$ and let $T$ be a skew increasing $\mathcal{P}$-tableau.  
If $U$ is a rectification of $T$, then $U|_\mathcal{F}$ is a rectification of $(T \rightarrow U)|_\mathcal{F}$. 
\end{proposition}
\begin{proof}
For concision, write $F$ for $(T \rightarrow U)|_\mathcal{F}$. The proposition is trivial if $F$ is the empty tableau, so we assume $F$ is supported on at least one element of $\mathcal{F}$.

Let $T \xrightarrow{\gamma_1, \dots, \gamma_n} U$.
Using the sequence of sets of inner corners $\gamma_1, \dots ,\gamma_n$, we will recursively construct sets of inner corners $\theta_1, \dots ,\theta_n$ that rectify $F$ to $S|_\mathcal{F}$.

We write $T_i^j$  to represent $T$ \emph{after} $i$ slides and just \emph{before} the $j$th swap; that is, 
\[T_i^j \coloneqq \begin{cases}
\mathsf{Swap}_{\bullet,j-1} \circ \cdots \circ \mathsf{Swap}_{\bullet,1} \circ \mathsf{AddDots}_{\gamma_{i+1}} \circ \mathsf{Slide}_{\gamma_1, \ldots, \gamma_{i}}(T), & \text{ if } j \geq 1;  \\
\mathsf{Slide}_{\gamma_1, \ldots, \gamma_{i}}(T), &\text{ if } j = 0.
\end{cases} \]
In particular, $T_0^0 = T$ and $T_0^1 = \mathsf{AddDots}_{\gamma_1}(T)$.

Let $F_0 \coloneqq F$.
For $1\leq i \leq n$, recursively define  $\theta_i$ and $F_i$ as follows
\begin{align*}
\theta_i \coloneqq  \{ & c \in \ic(F_{i-1}) : \text{there exists a $j$ with $T_i^j(c) = \bullet$ and} \\ 
& \text{$T_i^{j+1}(f) = \bullet$ for some $f \gtrdot c$ with $f \in \domain(F_{i-1})$}\}
\end{align*}
and
\[F_i \coloneqq \mathsf{Slide}_{\theta_1, \ldots, \theta_i}(F).\]
Similar to $T_i^j$, we let $F_i^j$ be
\[F_i^j \coloneqq \begin{cases}
\mathsf{Swap}_{\bullet,j-1} \circ \cdots \circ \mathsf{Swap}_{\bullet,1} \circ \mathsf{AddDots}_{\theta_{i+1}}(F_i), &  \text{ if } j \geq 1;  \\
F_i, &\text{ if } j = 0.
\end{cases}\]
Finally, we define certain subsets of $\theta_i$ which will be useful in our analysis. Let $\theta_i^0 = \emptyset$ and, for $k > 0$, let
\begin{align*}
\theta_i^k \coloneqq  \{ & c \in \ic(F_{i-1}) : \text{there exists a $j^\star \geq k$ with $T_i^{j^\star}(c) = \bullet$ and} \\ 
& \text{$T_i^{j^\star+1}(f) = \bullet$ for some $f \gtrdot c$ with $f \in \domain(F_{i-1})$}\}.
\end{align*}
Note that $\theta_i^1 = \theta_i$.

Set
$m \coloneqq U(\hat{0}_\mathcal{F}) = \min \range(F).$
For the remainder of this proof, we say $F_i^j$ and  $T_i^j$ {\bf N-agree} if $F_i^j$ and $T_i^j$ agree on all numeric labels within $\mathcal{F}$ greater than or equal to $m$; that is, for all $f \in \mathcal{F}$, if $T_i^j(f) \geq m$ or $F_i^j(f) \in \mathbb{Z}$, then $T_i^j(f) =F_i^j(f)$.
Additionally, for the remainder of this proof, we say that $F_i^j$ and $T_i^j$ {\bf agree} if they satisfy all the conditions:
\begin{enumerate}
\item[(A.0)] $F_i^j$ and $T_i^j$ N-agree;
\item[(A.1)] $F_i^j|_{\domain(F_i)} = T_i^j|_{\domain(F_i)}$;
\item[(A.2)] for all $c \in \ic (F_i)$, $F_i^j(c) = \bullet$ if and only if $c \in \theta_{i+1}^j$.
\end{enumerate}

We now establish inductively that $F_i^j$ and $T_i^j$ agree for all $i$ and $j$. 

First, note that, since $F_0^0 = F = (T \rightarrow U)|_\mathcal{F}$ is by definition a restriction of $T = T_0^0$ to the subposet of $\mathcal{F}$ consisting of those $x \in \mathcal{F}$ with $T_0^0(x) \geq m$, $F_0^0$ and $T_0^0$ satisfy (A.0) and (A.1). Furthermore, $\theta_1^0 = \emptyset$ and $F_0^0$ has no $\bullet$s, which proves (A.2).

Now, inductively assume that $F_\ell^h$ and $T_\ell^h$ agree for all $\ell \leq i$ and $h \leq j$.
Let 
\[M  = \max \range (T|_\mathcal{F}) = \max \range (F).\]

\medskip
\noindent
{\sf (Case 1: $j =M+1$):}
Since the largest label in $F$ and $T|_\mathcal{F}$ is $M$, our inductive step is to show that $F_{i+1}^0$ and $T_{i+1}^0$ agree.
We have $F_{i+1}^0 = \mathsf{RemoveDots}(F_i^{M+1})$ and $T_{i+1}^0 = \mathsf{RemoveDots}(T_i^{M+1})$. Hence, since $F_i^M$ and $T_i^M$ satisfy (A.0), it is clear that $F_{i+1}^0$ and $T_{i+1}^0$ satisfy (A.0) since removing $\bullet$s does not change numeric values.
Since $F_{i+1}^0$ and $T_{i+1}^0$ satisfy (A.0) and $F_{i+1}^0$ has no $\bullet$s, we have that $F_{i+1}^0$ and $T_{i+1}^0$ satisfy (A.1). Furthermore, by definition $\theta_{i+1}^0 = \emptyset$ and $F_{i+1}^0$ has no $\bullet$s, proving (A.2).

\bigskip
\noindent
{\sf (Case 2: $j =0$):}
In this case, $F_{i}^1 = \mathsf{AddDots}_{\theta_{i+1}}(F_i^0)$ and $T_{i}^1 = \mathsf{AddDots}_{\gamma_{i+1}}(T_i^0)$. (We note that, by definition, $\theta_{i+1}$ is a subset of $F_i$'s inner corners, so $\mathsf{AddDots}_{\theta_{i+1}}(F_i^0)$ is valid.)
To show that $F_{i}^1$ and $T_{i}^1$ agree, we must verify (A.0)--(A.2).

{\bf Proof of (A.0):}
Since $F_i^0$ and $T_i^0$ N-agree by assumption and $\mathsf{AddDots}_{\theta_{i+1}}$ does not affect numerical labels, it is clear that $F_{i}^1$ and $T_{i}^1$ N-agree. 

{\bf Proof of (A.1):}
Let $f \in \domain(F_i)$. Then $F_i(f) \in \mathbb{Z}$, so
\[F_i^1(f) = F_i(f) = T_i(f) = F_i^1 (f) \in \mathbb{Z},\]
where the first and last equalities are by adding $\bullet$s not affecting numeric labels and the middle equality is by (A.0) for $F_i^0$ and $T_i^0$.

{\bf Proof of (A.2):} Since $F_i^1 = \mathsf{AddDots}_{\theta_{i+1}}(F_i)$, we have that $F_i(c) = \bullet$ if and only if $c \in \theta_{i+1} = \theta_{i+1}^1$.

\bigskip
\noindent
{\sf (Case 3: $0 < j < M+1$):}
We must verify (A.0)--(A.2) for $F_i^{j+1}$ and $T_i^{j+1}$. First, we prove some helpful claims.

\begin{claim}
\label{claim F bullet implies T bullet or large covers}
For all $c \in \ic(F_i)$, if $F_i^j(c) = \bullet$, then either $T_i^j(c) = \bullet$ or else, for all $f \in \domain(F_i^j)$ with $c \lessdot f$, we have $F_i^j(f) = T_i^j(f) > j$;
\end{claim}
\begin{proof}[Proof of Claim~\ref{claim F bullet implies T bullet or large covers}]
Let $c \in \ic(F_i)$ and suppose $F_i^j(c) = \bullet$. 
Then, by (A.2), $f \in \theta_{i+1}^j$, 
so there exists a $j^\star \geq j$ such that $T_i^{j^\star}(c) = \bullet$ and 
$T_i^{j^\star+1}(f_0) = \bullet$. There are two cases to consider: either $j^\star = j$ or $j^\star > j$. 

First, suppose $j^\star = j$. Then $T_i^j(c) = T_i^{j^\star}(c) = \bullet$, as desired.

Otherwise, suppose $j^\star > j$. Since $T_i^{j^\star}(c) = \bullet$, we have that for all $c' \in \mathcal{P}$ with $c' \gtrdot c$ that $T_i^j(c') \geq j^\star > j$. Thus, by the inductive (A.1), we have for all $f \in \domain(F_i)$ with $f \gtrdot c$ that $F_i^j(f) = T_i^j(f) > j$.
\end{proof}

\begin{claim}
\label{claim ic or domain}
Suppose $p \in \mathcal{P} \setminus \left(  \domain(F_i) \cup \ic(F_i) \right)$. Further suppose $p \lessdot f$ with $f \in \mathcal{F}$ and $F_i^j(f) = j$. Then, $F_i^j(p) \neq \bullet$ and $T_i^j(p) \neq \bullet$.
\end{claim}
\begin{proof}[Proof of Claim~\ref{claim ic or domain}]
Since $p \notin \domain(F_i) \cup \ic(F_i)$, we have $p \notin \domain(F_i^j)$, and so it is not the case that $F_i^j(p) = \bullet$. 

Suppose $T_i^j(p) = \bullet$. 
We will obtain a contradiction by deriving that $p \in \ic(F_i)$. Suppose $p' \gtrdot p$ and $p' \notin \domain(F_i)$. 

If $p' \in \domain(T_i^j)$, then since $T_i^j(p) = \bullet$, we have that $T_i^j(p') \geq j$ (since it above a $\bullet$ after the $(j-1)$st swap), and we know $j \geq m$ since $F_i^j(f) = j$ and $m$ is the least label in $F$.
Hence, by the inductive (A.1), we have $T_i^j(p') = F_i^j(p') \geq j$. Therefore, $F_i^j(p') = F_i(p')$, so $p' \in \domain(F_i)$, in violation of our assumption.
Thus, $p' \notin \domain(T_i^j)$.
 
Let $p'' > p'$. 
Because the domain of $T_i^j$ is a skew shape and $p' \notin \domain (T_i^j)$, we know that $p'' \notin \domain(T_i^j)$.
Thus, since $\domain(F_i) \subseteq \domain(T_i) \subseteq \domain(T_i^j)$, we have $p'' \notin \domain(F_i)$.
Thus, for any $p' \gtrdot p$ with $p' \notin \domain(F_i)$, we then have, for all $p'' > p'$, that $p'' \notin \domain(F_i)$. Thus, $p \in \ic(F_i)$, which is our desired contradiction.
\end{proof}

{\bf Proof of (A.0)}:
By the inductive (A.0), $F_i^j$ and $T_i^j$ agree on numeric labels greater than or equal to $m$. Since $F_i^{j+1}$ and $T_i^{j+1}$ are just $F_i^{j}$ and $T_i^{j}$ after applying $\mathsf{Swap}_{\bullet, j}$, it suffices to consider the movement of the $j$ labels. If $j< m$, we are done. Hence, assume $j \geq m$. We know by (A.0) that for all $f \in \mathcal{F}$, that $T_i^j(f) = j$ if and only if $F_i^j(f)=j$. Hence, it suffices to show that whenever $f \in \mathcal{F}$ with $F_i^j(f) =T_i^j(f) = j$ and whenever $\hat{f} \lessdot f$, we have $F_i^j(\hat{f}) = \bullet$ if and only if $T_i^j(\hat{f}) = \bullet$.
Fix $f \in \mathcal{F}$ and $\hat{f} \lessdot f$. Suppose $T_i^j(f) = F_i^j(f) = j$. We must show $F_i^j(\hat{f}) = \bullet$ if and only if $F_i^j(\hat{f}) = \bullet$. Either $\hat{f} \in \domain(F_i)$, $\hat{f} \in \ic(F_i)$, or $\hat{f} \notin \domain(F_i) \cup \ic(F_i)$.

Suppose $\hat{f} \in \domain(F_i)$. Then, the inductive (A.1) directly gives that $F_i^j(\hat{f}) = \bullet$ if and only if $T_i^j(\hat{f}) = \bullet$, as desired.

Suppose $\hat{f} \notin \domain(F_i) \cup \ic(F_i)$. Then, Claim~\ref{claim ic or domain} gives that $F_i^j(\hat{f}) \neq \bullet$ and $T_i^j(\hat{f}) \neq \bullet$.

Finally, suppose $\hat{f} \in \ic(F_i)$. If $F_i^j(\hat{f}) = \bullet$, then Claim~\ref{claim F bullet implies T bullet or large covers} gives that $T_i^j(\hat{f}) = \bullet$. Conversely, if $T_i^j(\hat{f}) = \bullet$, then $T_i^{j+1}(f) = \bullet$ by the definition of swapping, so by definition $\hat{f} \in \theta_{i+1}^j$. Hence, $F_i^j(\hat{f}) = \bullet$ by the inductive (A.2).

{\bf Proof of (A.1)}:
Let $f \in \domain(F_i)$. Then by the inductive (A.0), we have that $f \in \domain(T_i)$, so either $T_i^{j+1}(f) \in \mathbb{Z}$ or $T_i^{j+1}(f) = \bullet$.

If $T_i^{j+1}(f) \in \mathbb{Z}$, then by construction of the swapping process $T_i^{j+1}(f) \geq T_i(f)$. Moreover, $T_i(f) = F_i(f) \geq m$ by the inductive (A.0). Hence, $T_i^{j+1}(f) \geq m$. Thus, we have $F_i^{j+1}(f) = T_i^{j+1}(f)$ by (A.0) for $F_i^{j+1}$ and $T_i^{j+1}$ (which has already been established at this point). 

Now, suppose that $T_i^{j+1}(f) = \bullet$. If $F_i^{j+1}(f) \neq \bullet$, then $F_i^{j+1}(f)  \geq m$. Hence, by (A.0) for $F_i^{j+1}$ and $T_i^{j+1}$ (which has already been established at this point), we have 
\[\bullet \neq F_i^{j+1}(f) = T_i^{j+1}(f) = \bullet, \] which is a contradiction. Thus, $F_i^{j+1}(f) = \bullet$.

{\bf Proof of (A.2)}: 
We have 
\[
\{f \in \ic(F_i) : F_i^{j+1}(f) = \bullet \} = \{f \in \ic(F_i): F_i^j(f) = \bullet \text{ and $F_i^j(f') \neq j$ for all $f' \gtrdot f$}\}.
\]
By the inductive (A.2), this equals
\[
\{f \in \theta_i^j: \text{$F_i^j(f') \neq j$ for all $f' \gtrdot f$}\},
\]
which in turn equals 
\[
\{f \in \theta_i^j: \text{$T_i^j(f') \neq j$ for all $f' \in \domain(F_i)$ with $f' \gtrdot f$}\}
\]
by the inductive (A.0).
As this last set is the definition of $\theta_i^{j+1}$,
this completes the proof of (A.2) and hence the induction.

As a consequence of our induction, we have $F_n^0 = T_n^0|_{\mathcal{F}} = U|_\mathcal{F}$. Hence, 
\[
F \xrightarrow{\theta_1, \dots, \theta_n} U|_\mathcal{F},
\]
as  desired. 
\end{proof}

We now derive a few straightforward corollaries of Proposition~\ref{Theorem corresponding respects slides}. We will not use these corollaries in the sequel; however, they seem interesting to us for elucidating some of the structure of URTs among collections of related posets.

\begin{corollary}
\label{Corollary agree on filters with minimums}
Let $T$ be a skew increasing $\mathcal{P}$-tableau and let $\mathcal{F}$ be a funnel of $\mathcal{P}$.
Let $U, V$ be rectifications of $T$. 
If $U|_\mathcal{F}$ is a URT in $\mathcal{F}$ and $U(\hat{0}_\mathcal{F}) = V(\hat{0}_\mathcal{F})$, then $U|_\mathcal{F} = V|_\mathcal{F}$.
\end{corollary}
\begin{proof}
By definition, $(T \rightarrow U)|_\mathcal{F}$ = $(T \rightarrow V)|_\mathcal{F}$, since $U(\hat{0}_\mathcal{F}) = V(\hat{0}_\mathcal{F})$. 
By Proposition~\ref{Theorem corresponding respects slides}, $(T \rightarrow U)|_\mathcal{F}$ rectifies to $U|_\mathcal{F}$ and $(T \rightarrow V)|_\mathcal{F}$ rectifies to $V|_\mathcal{F}$. 
Since by assumption $U|_\mathcal{F}$ is a URT in $\mathcal{F}$, it follows that $U|_\mathcal{F} = V|_\mathcal{F}$.  
\end{proof}

\begin{definition}
The \textbf{bottom chain} $\mathcal{C}$ of a poset $\mathcal{P}$ is the order ideal of $\mathcal{P}$ constructed as follows.
Define
\[\min(\mathcal{P}) \coloneqq 
     	\begin{cases} 
            \{ \hat{0}_\mathcal{P} \}, & \text{ if $\mathcal{P}$ has a $\hat{0}_\mathcal{P}$}; \\
            \emptyset, & \text{ otherwise}. 
        \end{cases}
\]
We construct the shapes $C_i$ recursively.
Let $C_0 \coloneqq \min(\mathcal{P})$ and let $C_{i+1} \coloneqq C_i \cup \min(\mathcal{P} \setminus C_i)$.
Finally, the bottom chain of $\mathcal{P}$ is \[\mathcal{C} \coloneqq \bigcup_{n \in \mathbb{Z}_{\geq 0}} C_n.\]
\end{definition}

\begin{lemma}
\label{Lemma unique rectification in bottom chain}
Let $\mathcal{C}$ be the bottom chain of a poset $\mathcal{P}$. Let $U,V$ be rectifications of a skew increasing $\mathcal{P}$-tableau $T$. Then $U|_\mathcal{C} = V|_\mathcal{C}$.
\end{lemma}
\begin{proof}
In any rectification $W$ of $T$, the labels of $W|_\mathcal{C}$ must be just the smallest numbers in the range of $T$ in increasing order.
\end{proof}

\begin{corollary}
\label{Corollary slant sum tree base case}
Let $\mathcal{P}_1, \dots, \mathcal{P}_n $ be pairwise disjoint posets with minimum elements $\hat{0}_{\mathcal{P}_k}$. 
Let $\mathcal{C} = \{ c \}$ be a chain of size 1.
Construct the slant sum $\mathcal{R} \coloneqq \mathcal{C} \slantsum{c} (\mathcal{P}_1, \dots ,\mathcal{P}_n)$. For $U$ a straight-shaped increasing $\mathcal{R}$-tableau, if $U|_{\mathcal{P}_k}$ is an URT in $\mathcal{P}_k$ for each $1\leq k \leq n$, then $U$ is a URT in $\mathcal{R}$.
\end{corollary}
\begin{proof}
Suppose $U,V$ are rectifications of some skew increasing $\mathcal{R}$-tableau $T$. 
Since $c$ is in the bottom chain of $\mathcal{R}$, $U|_\mathcal{C} = V|_\mathcal{C}$ by Lemma~\ref{Lemma unique rectification in bottom chain}.
Let $m \coloneqq U(c) = V(c)$.
It remains to show that $U|_{P_k} = V|_{P_k}$ for all $k$. Fix some $\mathcal{P}_k$. By the increasingness of $U$ and $V$, we have $U(\hat{0}_{\mathcal{P}_k}) > m$ and $V(\hat{0}_{\mathcal{P}_k}) > m$, but both $U(\hat{0}_{\mathcal{P}_k})$ and $V(\hat{0}_{\mathcal{P}_k})$ must also be less than all the other labels in their respective tableaux for elements in $\mathcal{P}_k$ that are greater than $\hat{0}_{\mathcal{P}_k}$. Thus, since $\hat{0}_{\mathcal{P}_k}$ only covers $c$ and all the $\mathcal{P}_k$ are disjoint funnels, it is easy to see that
\[U(\hat{0}_{\mathcal{P}_k}) = V(\hat{0}_{\mathcal{P}_k}) = \min (\range(T|_{\mathcal{P}_k}) \setminus \{m\}).\]
Finally, since $U(\hat{0}_{\mathcal{P}_k}) = V(\hat{0}_{\mathcal{P}_k})$, we have $U|_{\mathcal{P}_k} = V|_{\mathcal{P}_k}$ by applying Corollary~\ref{Corollary agree on filters with minimums}.
\end{proof}

We will say a poset $\mathscr{T}$ is a \textbf{tree} if $\mathscr{T}$ has a $\hat{0}_\mathscr{T}$ and if each other $x \in \mathscr{T}$ has exactly one parent. 

\begin{corollary}
Let $\mathscr{T}$ be a tree. Let $t_1, \dots, t_n$ be (not necessarily distinct) elements of $\mathscr{T}$. Let $\mathcal{P}_1, \ldots, \mathcal{P}_n$ be pairwise disjoint posets with minimum elements.
Define
\[\mathcal{R}_0 \coloneqq \mathscr{T}\]
and
\[\text{ $\mathcal{R}_{i+1} = \mathcal{R}_i \slantsum{t_{i+1}} \mathcal{P}_{i+1}$ for $i = 0, \dots, n-1$}.\]
Let $U$ be a straight-shaped increasing $\mathcal{R}_n$-tableau.
If $U|_{\mathcal{P}_k}$ is an URT in $\mathcal{P}_k$ for all $1\leq k \leq n$, then $U$ is a URT in $\mathcal{R}_n$.
\end{corollary}
\begin{proof}
By repeated application of Corollary~\ref{Corollary slant sum tree base case}.
\end{proof}

\section{Restricted rectifications and $p$-chain URTs}
\label{Section slant sum}
Most of the results in this section are fairly technical lemmas that we will need later.

\begin{definition}
Let $\mathcal{P}$ be a poset, let $\mathcal{Q}$ be a subset of $\mathcal{P}$, and let $T$ be a skew increasing $\mathcal{P}$-tableau.
Then, the \textbf{rectifications of $T$ restricted to $\mathcal{Q}$} are the elements of the set
\[\mathsf{rects}|_\mathcal{Q} (T) := \{U|_\mathcal{Q} : U \in \mathsf{rects}(T) \}.\]
\end{definition}

\begin{lemma}\label{Lemma poset "sees" chain}
\gap 
\begin{enumerate}
\item Fix $p \in \mathcal{P}$ and let $k \coloneqq |\{x \in \mathcal{P}: x \leq p\}|$ be the cardinality of the principal order ideal generated by $p$. Let $\mathcal{C}$ be a chain poset of size $k$. If 
\[\mathcal{R} = \mathcal{P} \slantsum{p} (\mathcal{Q}_1, \dots, \mathcal{Q}_n)\] for pairwise disjoint posets $\mathcal{Q}_1, \dots, \mathcal{Q}_n$ with minimum elements
and $T$ is a skew increasing $\mathcal{R}$-tableau, then there is a skew increasing $(\mathcal{P} \slantsum{p} \mathcal{C})$-tableau $C$ such that
\[\mathsf{rects}|_\mathcal {P} (T) = \mathsf{rects}|_\mathcal{P} (C).\]
\item More generally, if $p_1,..,p_n \in \mathcal{P}$ are distinct, define $k_i \coloneqq |\{x \in \mathcal{P}: x \leq p_i\}|$ and let $\mathcal{C}_i$ be the chain poset of size $k_i$. Then, if 
\[\mathcal{R} = \mathcal{P} \slantsum{p_1} (\mathcal{Q}_1^1, \dots ,\mathcal{Q}_{m_1}^1) \slantsum{p_2} \dots \slantsum{p_n} (\mathcal{Q}_1^n,\dots, \mathcal{Q}^n_{m_n})\] for pairwise disjoint posets $\mathcal{Q}^1_1, \dots , \mathcal{Q}^n_{m_n}$ with minimum elements
and $T$ is a skew increasing $\mathcal{R}$-tableau, then there is a skew increasing $(\mathcal{P} \slantsum{p_1} \mathcal{C}_1 \slantsum{p_2}\dots\slantsum{p_n} \mathcal{C}_n)$-tableau $C$ such that
\[\mathsf{rects}|_\mathcal {P} (T) = \mathsf{rects}|_\mathcal{P} (C).\]
\end{enumerate}
\end{lemma}
\begin{proof}
We first prove (1).
Since there are $k$ weak ancestors of $p$, there are at most $k$ distinct labels from $\bigcup_h \mathcal{Q}_h$ that can swap into $\mathcal{P}$ during rectification of the skew increasing $\mathcal{R}$-tableau $T$.
Let $q_1 < \dots < q_m \in \mathbb{Z}_{>0}$ be the $m$ smallest labels of elements in 
$T|_{\bigcup_{h}\mathcal{Q}_h},$
where $m$ is the lesser of $k$ and the number of distinct labels in $T|_{\bigcup_{h}\mathcal{Q}_h}$. 
Fix a chain $\mathcal{C} = \{ c_1 <\dots < c_k\}$ of size $k$. 
Define the skew increasing $\mathcal{P} \slantsum{p} \mathcal{C}$-tableau $C$ as follows:
\[C(x) = 
\begin{cases}
T(x), & x \in \mathcal{P} \cap \domain(T); \\
q_r, & x = c_r \text{ and } 1 \leq i \leq m.
\end{cases}
\]

We will show $\mathsf{rects}|_\mathcal {P} (T) = \mathsf{rects}|_\mathcal{P} (C)$. First, we show that $\mathsf{rects}|_\mathcal {P} (T) \subseteq \mathsf{rects}|_\mathcal{P} (C)$. To do this, suppose $T \xrightarrow{\gamma_1, \dots, \gamma_n} U$ in $\mathcal{R}$. We must show that $U|_\mathcal{P} \in  \mathsf{rects}|_\mathcal{P} (C)$. We use the sequence of inner corners $\theta_i \coloneqq \gamma_i \cap \mathcal{P}$ to rectify $C$.

Write $T_i^j$  to represent $T$ \emph{after} $i$ slides and just \emph{before} the $j$th swap; that is, 
\[T_i^j \coloneqq \begin{cases}
\mathsf{Swap}_{\bullet,j-1} \circ \cdots \circ \mathsf{Swap}_{\bullet,1} \circ \mathsf{AddDots}_{\gamma_{i+1}} \circ \mathsf{Slide}_{\gamma_1, \ldots, \gamma_{i}}(T), & \text{ if } j \geq 1;  \\
\mathsf{Slide}_{\gamma_1, \ldots, \gamma_{i}}(T), &\text{ if } j = 0.
\end{cases} \]

Similarly, let $C_i^j$ be 
\[C_i^j \coloneqq \begin{cases}
\mathsf{Swap}_{\bullet,j-1} \circ \cdots \circ \mathsf{Swap}_{\bullet,1} \circ \mathsf{AddDots}_{\theta_{i+1}} \circ \mathsf{Slide}_{\theta_1, \ldots, \theta_{i}}(C), & \text{ if } j \geq 1;  \\
\mathsf{Slide}_{\theta_1, \ldots, \theta_{i}}(C), &\text{ if } j = 0.
\end{cases} \]

For the rest of this proof, we say $T_i^j$ and $C_i^j$ are {\bf similar} if they satisfy both of the following conditions:
\begin{enumerate}
\item[(S.0)] $T_i^j|_\mathcal{P} = C_i^j|_\mathcal{P}$;
\item[(S.1)] $\{T_i^j(q) \in \mathbb{Z} : q \in \bigcup_h \mathcal{Q}_h \text{ and } T_i^j(q) \leq q_m\} = \{C_i^j(c) \in \mathbb{Z}: c \in \mathcal{C} \text{ and } C_i^j(c) \leq q_m\}$.
\end{enumerate}

We show by induction that $T_i^j$ and $C_i^j$ are similar for all $i$ and $j$. In particular, this will yield $T_i^j|_\mathcal{P} = C_i^j|_\mathcal{P}$, proving that $U|_\mathcal{P} = T_n^0|_\mathcal{P} = C_n^0|_\mathcal{P} \in \mathsf{rects}|_\mathcal{P}(C)$, as desired.

By construction, we have that $T_0^0|_\mathcal{P} = T|_\mathcal{P} = C|_\mathcal{P} = C_0^0|_\mathcal{P}$ so $T_0^0$ and $C_0^0$ satisfy (S.0). Condition (S.1) for  $T_0^0$ and $C_0^0$ is also by construction.

Now, inductively assume $T_i^j$ and $C_i^j$ are similar. For concision, write $\mathbb{S}_i^j$ for the set \[\{T_i^j(q) \in \mathbb{Z} : q \in \bigcup_h \mathcal{Q}_h \text{ and } T_i^j(q) \leq q_m\} = \{C_i^j(c) \in \mathbb{Z}: c \in \mathcal{C} \text{ and } C_i^j(c) \leq q_m\}\] considered in the inductive (S.1) condition.
Let 
\[M  = \max \range (T) \geq \max \range (C).\]

\medskip
\noindent
{\sf (Case 1: $j =M+1$):}
We have that
\begin{align*}
T_{i+1}^0|_\mathcal{P} &= \mathsf{RemoveDots}(T_{i}^{M+1})|_\mathcal{P}= \mathsf{RemoveDots}(T_{i}^{M+1}|_\mathcal{P}) \\
&=\mathsf{RemoveDots}(C_{i}^{M+1}|_\mathcal{P}) = \mathsf{RemoveDots}(C_{i}^{M+1})|_\mathcal{P} \\
&= C_{i+1}^0|_\mathcal{P},
\end{align*}
so $T_{i+1}^0$ and $C_{i+1}^0$ satisfy (S.0). 
Removing $\bullet$s does not affect the numerical labels in $\bigcup_h \mathcal{Q}_h$ or $\mathcal{C}$, so (S.1) for $T_{i+1}^0$ and $C_{i+1}^0$ is immediate from (S.1) for $T_{i}^{M+1}$ and $C_{i}^{M+1}$. This completes this case.

Before turning to the next case, note that if none of the weak ancestors of $p$ are skewed out of $T_i^0|_\mathcal{P} = C_i^0|_\mathcal{P}$, then (S.0) and (S.1) continue holding in perpetuity since no elements of $\bigcup_h \mathcal{Q}_h$ or $\mathcal{C}$ will be involved in any swaps. Thus, for the remaining cases, assume at least one weak ancestor of $p$ is skewed out. 

Further, note that if $\mathbb{S}_i^j = \emptyset$, then since at least one weak ancestor of $p$ is skewed out, we have $\{T_i^j(q) \in \mathbb{Z} : q \in \bigcup_h \mathcal{Q}_h\} = \{C_i^j(c) \in \mathbb{Z}: c \in \mathcal{C} \} = \emptyset$ by the definition of $q_m$. Hence if $\mathbb{S}_i^j = \emptyset$, then (S.0) and (S.1) continue to hold in perpetuity, since neither tableau has labels outside of $\mathcal{P}$. Thus, for the remaining cases, we may further assume $\mathbb{S}_i^j \neq \emptyset$.

\medskip
\noindent
{\sf (Case 2: $j =0$):} First, we verify that $\theta_{i+1} \subseteq \ic(C_i^0)$. Let $t \in \theta_{i+1}$.  Since $t \in \theta_{i+1}$, we have that $t \in \gamma_{i+1} \cap \mathcal{P}$, so $t$ is an inner corner in $T_i^0|_\mathcal{P} = C_i^0|_\mathcal{P}$. (This equality is by the inductive (S.0).) Hence, since $C$ has skewed out nodes only in $\mathcal{P}$, it follows that $t \in \ic(C_i^0)$. Thus, $\theta_{i+1} \subseteq \ic(C_i^0)$.

Since $\theta_{i+1} = \gamma_{i+1} \cap \mathcal{P}$, we have 
\begin{align*}
T_i^{1}|_\mathcal{P} &=  \mathsf{AddDots}_{\gamma_{i+1}} (T_i^{0})|_\mathcal{P}
= \mathsf{AddDots}_{\theta_{i+1}} (T_i^{0})|_\mathcal{P} \\
&= \mathsf{AddDots}_{\theta_{i+1}} (T_i^{0}|_\mathcal{P})
= \mathsf{AddDots}_{\theta_{i+1}} (C_i^{0}|_\mathcal{P}) \\
&= C_i^1|_\mathcal{P},
\end{align*}
proving (S.0) for $T_i^1$ and $C_i^1$.
Similarly to the previous case, since adding $\bullet$s does not affect the numerical labels in $\bigcup_h \mathcal{Q}_h$ or $\mathcal{C}$, (S.1) for $T_i^1$ and $C_i^1$ is immediate from (S.1) for $T_i^0$ and $C_i^0$.

\medskip
\noindent
{\sf (Case 3: $0 < j < M+1$):}
We show (S.0) first.
Observe that the structure of $\mathcal{R}$, $\mathcal{P} \slantsum{p} \mathcal{C}$, and $\mathcal{P}$ easily ensures that if $q \in \mathcal{P}$, then 
\begin{itemize}
\item[(F.1)] for all $\hat{q} \lessdot q$ (in either $\mathcal{R}$ or $\mathcal{P} \slantsum{p} \mathcal{C}$), $\hat{q} \in \mathcal{P}$, and 
\item[(F.2)] if $q \neq p$, then for all $q' \gtrdot q$ (in either $\mathcal{R}$ or $\mathcal{P} \slantsum{p} \mathcal{C}$) we have that $q' \in \mathcal{P}$.
\end{itemize}
Now (F.1), (F.2), the fact that $T_i^j|_\mathcal{P} = C_i^j|_\mathcal{P}$, and the local nature of the swapping process, together ensure that $T_i^{j+1}(q) = F_i^{j+1}(q)$ for all $q \in \mathcal{P}$ with $q \neq p$. Furthermore, if $T_i^j(p) \neq \bullet$, then (F.1) ensures $T_i^{j+1}(p) = F_i^{j+1}(p)$. 

Hence, it remains to consider the situation where $T_i^j(p) =C_i^j(p)=\bullet$. Then
\begin{align*}
T_i^{j+1}(p) &= \begin{cases}
j, & \text{if $j = \min \{T_i^j(p'') \in \mathbb{Z}: p'' \in \mathcal{R}$ and $p'' > p$\}}\\
\bullet, & \text{otherwise}
\end{cases} \\ 
&= \begin{cases}
j, & \text{if $j = \min \{C_i^j(p'') \in \mathbb{Z}: p'' \in \mathcal{P} \slantsum{p} \mathcal{C}$ and $p'' > p$\}}\\
\bullet, & \text{otherwise}
\end{cases} \\ 
 &=C_i^{j+1}(p).
\end{align*}
Here, the first and third equalities are by the definition of the swapping process and the increasingness of the tableaux. The second equality is because 
\[
\min \text{$\{T_i^j(p'') \in \mathbb{Z}: p'' \in \mathcal{R}$ and $p'' > p$\}}
= \min \text{$\{C_i^j(p'') \in \mathbb{Z}: p'' \in \mathcal{P} \slantsum{p} \mathcal{C}$ and $p'' > p$\}},
\] 
as follows from the inductive (S.0), (S.1), and the assumption that $\mathbb{S}_i^j \neq \emptyset$. This proves (S.0).

It remains to show (S.1). The swapping process only affects labels with value $j$. We already have \[\{T_i^j(q) \in \mathbb{Z} : q \in \bigcup \mathcal{Q}_h \text{ and } T_i^j(q) \leq q_m\} = \{C_i^j(c) \in \mathbb{Z}: c \in \mathcal{C} \text{ and } C_i^j(c) \leq q_m\},\] so it just remains to show that 
\begin{align*}
j \in \{C_i^{j+1}(c) \in \mathbb{Z}: c \in \mathcal{C} &\text{ and } C_i^{j+1}(c) \leq q_m\} \\
&\Updownarrow \\
j \in \{T_i^{j+1}(q) \in \mathbb{Z} : q \in \bigcup \mathcal{Q}_h &\text{ and } T_i^{j+1}(q) \leq q_m\}.
\end{align*}

Either $T_i^{j+1}(p) = j$ or not.
If $T_i^{j+1}(p) = j$, then (S.0) for $T_i^{j+1}$ and $C_i^{j+1}$ (which is already established at this point) implies $C_i^{j+1}(p) =j$. Hence, the increasingness of $T_i^{j+1}$ and $C_i^{j+1}$ ensures that
\[
j \notin \{T_i^{j+1}(q) \in \mathbb{Z} : q \in \bigcup \mathcal{Q}_h \text{ and } T_i^{j+1}(q) \leq q_m\}
\]
and
\[
j \notin \{C_i^{j+1}(c) \in \mathbb{Z}: c \in \mathcal{C} \text{ and } C_i^{j+1}(c) \leq q_m\}.
\]

Otherwise, $T_i^{j+1}(p) \neq j$. Then, nothing could have been swapped into $p$ at this stage. Thus, since $p$ is the only element connecting $\mathcal{P}$ to $\bigcup \mathcal{Q}_h$ or $\mathcal{C}$, in this situation (S.1) for $T_i^{j+1}$ and $C_i^{j+1}$ is immediate from (S.1) for $T_i^j$ and $C_i^j$.
This completes the induction and shows that $\mathsf{rects}|_\mathcal {P} (T) \subseteq \mathsf{rects}|_\mathcal{P} (C)$. 

To show the reverse containment $\mathsf{rects}|_\mathcal {P} (C) \subseteq \mathsf{rects}|_\mathcal{P} (T)$, we follow the same strategy, except that we first remove any skewed out nodes in $\bigcup_h \mathcal{Q}_h$. Suppose $C \xrightarrow{\theta_1, \dots, \theta_n} V$. We must find a sequence of inner corners that yields a rectification $U$ of $T$ such that $U|_\mathcal{P} = V|_\mathcal{P}$. First, we remove any skewed nodes in $\bigcup_h \mathcal{Q}_h$, as follows. Let $T_0 = T$. Recursively define 
\[\alpha_{i+1} \coloneqq \ic(T_i) \cap (\bigcup_h \mathcal{Q}_h)\] and 
\[T_{i+1} \coloneqq \mathsf{Slide}_{\alpha_{i+1}}(T_i).\]
Let $k$ be least such that $\alpha_k = \emptyset$. 
Then $T_k$ has no skewed out nodes in $\bigcup_h \mathcal{Q}_h$. Finally, let
$\gamma_{i} \coloneqq \theta_i \cap \mathcal{P}$. Then $\mathsf{Slide}_{\gamma_1, ..., \gamma_n}(T_k) = U|_\mathcal{P}$. The proof is exactly the same as before except $C$, $V$, $T_k$, $\theta_i$, and $\gamma_i$ play the respective roles of $T$, $U$, $C$, $\gamma_i$, and $\theta_i$. 
This completes the proof of (1).

The proof for (2) is by induction on $n$.
The base case $n=1$ is the previously proven statement (1). For $n>1$, let 
\[
\mathcal{P}' = \mathcal{P} \slantsum{p_1} (\mathcal{Q}_1^1, \dots ,\mathcal{Q}_{m_1}^1)
\]
and let 
\[
\mathcal{S} =  \mathcal{P}' \slantsum{p_2} \mathcal{C}_2 \slantsum{p_3}\dots\slantsum{p_n} \mathcal{C}_n.
\]
By the inductive hypothesis, there is a skew increasing $\mathcal{S}$-tableau $T_\mathcal{S}$ such that 
\[
\mathsf{rects}|_{\mathcal{P}'}(T_S) = \mathsf{rects}|_{\mathcal{P}'}(T),
\]
so furthermore by restriction we have
\[
\mathsf{rects}|_{\mathcal{P}}(T_S) = \mathsf{rects}|_{\mathcal{P}}(T).
\]
Observe that 
\[
\mathcal{S} = 
\mathcal{P}' \slantsum{p_2} \mathcal{C}_2 \slantsum{p_3}\dots\slantsum{p_n} \mathcal{C}_n = (\mathcal{P} \slantsum{p_2} \mathcal{C}_2 \slantsum{p_3}\dots\slantsum{p_n} \mathcal{C}_n) \slantsum{p_1} (\mathcal{Q}_1^1, \dots ,\mathcal{Q}_{m_1}^1).
\]
Then, by (1), there is an increasing skew $\mathcal{P} \slantsum{p_1} \mathcal{C}_2 \slantsum{p_2}\dots\slantsum{p_n} \mathcal{C}_n$-tableau $T_\mathcal{C}$ such that 
\[
\mathsf{rects}|_{(\mathcal{P} \slantsum{p_2} \mathcal{C}_2 \slantsum{p_3}\dots\slantsum{p_n} \mathcal{C}_n)}(T_\mathcal{S}) = \mathsf{rects}|_{(\mathcal{P} \slantsum{p_2} \mathcal{C}_2 \slantsum{p_3}\dots\slantsum{p_n} \mathcal{C}_n)}(T_\mathcal{C})
\]
so furthermore by restriction
\[
\mathsf{rects}|_{\mathcal{P}}(T_\mathcal{S}) = \mathsf{rects}|_{\mathcal{P}}(T_\mathcal{C}).
\]
Thus,
\[
\mathsf{rects}|_{\mathcal{P}}(T_\mathcal{C}) = \mathsf{rects}|_{\mathcal{P}}(T_\mathcal{S}) = \mathsf{rects}|_{\mathcal{P}}(T),
\]
as desired.
\end{proof}

\begin{definition}
Let $\mathcal{P}$ be a poset and fix $p \in \mathcal{P}$.  
Let $U$ be a URT in $\mathcal{P}$. 
Then $U$ is a \textbf{$p$-chain unique rectification target} in $\mathcal{P}$ if $U$ is a URT in $\mathcal{P} \slantsum{p} \mathcal{C}$ for every chain poset $\mathcal{C}$. More generally, $U$ is a {\bf $\{p_1,\dots,p_n\}$-chain URT} in $\mathcal{P}$ if $U$ is a URT in $\mathcal{P} \slantsum{p_1} \mathcal{C}_1 \slantsum{p_2} \cdots \slantsum{p_n} \mathcal{C}_n$ for all pairwise disjoint chains $\mathcal{C}_1, \dots , \mathcal{C}_n$.
\end{definition}

Being a $p$-chain URT is a strictly stronger notion than being a URT. For an example of a URT that is not a $p$-chain URT, see Remark~\ref{rem:URT_not_pchain}.

\begin{proposition}
\label{Theorem p-chain URT rectifies uniquely on restriction}
Let $\mathcal{R}$ be the slant sum 
\[
\mathcal{R} \coloneqq \mathcal{P} \slantsum{p_1} (\mathcal{Q}_1^1, \dots ,\mathcal{Q}_{m_1}^1) \slantsum{p_2} \cdots \slantsum{p_n} (\mathcal{Q}_1^n,\dots, \mathcal{Q}^n_{m_n}),
\]
for $p_i$ distinct and $Q_i^j$ all pairwise disjoint with minimum elements.
Let $T$ be a skew increasing $\mathcal{R}$-tableau with rectifications $U$ and $V$.
If $U|_\mathcal{P}$ is a $\{p_1,..,p_n\}$-chain URT in $\mathcal{P}$, then $U|_\mathcal{P} = V|_\mathcal{P}$.
\end{proposition}
\begin{proof}
By Lemma~\ref{Lemma poset "sees" chain}, there exist chain posets $\mathcal{C}_1,\dots,\mathcal{C}_n$ and a skew increasing $\mathcal{P} \slantsum{p_1}\mathcal{C}_1 \slantsum{p_2}\cdots \slantsum{p_n} \mathcal{C}_n$-tableau $T_\mathcal{C}$ such that
\[
\mathsf{rects}|_\mathcal{P} (T) = \mathsf{rects}|_\mathcal{P} (T_\mathcal{C}).
\]
Since $U|_\mathcal{P}$ is a $\{p_1,..,p_n\}$-chain URT, we know $|\mathsf{rects}|_\mathcal{P} (T_\mathcal{C})| = 1 $, so $|\mathsf{rects}|_\mathcal{P} (T)| = 1$.
Hence $U|_\mathcal{P} = V|_\mathcal{P}$.
\end{proof}

\begin{proposition}
\label{slant sums with AB-chain URTs}
\gap
\begin{enumerate}
\item Let $\mathcal{R}$ be the slant sum $\mathcal{P} \slantsum{p}\mathcal{Q}$ and let $U$ be an increasing $\mathcal{R}$-tableau of straight shape.
Suppose $A \subseteq \mathcal{P}$ and $B \subseteq \mathcal{Q}$.
If $p \in A$ and $U|_\mathcal{P}$ is an $A$-chain URT in $\mathcal{P}$ and $U|_\mathcal{Q}$ is a $B$-chain URT in $\mathcal{Q}$, then $U$ is an $A \cup B$-chain URT in $\mathcal{R}$.
\item More generally, let
\[ \mathcal{R} \coloneqq \mathcal{P} \slantsum{p_1} (\mathcal{Q}_1^1, \dots ,\mathcal{Q}_{m_1}^1) \slantsum{p_2} \cdots \slantsum{p_n} (\mathcal{Q}_1^n,\dots, \mathcal{Q}^n_{m_n})\] and 
let $U$ be an increasing $\mathcal{R}$-tableau of straight shape.
Suppose $\{p_1,\dots, p_n\} \subseteq A \subseteq \mathcal{P}$ and $B_i^j \subseteq \mathcal{Q}_i^j$.
Set 
\[D \coloneqq A \cup \left( \bigcup_{i,j} B_i^j \right).\]
If $U|_\mathcal{P}$ is an $A$-chain URT in $\mathcal{P}$ and $U|_{\mathcal{Q}_i^j}$ is a $B_i^j$-chain URT in $\mathcal{Q}_i^j$ for each $i,j$, then $U$ is a $D$-chain URT in $\mathcal{R}$.
\end{enumerate}
\end{proposition}
\begin{proof}
For simplicity, we only explicitly prove part (1). The proof of part (2) follows the same strategy.

Let $\mathcal{C}$ be a poset formed by taking $\mathcal{R}$ and slant summing chains on top of elements in $A \cup B$. We must show that $U$ is a URT in $\mathcal{C}$. Hence, suppose some skew increasing $\mathcal{C}$-tableau $C$ rectifies to $U$ and $V$. Then, we must show $U = V$.
Since $U$ is an $A$-chain URT in $\mathcal{P}$, by Proposition \ref{Theorem p-chain URT rectifies uniquely on restriction}, we have that $U|_\mathcal{P} = V|_\mathcal{P}.$

It is easy to see that since $U$ and $V$ agree on $\mathcal{P}$, they must also agree on any chain $\mathcal{C}_a$ slant summed onto an element $a \in A \subseteq \mathcal{P}$. This is since, in any such chain, the labels of $\mathcal{C}_a$ in $U$ and $V$ must be exactly those labels of $\mathcal{C}_a$ in $T$ that have values greater than $U(a) = V(a)$. By increasingness of $U$ and $V$, these are necessarily written in increasing order along $\mathcal{C}_a$ in both tableaux. Thus, $U|_{\mathcal{C}_a} = V|_{\mathcal{C}_a}$.

Let $\mathcal{Q}_C$ be the principal order filter of $\mathcal{C}$ generated by $\hat{0}_\mathcal{Q}$.
Since $\mathcal{Q}_C$ is a funnel in $\mathcal{C}$, we may consider the tableaux $(T \to U)|_{\mathcal{Q}_C}$ and $(T \to V)|_{\mathcal{Q}_C}$. 
By Definition~\ref{definition corresponding}, we have $(T \to U)|_{\mathcal{Q}_C} = (T \to V)|_{\mathcal{Q}_C}$, since both tableaux defined to the the restriction of $T$ to the set 
\[\mathcal{E} \coloneqq \{q \in \mathcal{Q} : T(q) > U(p) \} = \{q \in \mathcal{Q} : T(q) > V(p) \},\] where the second equality is by recalling $U(p) = V(p)$ (since $p \in \mathcal{P}$) and noting that $p$ is the only element of $\mathcal{C}$ covered by $\hat{0}_\mathcal{Q}$.
By Proposition~\ref{Theorem corresponding respects slides}, $U|_{\mathcal{Q}_C}$ and $V|_{\mathcal{Q}_C} $ are rectifications of $(T \to U)|_\mathcal{Q} =(T \to V)|_\mathcal{Q}$.
However, since $U|_{\mathcal{Q}}$ is a $B$-chain URT in $\mathcal{Q}$, we have that $U|_{\mathcal{Q}}$ is a URT in $\mathcal{Q}_C$. Hence $U|_{\mathcal{Q}_C} = V|_{\mathcal{Q}_C}$.

Thus, we have shown that $U(c) = V(c)$ for all $c \in \mathcal{C}$, so $U=V$, as desired.
\end{proof}

The following result follows inductively from Corollary~\ref{Corollary slant sum tree base case}; we, however, prove it here as a useful demonstration of working with $p$-chain URTs in preparation for more sophisticated uses later.

\begin{corollary}
\label{Trees everything URT}
Let $\mathscr{T}$ be a tree. Let $U$ be any increasing $\mathscr{T}$-tableau of straight shape. Then $U$ is a URT in $\mathscr{T}$.
\end{corollary}
\begin{proof}
Let $n = |\mathscr{T}|$. Define $\mathscr{T}_1 \subseteq  \dots \subseteq \mathscr{T}_n$ such that for all $i$, $\mathscr{T}_i$ is an order ideal of $\mathscr{T}$ and $|\mathscr{T}_i| = i$. In particular, we have $\mathscr{T}_1 = \{\hat{0}_\mathscr{T}\}$ and $\mathscr{T}_n = \mathscr{T}$. We claim that for all $i$, $U|_{\mathscr{T}_i}$ is a $\mathscr{T}_i$-chain URT in $\mathscr{T}_i$. We work by induction on $i$.
First, we note that in any singleton poset $\mathcal{P}$, every increasing $\mathcal{P}$-tableau of straight shape is a $\mathcal{P}$-chain URT by Lemma \ref{Lemma unique rectification in bottom chain}.
Thus, since $|\mathscr{T}_1| = 1$, $U|_{\mathscr{T}_1}$ is a $\mathscr{T}_1$-chain URT in $\mathscr{T}_1$. Now suppose $i>1$. Let $t$ be the unique element in $\mathscr{T}_i \setminus \mathscr{T}_{i-1}$ and let $p$ be the unique parent of $t$ in $\mathscr{T}$. Then $\mathscr{T}_i = \mathscr{T}_{i-1} \; \slantsum{p} \{t\}$. By the inductive hypothesis, $U|_{\mathscr{T}_{i-1}}$ is a $\mathscr{T}_{i-1}$-chain URT in $\mathscr{T}_{i-1}$. Because $|\{t\}|=1$, $U|_{\{p\}}$ is a $\{t\}$-chain URT in $\{t\}$. Thus by Proposition \ref{slant sums with AB-chain URTs}, $U|_{\mathscr{T}_i}$ is a $\mathscr{T}_i$-chain URT in $\mathscr{T}_i$. This completes the induction. Hence $U|_{\mathscr{T}_n} = T$ is a URT in $\mathscr{T}_n = \mathscr{T}$.
\end{proof}

Trees are a particularly simple subfamily of the $d$-complete posets studied in this paper. Corollary~\ref{Trees everything URT} should be understood as a particularly strong version of Theorem~\ref{thm:main} for this special subfamily.

\section{Double-tailed diamonds}
\label{Section doubled tailed diamonds}
In this section, we investigate the $p$-chain unique rectification targets of certain posets, called double-tailed diamonds. This special family of $d$-complete posets plays a central role in the study of general $d$-complete posets. We will apply the results developed here to the general case in Section~\ref{Section $d$-complete posets}.

For $k \geq 3$, a {\bf double-tailed diamond} $\mathcal{D}(k)$ has $2k - 2$ elements, two of which are incomparable elements in the middle with chains of size $k-2$ above and below them. Figure~\ref{fig:double-tailed} illustrates the Hasse diagrams of some of these posets. It is easy to work out that any increasing tableau on any order ideal of a double-tailed diamond is a URT. (This is even explicitly observed in \cite[Proof of Theorem~3.12]{BS16}.) For application in Section~\ref{Section $d$-complete posets}, we need to strengthen this observation to the setting of $p$-chain URTs.

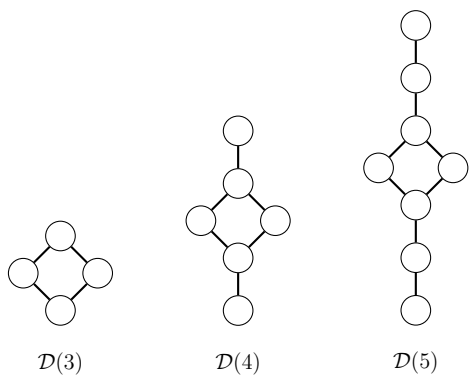
\begin{figure}[h]
\begin{tikzpicture}[every node/.append style={circle, draw=black, inner sep=0pt, minimum size=16pt}, every draw/.append style={black, thick}]
    \node (top) at (0,0) {};
    \node [below left of=top] (left)  {};
    \node [below right of=top](right) {};
    \node [below right of=left] (bottom) {};
    \node [draw=white, below  of=bottom] {$\mathcal{D}(3)$};

    \draw [black,  thick] (top) -- (left);
    \draw [black, thick] (top) -- (right);
    \draw [black, thick] (right) -- (bottom);
    \draw [black, thick] (left) -- (bottom);
\end{tikzpicture}
\qquad
\begin{tikzpicture}[every node/.append style={circle, draw=black, inner sep=0pt, minimum size=16pt}, every draw/.append style={black, thick}]
    \node [above of =top] (t1)  {};
    \node  (top) at (0,0) {};
    \node [below left  of=top] (left)  {};
    \node [below right of=top] (right) {};
    \node [below right of=left] (bottom) {};
    \node [below of=bottom] (b1) {};
    \node [draw=white, below  of=b1] {$\mathcal{D}(4)$};

    \draw [black,  thick] (top) -- (t1);
    \draw [black,  thick] (top) -- (left);
    \draw [black, thick] (top) -- (right);
    \draw [black, thick] (right) -- (bottom);
    \draw [black, thick] (left) -- (bottom);
    \draw [black, thick] (b1) -- (bottom);
\end{tikzpicture}
\qquad
\begin{tikzpicture}[every node/.append style={circle, draw=black, inner sep=0pt, minimum size=16pt}, every draw/.append style={black, thick}]
    \node [ above of =t1] (t2)  {};
    \node [ above of =top] (t1)  {};
    \node  (top) at (0,0) {};
    \node [below left  of=top] (left)  {};
    \node [below right of=top] (right) {};
    \node [below right of=left] (bottom) {};
    \node [below of=bottom] (b1) {};
    \node [below of=b1] (b2) {};
    \node [draw=white, below  of=b2] {$\mathcal{D}(5)$};

    \draw [black,  thick] (t2) -- (t1);
    \draw [black,  thick] (top) -- (t1);
    \draw [black,  thick] (top) -- (left);
    \draw [black, thick] (top) -- (right);
    \draw [black, thick] (right) -- (bottom);
    \draw [black, thick] (left) -- (bottom);
    \draw [black, thick] (b1) -- (bottom);
    \draw [black, thick] (b1) -- (b2);
\end{tikzpicture}
\qquad
\caption{The Hasse diagrams of the three smallest double-tailed diamonds.}\label{fig:double-tailed}
\end{figure}

To study the $p$-chain URTs of double-tailed diamonds, we introduce a \textbf{chained double-tailed diamond}. A chained double-tailed diamond is formed by slant summing a chain onto each of the two middle elements of the double-tailed diamond. We index the elements of a chained double-tailed diamond as shown in Figure~\ref{fig:chained_DTD}. We refer to the set of elements indexed as $\ell_k$ as the {\bf left chain} of the poset, and those indexed as $r_k$ as the {\bf right chain}.  In this notation, a chained double-tailed diamond corresponds to a triple of positive integers $m,n,p \geq 1$. We denote the chained double-tailed diamond for $(m,n,p)$ by $\mathcal{D}(m,n,p)$. In particular, $\mathcal{D}(k) = \mathcal{D}(1,k,1)$.

\begin{figure}[h]
\begin{tikzpicture}[every node/.append style={circle, draw=black, inner sep=0pt, minimum size=16pt}, every draw/.append style={black, thick}]
\node[label=right:{$t_n$}] (tn) at (0,0) {};
\node[draw=white,below of=tn] (tdots) {\vdots};
\node[label=right:{$t_2$}, below of=tdots] (t2) {};
\node[label=right:{$t_1$}, below of=t2] (t1) {};
\node[label=left:{$\ell_1$}, below left of=t1] (l1) {};
\node[label= right:{$r_1$}, below right of=t1] (r1) {};
\node[label=left:{$b_1$}, below right of=l1] (b1) {};
\node[label=left:{$b_2$}, below of=b1] (b2) {};
\node[draw=white,below of=b2] (bdots) {\vdots};
\node[label=left:{$b_n$}, below of=bdots] (bn) {};
\node[label=left:{$\ell_2$}, above left of=l1] (l2) {};
\node [draw=white, circle, above left of=l2,rotate=135] (ldots)  {\Large\ldots};
\node[label=left:{$\ell_m$}, above left of=ldots] (lm) {};
\node[label=right:{$r_2$}, above right of=r1] (r2) {};
\node [draw=white,circle, above right of=r2,rotate=45] (rdots)  {\Large\ldots};
\node[label=right:{$r_p$},above right of=rdots] (rp) {};

\draw [black,  thick] (tn) -- (tdots);
\draw [black,  thick] (t1) -- (t2);
\draw [black,  thick] (tdots) -- (t2);
\draw [black,  thick] (t1) -- (l1);
\draw [black,  thick] (l2) -- (l1);
\draw [black,  thick] (l2) -- (ldots);
\draw [black,  thick] (ldots) -- (lm);
\draw [black,  thick] (t1) -- (r1);
\draw [black,  thick] (r2) -- (r1);
\draw [black,  thick] (rdots) -- (r2);
\draw [black,  thick] (rdots) -- (rp);
\draw [black,  thick] (b1) -- (r1);
\draw [black,  thick] (b1) -- (l1);
\draw [black,  thick] (b1) -- (b2);
\draw [black,  thick] (bdots) -- (bn);
\draw [black,  thick] (bdots) -- (b2);
\end{tikzpicture}
\caption{Our standard indexing of the nodes of the chained double-tailed diamond poset $\mathcal{D}(m,n,p)$.}\label{fig:chained_DTD}
\end{figure}
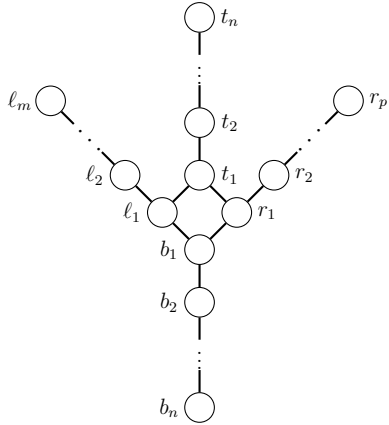

\begin{proposition}
\label{Theorem chained double tailed diamond rectifies uniquely}
Let $T$ be a skew increasing $\mathcal{D}(m,n,p)$-tableau.
Then, $T$ rectifies uniquely.
\end{proposition}
\begin{proof}
Suppose $T$ has shape $\nu/\lambda$.
If $\{\ell_1, r_1\} \not \subseteq \lambda$, then there are no choices to be made during rectification and hence $T$ rectifies uniquely. 
Thus, assume $\{\ell_1, r_1\} \subseteq \lambda$.
We can repeatedly perform slides on inner corners not equal to $\ell_1$ or $r_1$ because all the descendants of $\ell_1$ and $r_1$ are in disjoint chains and hence these slides clearly commute.
Thus, we may assume $T$ has exactly two inner corners $\ell_1$ and $r_1$. Write $I \coloneqq \{ \ell_1, r_1 \}$.
Let $s \in \mathbb{Z}$ be largest such that $t_s \in \nu$ (set $s=0$ if $t_1 \notin \nu$). 

We induct on $s$.
If $s=0$, then $\nu$ is a tree, and so $T$ rectifies uniquely by Corollary~\ref{Trees everything URT}.
Assume $s \geq 1$ and that the proposition holds for smaller $s$.

\medskip
\noindent
{\sf (Case 1: $T(t_1) \neq \min \range(T)$):}
Without loss of generality, we may assume $T(\ell_2) < T(t_1)$.
There are three possibilities for $\gamma \subseteq I$: either $\gamma = \{\ell_1\}$, $\gamma = \{r_1\}$, or $\gamma = \{\ell_1, r_1\}$. 
If we choose $\gamma = \{\ell_1\}$, then after $\mathsf{Slide}_{\{\ell_1\}}$ is completed, $r_1$ will be the unique inner corner; thus, the following slide is necessarily at $r_1$. 
Similarly, if we choose $\gamma = \{r_1\}$, the next slide is necessarily at $\ell_1$.

By routine case analysis, one checks that
\[\mathsf{Slide}_{\{\ell_1\}, \{r_1\}}(T) = \mathsf{Slide}_{\{r_1\}, \{\ell_1\}}(T) = \mathsf{Slide}_{\{r_1,\ell_1\}}(T)\]
in any of the various cases: $T(t_1) < T(r_2),T(t_1) = T(r_2),$ or $T(t_1) > T(r_2)$.
Thus, any choice made at this first step of rectifying $T$ yields the same tableau after one or two slides. That latter tableau has a unique rectification, as there are no further choices to be made. Hence, $T$ rectifies uniquely.

\medskip
\noindent
{\sf (Case 2: $T(t_1) = \min \range(T)$):}

\medskip
\noindent
{\sf (Case 2.1: $s=1$):} 
Then $T$ looks like the following.
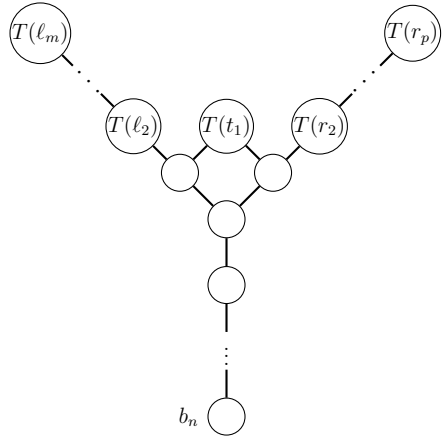
\[
\begin{tikzpicture}[every node/.append style={circle, draw=black, inner sep=0pt, minimum size=20pt}, every draw/.append style={black, thick},node distance=1.25cm]
\node[circle] (t1) at (0,0) {$T(t_1)$};
\node[circle, below left of=t1] (l1) {};
\node[circle, below right of=t1] (r1) {};
\node[circle, below right of=l1] (b1) {};
\node[circle, below of=b1] (b2) {};
\node[draw=white,below of=b2] (bdots) {\vdots};
\node[circle, label=left:{$b_n$}, below of=bdots] (bn) {};
\node[circle, above left of=l1] (l2) {$T(\ell_2)$};
\node [circle,draw=white,above left of=l2,rotate=135] (ldots)  {\Large\ldots};
\node[circle, above left of=ldots] (lm) {$T(\ell_m)$};
\node[circle, above right of=r1] (r2) {$T(r_2)$};
\node [circle,draw=white, above right of=r2,rotate=45] (rdots)  {\Large\ldots};
\node[circle, above right of=rdots] (rp) {$T(r_p)$};

\draw [black,  thick] (t1) -- (l1);
\draw [black,  thick] (l2) -- (l1);
\draw [black,  thick] (l2) -- (ldots);
\draw [black,  thick] (ldots) -- (lm);
\draw [black,  thick] (t1) -- (r1);
\draw [black,  thick] (r2) -- (r1);
\draw [black,  thick] (rdots) -- (r2);
\draw [black,  thick] (rdots) -- (rp);
\draw [black,  thick] (b1) -- (r1);
\draw [black,  thick] (b1) -- (l1);
\draw [black,  thick] (b1) -- (b2);
\draw [black,  thick] (bdots) -- (b2);
\draw [black,  thick] (bdots) -- (bn);
\end{tikzpicture}
\]

Consider any rectifications $U$ and $V$ of $T$.
By Lemma~\ref{Lemma unique rectification in bottom chain}, for all $k$ we have $U(b_k) = V(b_k)$ is the $(n-k+1)$st smallest element of $\range(T)$. Thus, since $T(t_1) = \min \range(T)$ by assumption, we have $U(b_n)=T(t_1)$. Hence, $t_1 \notin \domain(R)$ and so $U$ looks like: 
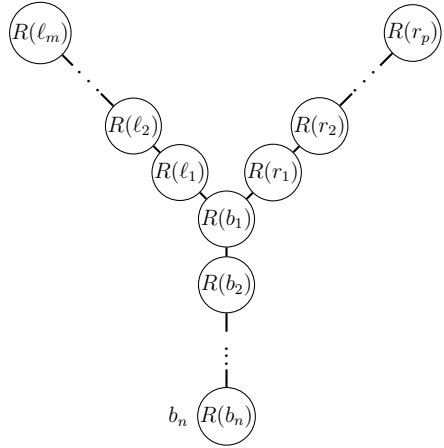
\[
\begin{tikzpicture}[every node/.append style={circle, draw=black, inner sep=0pt, minimum size=20pt}, every draw/.append style={black, thick},node distance=1.25cm]
\node[circle] (b1) at (0,0) {$R(b_1)$};
\node[circle, above left of=b1] (l1) {$R(\ell_1)$};
\node[circle, above right of=b1] (r1) {$R(r_1)$};
\node[circle, below of=b1] (b2) {$R(b_2)$};
\node[draw=white,below of=b2] (bdots) {\Large{\vdots}};
\node[circle, label=left:{$b_n$}, below of=bdots] (bn) {$R(b_n)$};
\node[circle, above left of=l1] (l2) {$R(\ell_2)$};
\node [circle,draw=white,above left of=l2,rotate=135] (ldots)  {\Large\ldots};
\node[circle, above left of=ldots] (lm) {$R(\ell_m)$};
\node[circle, above right of=r1] (r2) {$R(r_2)$};
\node [circle,draw=white, above right of=r2,rotate=45] (rdots)  {\Large\ldots};
\node[circle, above right of=rdots] (rp) {$R(r_p)$};

\draw [black,  thick] (l2) -- (l1);
\draw [black,  thick] (l2) -- (ldots);
\draw [black,  thick] (ldots) -- (lm);
\draw [black,  thick] (r2) -- (r1);
\draw [black,  thick] (rdots) -- (r2);
\draw [black,  thick] (rdots) -- (rp);
\draw [black,  thick] (b1) -- (r1);
\draw [black,  thick] (b1) -- (l1);
\draw [black,  thick] (b1) -- (b2);
\draw [black,  thick] (bdots) -- (b2);
\draw [black,  thick] (bdots) -- (bn);
\end{tikzpicture}
\]
Consulting these pictures, one observes that 
\[U(\ell_1) = \min (\{T(\ell_i): 2 \leq i \leq m \text{ and } T(\ell_i) > U(b_1)\}\] and similarly for $V$.
Since $U(b_1) = V(b_1)$ by Lemma~\ref{Lemma unique rectification in bottom chain}, this means $U(\ell_1) = V(\ell_1)$. 
Clearly, the labels of $U$ in the left chain of $\mathcal{D}(m,n,p)$ are exactly the labels on the left chain of $\mathcal{D}(m,n,p)$ in $T$ that are at least $U(\ell_1)$ written in increasing order. Since the same is true for $V$, we have $U(\ell_q) = V(\ell_q)$ for all $q$. The same argument shows $U(r_q) = V(r_q)$.
Thus, $U=V$ and $T$ rectifies uniquely.

\medskip
\noindent
{\sf (Case 2.2: $s \geq 2$):}
Let $U,V$ be rectifications of $T$.
As in Case 2.1, we have $U(b_n) = V(b_n) = T(t_1)$.  
Let $\mathcal{Q} \coloneqq \mathcal{D}(m,n,p) \setminus \{b_n\}$. We must show $U|_\mathcal{Q} = V|_\mathcal{Q}$.  Since $n \geq s \geq 2$, $\mathcal{Q}$ has a minimum and is a funnel of $\mathcal{D}(m,n,p)$. Then by Proposition~\ref{Theorem corresponding respects slides}, $U|_\mathcal{Q}$ is a rectification of $S_U \coloneqq (T \to U)|_\mathcal{Q}$ and $V|_\mathcal{Q}$ is a rectification of $S_V \coloneqq (T \to V)|_\mathcal{Q}$.
Since $U(b_n)= V(b_n) = T(t_1)$, it follows from the definition of corresponding tableaux that $S_U = S_V$. In fact, $S_U$ and $S_V$ are merely $T$ restricted by deleting all labels of value $T(t_1)$. Hence, write $S \coloneqq S_U = S_V$. It remains to show that $S$ rectifies uniquely in $\mathcal{Q}$.

Since $S$ is $T$ restricted by deleting all labels of value $T(t_1)  = \min \range(T)$ and the inner corners of $T$ are exactly $\{\ell_1, r_1 \}$, it follows that the inner corners of $S$ are exactly those elements $q \in \mathcal{Q}$ with $T(q)=T(t_1)$. The structure of $\mathcal{Q}$ ensures that $\ell_2$ and $r_2$ are the only two nodes $q$ besides $t_1$ that could possibly have the label $T(t_1)$. Let $J \subseteq \{t_1, \ell_2, r_r\}$ be the set of inner corners of $S$.  Clearly, since the various slides only affect disjoint chains, for any set partitions $(\gamma_1, \dots, \gamma_h)$ and $(\delta_1, \dots, \delta_k)$ of $J$, we have \[\mathsf{Slide}_{\gamma_h} \circ \dots \circ \mathsf{Slide}_{\gamma_1}(T) = \mathsf{Slide}_{\delta_k} \circ \dots \circ \mathsf{Slide}_{\delta_1}(T).\] Hence, without loss of generality, we may assume that we perform $\mathsf{Slide}_{\{t_1\}}$ first. That is, set  $S' \coloneqq \mathsf{Slide}_{\{t_1\}}(S)$ and observe that $\mathsf{Rects}(S') = \mathsf{Rects}(S).$

Finally, we must show that $S'$ rectifies uniquely. Let the shape of $S'$ be $\eta/\theta$. Recall, $s$ is defined to be the largest integer with $t_s \in \domain(T)$, so by the construction of $S$, we also have that $s$ is the largest integer with $t_s \in \domain(S)$. Hence, since $S' \coloneqq \mathsf{Slide}_{\{t_1\}}(S)$, we have that $t_s \notin \domain(S')$. Since $s \leq n$, this ensures that $\eta$ is an order ideal of 
\[\mathcal{Q} \setminus \{t_n\}  = \mathcal{D}(m,n,p) \setminus \{b_n, t_n\} = \mathcal{D}(m,n-1,p),\]
where the first equality is by the definition of $\mathcal{Q}$ and the second equality follows from $n \geq s \geq 2$.
Hence, $S'$ is a skew increasing $\mathcal{D}(m,n-1,p)$-tableau. Moreover, the largest $i$ such that $S'(t_i)$ is defined is $s-1$, so by the inductive hypothesis, $S'$ rectifies uniquely in $\mathcal{D}(m,n-1,p)$. Thus, $S$ rectifies uniquely in $\mathcal{Q}$, and so $W|_\mathcal{Q}$ is the same for all rectifications $W$ of $T$, so $T$ rectifies uniquely.
\end{proof}

\begin{corollary}
\label{Corollary chained double tailed diamond URTs}
Every increasing $\mathcal{D}(m,n,p)$-tableau of straight shape is a URT.
\end{corollary}
\begin{proof}
Immediate from Proposition~\ref{Theorem chained double tailed diamond rectifies uniquely}.
\end{proof}

\begin{corollary}
\label{everything is chain URT in DTD}
Every increasing $\mathcal{D}(n)$-tableau of straight shape is an $\{\ell_1, r_1\}$-chain URT.
\end{corollary}
\begin{proof}
Immediate from Corollary~\ref{Corollary chained double tailed diamond URTs}
\end{proof}

Proposition~\ref{Theorem chained double tailed diamond rectifies uniquely} is a special case of the following more general conjecture, for which we have some additional experimental evidence. (For the definition of `$d$-complete', see Section~\ref{Section $d$-complete posets}.)

\begin{conjecture}\label{conj:bottom_tree}
Let $\mathcal{P}$ be a $d$-complete poset with bottom tree $\mathcal{B}$. If $T$ is a skew increasing $\mathcal{P}$-tableau with rectifications $R$ and $S$, then we have $R|_\mathcal{B} = S|_\mathcal{B}$.
\end{conjecture}

Special cases of Conjecture~\ref{conj:bottom_tree} are key lemmas in \cite{Thomas.Yong:K} and \cite{Clifford.Thomas.Yong}. These lemmas have additional combinatorial applications \cite{Thomas.Yong:Plancherel,Pechenik:frames}; Conjecture~\ref{conj:bottom_tree} might have similar applications.

\section{$d$-complete posets and minuscule posets}
\label{Section $d$-complete posets}
In this section, we recall the definition of $d$-complete posets following \cite{Proctor:JACO}, and prove our main result Theorem~\ref{thm:main} regarding slant sum trees of minuscule posets. (We use, however, the convention of \cite{Proctor:algebra} regarding the orientation of our posets; the paper \cite{Proctor:JACO} uses the opposite convention, so the posets in \cite{Proctor:JACO} are the duals of ours.) The proofs presented in this section are all straightforward, relying on the technical results of the previous sections. We also develop appropriate terminology here to give precise interpretations of Conjectures~\ref{conj:URTs} and \ref{conj:geometry}.

If $x, y \in \mathcal{P}$, the {\bf interval} $[x, y]$ is the set $\{z \in \mathcal{P}: x \leq z \leq y\}$. We call an interval $[x, y]$ in $\mathcal{P}$ a {\bf $\mathcal{Q}$-interval} if it is isomorphic to the poset $\mathcal{Q}$.
We will be especially interested in $\mathcal{D}(k)$-intervals ($k \geq 3$). Let $\mathcal{D}_0(k) \coloneqq \mathcal{D}(k) \setminus \{t\}$, where $t$ is the minimal element of $\mathcal{D}(k)$. We will also be interested in $\mathcal{D}_0(k)$-intervals ($k \geq 4$). 
Examples of $\mathcal{D}_0(k)$-intervals are shown in Figure~\ref{fig:partial_DTD}; the corresponding posets $\mathcal{D}(k)$ are shown in Figure~\ref{fig:double-tailed}.

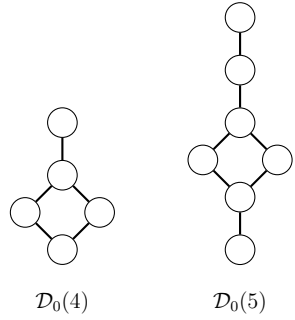
\begin{figure}[ht]
\begin{tikzpicture}[every node/.append style={circle, draw=black, inner sep=0pt, minimum size=16pt}, every draw/.append style={black, thick}]
    \node [circle] (top) at (0,0) {};
    \node [above of=top] (t1) {};
    \node [circle,below left  of=top] (left)  {};
    \node [circle,below right of=top] (right) {};
    \node [circle,below right of=left] (bottom) {};
    \node [draw=white,below  of=bottom] {$\mathcal{D}_0(4)$};
    \draw [black,  thick] (top) -- (left);
    \draw [black, thick] (top) -- (right);
    \draw [black, thick] (right) -- (bottom);
    \draw [black, thick] (left) -- (bottom);
    \draw [black, thick] (t1) -- (top);
\end{tikzpicture}
\qquad
\begin{tikzpicture}[every node/.append style={circle, draw=black, inner sep=0pt, minimum size=16pt}, every draw/.append style={black, thick}]
    \node [circle, above of =t1] (t2)  {};
    \node [circle, above of =top] (t1)  {};
    \node [circle] (top) at (0,0) {};
    \node [circle,below left  of=top] (left)  {};
    \node [circle,below right of=top] (right) {};
    \node [circle,below right of=left] (bottom) {};
    \node [circle,below of=bottom] (b1) {};
    \node [draw=white,below  of=b1] {$\mathcal{D}_0(5)$};
    \draw [black,  thick] (top) -- (t1);
    \draw [black,  thick] (top) -- (left);
    \draw [black, thick] (top) -- (right);
    \draw [black, thick] (right) -- (bottom);
    \draw [black, thick] (left) -- (bottom);
    \draw [black, thick] (b1) -- (bottom);
    \draw [black, thick] (t1) -- (t2);
\end{tikzpicture}
\caption{The Hasse diagrams of some small truncated double-tailed diamonds $\mathcal{D}_0(k)$.}\label{fig:partial_DTD}
\end{figure}

\begin{definition}
A poset $\mathcal{P}$ is {\bf $\mathcal{D}(3)$-complete} if it satisfies the following three conditions:
\begin{enumerate}
  \item anytime an element $z$ covers two distinct elements $x$ and $y$, there exists a fourth
element $w$ that $x$ and $y$ both cover;
  \item if $[w,z]$ is a $\mathcal{D}(3)$-interval in $\mathcal{P}$ with elements $\{w,x,y,z\}$, then $w$ is only covered by $x$ and $y$ in $\mathcal{P}$; and
  \item in such a $\mathcal{D}(3)$-interval, there is no $w' \neq w$ that both $x$ and $y$ cover.
\end{enumerate}
\end{definition}

Let $k \geq 4$. Suppose $[x, z]$ is a $\mathcal{D}_0(k)$-interval in which $y$ is the unique element with $y \gtrdot x$.
If there is no $w \in \mathcal{P}$ with $w \lessdot x$ such that $[w,z]$ is a $\mathcal{D}(k)$-interval, then $[x, z]$ is an {\bf incomplete} $\mathcal{D}_0(k)$-interval. If there exists $x' \neq x$ with $y \gtrdot x'$ such that $[x',z]$ is also a $\mathcal{D}_0(k)$-interval, then we say that $[x, z]$ and $[x', z]$ {\bf overlap}. 

\begin{definition}
For any $k \geq 4$, a poset $\mathcal{P}$ is {\bf $\mathcal{D}(k)$-complete} if it satisfies the following three conditions:
\begin{enumerate}
  \item there are no incomplete $\mathcal{D}_0(k)$-intervals;
  \item if $[w,z]$ is a $\mathcal{D}(k)$-interval, then $w$ is covered by only one element in $\mathcal{P}$; and
  \item there are no overlapping $\mathcal{D}_0(k)$-intervals.
\end{enumerate}
A poset $\mathcal{P}$ is {\bf $d$-complete} if it is $\mathcal{D}(k)$-complete for every $k \geq 3$.
\end{definition}

Briefly, the algebraic context of $d$-complete posets is as follows. (For further details, see \cite{Stembridge:FC,Proctor:algebra,CP12}.) Let $\Lambda$ be a dominant integral weight of a Kac-Moody Lie algebra $\mathfrak{g}$ with (generally infinite) Weyl group $W$. The Weyl group element $w \in W$ is called {\bf $\Lambda$-minuscule} if it can be written as a reduced word in the simple reflections as
\[
w = s_{i_1} s_{i_2} \cdots s_{i_\ell}, 
\]
so that for all $j$
\[(s_{i_{j+1}} \cdots s_{i_\ell} - s_{i_j} \cdots s_{i_\ell})\Lambda = \alpha_{i_j},
\]
where $\alpha_{i_j}$ is the simple root for $s_{i_j}$. (In fact, this property is independent of the choice of reduced word \cite[Proposition~2.1]{Stembridge:minuscule}.) Now, if $w$ is $\Lambda$-minuscule, then the interval $[{\rm id}, w]$ in Bruhat order is a distributive lattice. A poset $\mathcal{P}$ is $d$-complete if and only if it is isomorphic to the poset of join irreducibles of such a `$\Lambda$-minuscule distributive lattice'; equivalently, a poset $\mathcal{Q}$ is isomorphic to a Bruhat interval $[{\rm id}, w]$ for some $\Lambda$-minuscule $w$ if and only if $\mathcal{Q}$ is isomorphic to the poset of order ideals of a $d$-complete poset. Since Bruhat order on $W$ also describes containment of Schubert varieties in the Kac-Moody homogeneous space $X = G/B$, we have for $u,v \leq w$ all $\Lambda$-minuscule that the inclusion of Schubert varieties $X_u \subseteq X_v$ is equivalent to the reverse inclusion $\lambda_v \subseteq \lambda_u$ of the corresponding order ideals in the $d$-complete poset for $w$. In addition to their algebraic relations, $d$-complete posets enjoy a number of beautiful combinatorial properties, including an analogue of the classical hook-length formula (for a full proof of this fact, see \cite{Kim.Yoo}). Figure~\ref{fig:big_poset} shows an example of a reasonably large $d$-complete poset.

Say a $d$-complete poset is {\bf irreducible} if it is not the slant sum of two $d$-complete posets.
R.~Proctor \cite{Proctor:JACO} showed that all $d$-complete posets can be uniquely decomposed as a slant sum of irreducible $d$-complete components. In this decomposition, irreducible components are only slant summed onto special nodes of other irreducible components, called \emph{acyclic nodes} \cite{Proctor:JACO}; that is, if $\mathcal{P} = \mathcal{Q} \slantsum{q}\mathcal{R}$ is $d$-complete and $\mathcal{R}$ is irreducible, then $q$ is an acyclic node of its irreducible component. (We avoid giving the somewhat technical definition of acyclic nodes, as it is sufficient for our purposes to use Proctor's  explicit identification \cite{Proctor:JACO} of all acyclic nodes of all irreducible $d$-complete posets.) The irreducible $d$-complete posets are classified into $15$ (mostly infinite) families; we follow Proctor's numbering and naming conventions for these families from \cite{Proctor:JACO}. Of these $15$ families, only the components from families $1$--$9$ and $11$ have any acyclic nodes.

For a poset $\mathcal{P}$, we say an increasing $\mathcal{P}$-tableau $T$ of straight shape $\lambda \subseteq \mathcal{P}$ is {\bf minimally-labeled} if 
it is minimal among all increasing $\mathcal{P}$-tableaux of shape $\lambda$ under nodewise comparison of labels; that is, if $U$ is another increasing tableau of shape $\lambda$, then $U(x) \geq T(x)$ for all $x \in \lambda$. It is easy to see that there exists a unique minimally-labeled $\mathcal{P}$-tableau of each straight shape $\lambda$. We write $M_\lambda$ for this unique tableau. The precise version of Conjecture~\ref{conj:URTs} is the following.

\begin{conjecture}\label{conj:URT_precise}
Let $\mathcal{P}$ be $d$-complete and let $\lambda \subseteq \mathcal{P}$ be an order ideal. Then, the minimally-labeled increasing $\mathcal{P}$-tableau $M_\lambda$ of shape $\lambda$ is a unique rectification target.
\end{conjecture}

In light of the slant sum structure of $d$-complete posets, Conjecture~\ref{conj:URT_precise} would follow from Proposition~\ref{slant sums with AB-chain URTs} together with information about ($p$-chain) URTs in the $15$ families of irreducible $d$-complete posets. Specifically, it remains to show that
\begin{itemize}
\item for each irreducible $d$-complete poset $\mathcal{Q}$ with acyclic nodes, that $M_\lambda$ is a $p$-chain URT for each  order ideal $\lambda \subseteq \mathcal{Q}$ and each acyclic node $p \in \mathcal{Q}$, and that
\item for each irreducible $d$-complete poset $\mathcal{Q}$ without acyclic nodes, that $M_\lambda$ is a URT for each  order ideal $\lambda \subseteq \mathcal{Q}$.
\end{itemize}
Unfortunately, we are unable to establish the necessary results for some of these families; hence, we can only leverage 
Proposition~\ref{slant sums with AB-chain URTs} to prove a weaker version of Conjecture~\ref{conj:URT_precise}, namely Theorem~\ref{thm:main}. First, we recall the \emph{minuscule posets}, a special subset of $d$-complete posets. Except for some trivial instances, all minuscule posets are irreducible.

Algebraically, one obtains the minuscule posets as follows. Suppose the Kac-Moody group $G$ is in fact complex reductive. Put a partial order on the positive roots $\Phi^+$ of $G$ by taking the transitive closure of the covering relation $\alpha \lessdot \beta$ if and only if $\beta - \alpha$ is a simple root. The simple root $\delta$ is a {\bf minuscule root} if for every positive root $\alpha \in \Phi^+$, the multiplicity of $\delta^\vee$ in the simple coroot expansion of $\alpha^\vee$ is at most $1$. For each minuscule root, one obtains a corresponding {\bf minuscule poset} $\mathcal{P}_\delta$ by restricting the partial order on $\Phi^+$ to those positive roots that use $\delta$ in their simple root expansion. There is also a corresponding {\bf minuscule variety} obtained as the quotient $G / P_\delta$, where $P_\delta$ is the maximal parabolic subgroup associated to the minuscule root $\delta$. The minuscule poset $\mathcal{P}_\delta$ encodes the Schubert stratification of $G / P_\delta$; specifically, the Schubert varieties are naturally indexed by the order ideals of $\mathcal{P}_\delta$, and inclusions of order ideals correspond to reverse inclusions of Schubert varieties.

\begin{table}[ht]
\begin{center}
\begin{tabular} {|l|l|l|}
\hline
Minuscule poset & Minuscule variety & Irreducible $d$-complete classification \\
\hline
\hline
rectangle & Grassmanian & shapes (family 1) \\
\hline shifted staircase & orthogonal Grassmanian & shifted shapes (family 2) \\
\hline double-tailed diamond & quadric hypersurface & insets (family 4--special case)\\
\hline Cayley-Moufang swivel & octonion projective plane & swivels (family 8--special case)\\
\hline bat & Freudenthal variety & bat (family 12)\\
\hline
\end{tabular}
\end{center}
\caption{The $5$ families of minuscule posets are named in the first column. The second column identifies the corresponding minuscule homogeneous space. The third column shows how the minuscule posets fall into R.~Proctor's classification of irreducible $d$-complete posets from \cite{Proctor:JACO}.}
\label{table:minuscules}
\end{table}

Combinatorially, the minuscule posets are completely classified. Minuscule posets consist of three infinite families together with a pair of exceptional examples. This classification is given in Table~\ref{table:minuscules}, with examples shown in Figure~\ref{fig:minuscule}. One infinite family of minuscule posets is the {\bf rectangles}; combinatorially, these are the products $\mathcal{C}_i
\times \mathcal{C}_j$ of two chain posets. Another infinite family is the double-tailed diamonds studied in Section~\ref{Section doubled tailed diamonds}. The final infinite family is the {\bf shifted staircases}; identifying the chain $\mathcal{C}_i$ with the natural order on $\{ 1, \dots, i\}$, shifted staircases are of the form 
\[
\{ (x_1,x_2) \in \mathcal{C}_i \times \mathcal{C}_i : x_1 \geq x_2 \},
\]
with the order structure restricted from $\mathcal{C}_i \times \mathcal{C}_i$. For convenience, we will assume that shifted staircases have at least $10$ nodes, as the smaller shifted staircases coincide with small rectangles/double-tailed diamonds. Lastly, for the definitions of the exceptional {\bf Cayley-Moufang swivel} and {\bf bat}, see their Hasse diagrams depicted in the second row of Figure~\ref{fig:minuscule}. The acyclic nodes of the minuscule posets are also shown in Figure~\ref{fig:minuscule}; we will use the indexing of these nodes as $L$ and $R$, as in that figure.

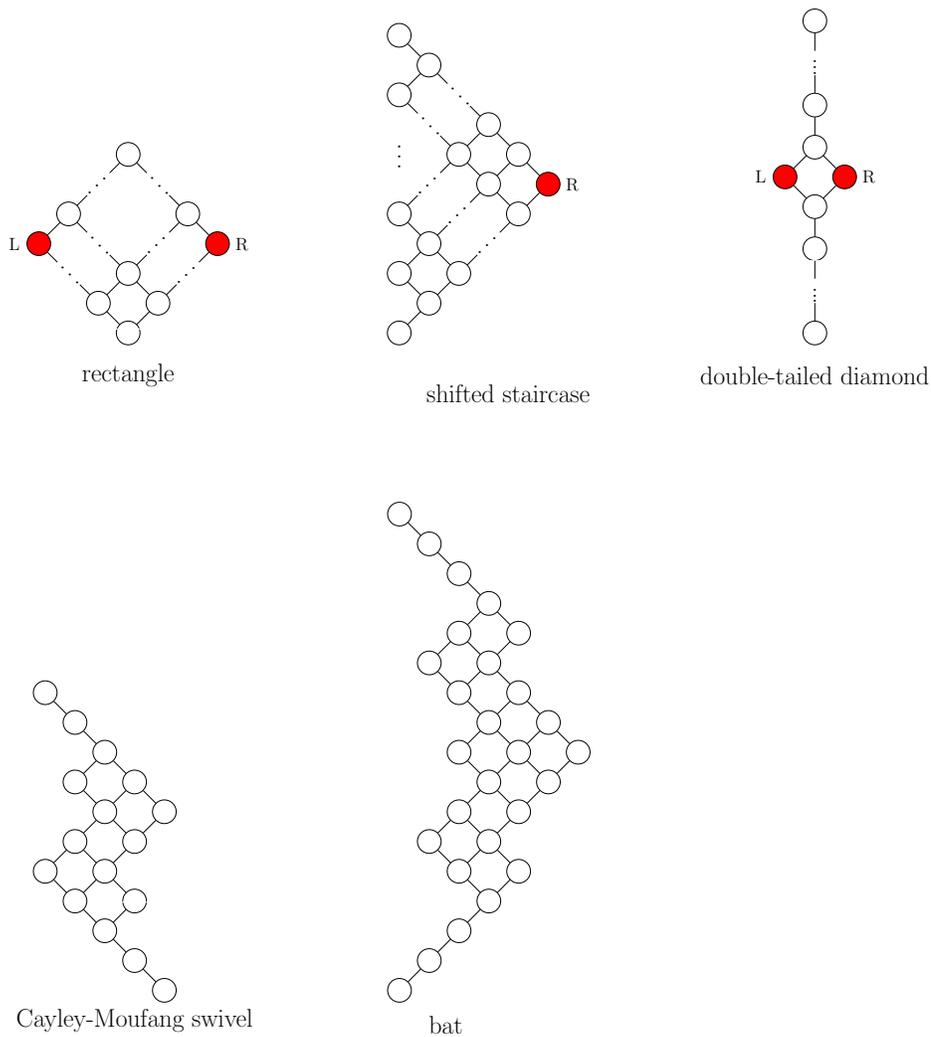
\begin{figure}[ht]
\begin{align*}
\Scale[0.8]{\begin{tikzpicture}[every node/.append style={circle, draw=black, inner sep=0pt, minimum size=16pt}, every draw/.append style={black, thick},anchor=base,baseline,]
\node (a1) at (0,0) {};
\node [above left of=a1] (a2) {};
\node [draw=white, above left of=a2,rotate=135] (a3) {\Large\ldots};
\node [fill=red,above left of=a3, label=left:{L}] (a4) {};
\node [above right of=a1](b1) {};
\node [above left of=b1] (b2) {};
\node [draw=white, above left of=b2,rotate=135] (b3) {\Large\ldots};
\node [above left of=b3] (b4) {};
\node [draw=white, above right of=b1, rotate = 45](c1) {\Large\ldots };
\node [draw=white,above left of=c1, rotate = 45] (c2) {\Large\ldots};
\node [draw=white, above left of=c2, rotate = 45] (c3) {};
\node [draw=white, above left of=c3, rotate = 45] (c4) {\Large\ldots};
\node [above right of=c1, fill=red, label= right:{R}](d1) {};
\node [above left of=d1] (d2) {};
\node [draw=white, above left of=d2,rotate=135] (d3) {\Large\ldots};
\node [above left of=d3] (d4) {};
\node [draw=white, below of= a1] {\Large rectangle};
\draw (a1)--(a2);
\draw (a2)--(a3);
\draw (a3)--(a4);
\draw (b1)--(b2);
\draw (b2)--(b3);
\draw (b3)--(b4);
\draw (d1)--(d2);
\draw (d2)--(d3);
\draw (d3)--(d4);
\draw (a1)--(b1);
\draw (a2)--(b2);
\draw (a4)--(b4);
\draw (c1)--(b1);
\draw (c2)--(b2);
\draw (c4)--(b4);
\draw (c1)--(d1);
\draw (c2)--(d2);
\draw (c4)--(d4);
\end{tikzpicture}}&&
\Scale[0.8]{\begin{tikzpicture}[every node/.append style={circle, draw=black, inner sep=0pt, minimum size=16pt}, every draw/.append style={black, thick},anchor=base,baseline,]
\node (a1) at (0,0) {};
\node [ above right of= a1] (b1) {};
\node [above right of= b1] (c1) {};
\node [draw=white,above right of= c1,rotate=45] (d1) {\Large\ldots};
\node [above right of= d1] (e1) {};
\node [fill=red,above right of= e1, label=right:{R}] (f1) {};
\node [fill=white,above left of= b1, label=left:{}] (b2) {};
\node [above left of= c1] (c2) {};
\node [above left of= c2] (c3) {};
\node [draw=white,above left of= d1,rotate=45] (d2) {\Large\ldots};
\node [draw=white,above left of= d2,rotate=45] (d3) {\Large\ldots};
\node [draw=white,above left of= d3,rotate=90] (d4) {\Large\ldots};
\node [above left of= e1] (e2) {};
\node [above left of= e2] (e3) {};
\node [draw=white,above left of= e3,rotate=135] (e4) {\Large\ldots};
\node [above left of= e4] (e5) {};
\node [above left of= f1] (f2) {};
\node [above left of= f2] (f3) {};
\node [draw=white,above left of= f3,rotate=135] (f4) {\Large\ldots};
\node [above left of= f4] (f5) {};
\node [above left of= f5] (f6) {};
\node [draw=white, below right = -0.1 and 0.7 of a1] {\Large shifted staircase};
\draw (a1) -- (b1);
\draw (b1) -- (c1);
\draw (c1) -- (d1);
\draw (d1) -- (e1);
\draw (e1) -- (f1);
\draw (b2) -- (c2);
\draw (c2) -- (d2);
\draw (d2) -- (e2);
\draw (e2) -- (f2);
\draw (c3) -- (d3);
\draw (d3) -- (e3);
\draw (e3) -- (f3);
\draw (e5) -- (f5);
\draw (b1) -- (b2);
\draw (c1) -- (c2);
\draw (c3) -- (c2);
\draw (e1) -- (e2);
\draw (e3) -- (e2);
\draw (e3) -- (e4);
\draw (e5) -- (e4);
\draw (f1) -- (f2);
\draw (f3) -- (f2);
\draw (f3) -- (f4);
\draw (f5) -- (f4);
\draw (f5) -- (f6);
\end{tikzpicture}}
&&
\Scale[0.8]{\begin{tikzpicture}[every node/.append style={circle, draw=black, inner sep=0pt, minimum size=16pt}, every draw/.append style={black, thick},anchor=base,baseline,]
    \node  (b2) at(0,0) {};
    \node [draw=white,above of =b2] (b1)  {\Large \vdots};
    \node [circle,above of=b1] (b0) {};
    \node [circle,above of=b0] (bottom) {};
    \node [circle,fill=red,above left of=bottom, label=left:{L}] (left)  {};
    \node [circle,fill=red,above right of=bottom, label=right:{R}] (right) {};
    \node [circle, above right of=left] (top) {};
    \node [circle,above of =top] (t0)  {};
    \node [draw=white,above of =t0] (t1)  {\Large \vdots};
    \node [circle, above of =t1] (t2)  {};
    \node [draw=white,below  of=b2] {\Large double-tailed diamond};
    \draw (t2) -- (t1);
    \draw  (t0) -- (t1);
    \draw (top) -- (t0);
    \draw (top) -- (left);
    \draw (top) -- (right);
    \draw  (right) -- (bottom);
    \draw (left) -- (bottom);
    \draw  (b0) -- (bottom);
    \draw  (b1) -- (b0);
    \draw (b1) -- (b2);
\end{tikzpicture}} \\
\Scale[0.8]{\begin{tikzpicture}[every node/.append style={circle, draw=black, inner sep=0pt, minimum size=16pt}, every draw/.append style={black, thick},anchor=base,baseline, ]
  \node (a1) at (0,0) {};
  \node [above left of= a1](a2)  {};
  \node [above left of= a2](a3)  {};
  \node [above left of= a3](a4)  {};
  \node [above left of= a4](a5)  {};
  \node [above right of= a3](b1)  {};
  \node [above left of= b1](b2)  {};
  \node [above left of= b2](b3)  {};
  \node [above right of= b2](c1)  {};
  \node [above left of= c1](c2)  {};
  \node [above left of= c2](c3)  {};
  \node [above right of= c1](d1)  {};
  \node [above left of= d1](d2)  {};
  \node [above left of= d2](d3)  {};
  \node [above left of= d3](d4)  {};
  \node [above left of= d4](d5)  {};
  \node [below left of=a1, draw=white] {\Large Cayley-Moufang swivel};
  \draw (a1)--(a2);
  \draw (a2)--(a3);
  \draw (a3)--(a4);
  \draw (a4)--(a5);
  \draw (b1)--(b2);
  \draw (b2)--(b3);
  \draw (c1)--(c2);
  \draw (c2)--(c3);
  \draw (d1)--(d2);
  \draw (d2)--(d3);
  \draw (d3)--(d4);
  \draw (d4)--(d5);
  \draw (a3)--(b1);
  \draw (a4)--(b2);
  \draw (a5)--(b3);
  \draw (c1)--(b2);
  \draw (c2)--(b3);
  \draw (c1)--(d1);
  \draw (c2)--(d2);
  \draw (c3)--(d3);
\end{tikzpicture}} &&
\Scale[0.8]{\begin{tikzpicture}[every node/.append style={circle, draw=black, inner sep=0pt, minimum size=16pt}, every draw/.append style={black, thick},anchor=base,baseline,]
  \node (a1) at (0,0) {};
  \node [above right of= a1](a2)  {};
  \node [above right of= a2](a3)  {};
  \node [above right of= a3](a4)  {};
  \node [above left of= a4](a5)  {};
  \node [above left of= a5](a6)  {};
  \node [above right of= a4](b1)  {};
  \node [above left of= b1](b2)  {};
  \node [above left of= b2](b3)  {};
  \node [above right of= b2](c1)  {};
  \node [above left of= c1](c2)  {};
  \node [above left of= c2](c3)  {};
  \node [above right of= c1](d1)  {};
  \node [above left of= d1](d2)  {};
  \node [above left of= d2](d3)  {};
  \node [above left of= d3](d4)  {};
  \node [above left of= d4](d5)  {};
  \node [above right of= d1](e1)  {};
  \node [above left of= e1](e2)  {};
  \node [above left of= e2](e3)  {};
  \node [above left of= e3](e4)  {};
  \node [above left of= e4](e5)  {};
  \node [above right of= e4](f1)  {};
  \node [above left of= f1](f2)  {};
  \node [above left of= f2](f3)  {};
  \node [above left of= f3](f4)  {};
  \node [above left of= f4](f5)  {};
  \node [below right = 0.2 and 0.4 of a1, draw=white] {\Large bat};
  \draw (a1)--(a2);
  \draw (a2)--(a3);
  \draw (a3)--(a4);
  \draw (a4)--(a5);
  \draw (a5)--(a6);
  \draw (b1)--(b2);
  \draw (b2)--(b3);
  \draw (c1)--(c2);
  \draw (c2)--(c3);
  \draw (d1)--(d2);
  \draw (d2)--(d3);
  \draw (d3)--(d4);
  \draw (d4)--(d5);
  \draw (e1)--(e2);
  \draw (e2)--(e3);
  \draw (e3)--(e4);
  \draw (e4)--(e5);
  \draw (f1)--(f2);
  \draw (f2)--(f3);
  \draw (f3)--(f4);
  \draw (f4)--(f5);
  \draw (a4)--(b1);
  \draw (a5)--(b2);
  \draw (a6)--(b3);
  \draw (c1)--(b2);
  \draw (c2)--(b3);
  \draw (c1)--(d1);
  \draw (c2)--(d2);
  \draw (c3)--(d3);
  \draw (e1)--(d1);
  \draw (e2)--(d2);
  \draw (e3)--(d3);
  \draw (e4)--(d4);
  \draw (e5)--(d5);
  \draw (e4)--(f1);
  \draw (e5)--(f2);
\end{tikzpicture}}
\end{align*}
\caption{Examples of the $5$ families of minuscule posets. The labeled red nodes mark the acyclic nodes  of these posets. The exceptional posets of the bottom row have no acyclic nodes.}
\label{fig:minuscule}
\end{figure}

We will only use the following proposition in the case $k=1$ of rectangles; however, for possible future use, we note that it is equally true for four of Proctor's other families: \emph{birds} (family 3), \emph{tailed insets} (family 5), \emph{banners} (family 6), and \emph{nooks} (family 7).

\begin{lemma}
\label{irreducibles that make URTs into p-chain URTs}
Let $k \in \{1,3,5,6,7\}$.
Let $\mathcal{P}$ be an irreducible $d$-complete poset from family $k$ and let $A \subseteq \mathcal{P}$ be the set of all acyclic nodes in $\mathcal{P}$. 
If a straight-shaped increasing $\mathcal{P}$-tableau $U$ is a URT for all posets in family $k$, then $U$ is a $A$-chain URT in all such posets.
\end{lemma}
\begin{proof}
Suppose $A = \{a_1, ..., a_k\}$. Let $i_1, ..., i_k$ be arbitrary positive integers.
Let $\mathcal{R}$ be the iterated slant sum $\mathcal{P} \slantsum{a_1} \mathcal{C}_{i_1} \slantsum{a_2} \dots \slantsum{a_k} \mathcal{C}_{i_k}$ of $\mathcal{P}$ with a collection of chains. Observe that $\mathcal{R}$ is an order ideal of a  larger poset in the same irreducible family. Thus, $U$ is a URT in $\mathcal{R}$, as desired.
\end{proof}

\begin{theorem}[\cite{BS16}]
\label{Minimally labeled URTs in Minuscules}
Let $\mathcal{P}$ be a minuscule poset. Then, for every order ideal $\lambda \subseteq \mathcal{P}$, the minimally-labeled increasing $\mathcal{P}$-tableau $M_\lambda$ of shape $\lambda$ is a URT in $\mathcal{P}$. \qed
\end{theorem}

\begin{corollary}
\label{Grassmainian chain URT}
Let $\mathcal{P}$ be a rectangle.
Let $M_\lambda$ be an minimally-labeled $\mathcal{P}$-tableau of straight shape.
Then, $M_\lambda$ is an $\{L,R\}$-chain URT in $\mathcal{P}$.
\end{corollary}
\begin{proof}
This follows from Lemma~\ref{irreducibles that make URTs into p-chain URTs} and Theorem~\ref{Minimally labeled URTs in Minuscules}.
\end{proof}

\begin{corollary}
\label{large OG chain URT}
Let $\mathcal{P}$ be a shifted staircase with at least $10$ nodes.
If $M_\lambda$ is a minimally-labeled $\mathcal{P}$-tableau of straight shape,
then $M_\lambda$ is an $\{R\}$-chain URT in $\mathcal{P}$.
\end{corollary}
\begin{proof}
If $\mathcal{S}$ is the slant sum $\mathcal{P} \slantsum{R} \mathcal{C}_j$ of $\mathcal{P}$ with a chain, then $\mathcal{S}$ is an order ideal of a larger shifted staircase, in which minimally-labeled tableaux are URTs by Theorem~\ref{Minimally labeled URTs in Minuscules}.
\end{proof}

In order to state the following, we adopt the convention that an \emph{$\emptyset$-chain URT} in $\mathcal{P}$ is just a URT in $\mathcal{P}$.

\begin{proposition}
\label{A chain URTs for minuscules}
Let $\mathcal{P}$ be a minuscule poset. Let $A$ be the set of acyclic nodes in  $\mathcal{P}$.
Let $M_\lambda$ be a minimally-labeled increasing $\mathcal{P}$-tableau of straight shape.
Then, $M_\lambda$ is an $A$-chain URT in $\mathcal{P}$.
\end{proposition}
\begin{proof}
If $\mathcal{P}$ is the Cayley-Moufang swivel or the bat, then it has no acyclic nodes, so $A = \emptyset$. Hence, in these cases, it suffices to verify that $M_\lambda$ is a URT in $\mathcal{P}$. This fact is a special case of Theorem~\ref{Minimally labeled URTs in Minuscules}.

If $\mathcal{P}$ is a rectangle, then $A = \{L,R\}$, and $M_\lambda$ is a $\{L,R\}$-chain URT  in $\mathcal{P}$ by Corollary~\ref{Grassmainian chain URT}.
If $\mathcal{P}$ is a double-tailed diamond, then $A = \{L,R\}$, and $M_\lambda$ is a $\{L,R\}$-chain URT in $\mathcal{P}$ by Corollary~\ref{everything is chain URT in DTD}.
Finally, if $\mathcal{P}$ is a shifted staircase with at least $10$ nodes,
then $M_\lambda$ is an $A$-chain URT in $\mathcal{P}$ by Corollary~\ref{large OG chain URT}.
\end{proof}

Proposition~\ref{slant sums with AB-chain URTs} allows us to extend Proposition~\ref{A chain URTs for minuscules} to show that minimally-labeled tableaux are unique rectification targets in iterated slant sums of minuscule posets.

\begin{theorem}
\label{Thm slant sum of minuscule URTs}
Let $\mathcal{P}$ be a $d$-complete poset. If $\mathcal{P}$ is an iterated slant sum of minuscule posets, then all minimally-labeled increasing $\mathcal{P}$-tableaux of straight shape are unique rectification targets.
\end{theorem}
\begin{proof}
We prove the stronger statement that all 
minimally-labeled increasing $\mathcal{P}$-tableaux of straight shape are $A$-chain URTs in $\mathcal{P}$, where $A$ denotes the set of acyclic nodes in $\mathcal{P}$.
We induct on the number $n$ of irreducible components in the slant sum decomposition of $\mathcal{P}$. For $\mathcal{Q}$ an irreducible component of $\mathcal{P}$, write $A_\mathcal{Q}$ for the set of acyclic nodes of $\mathcal{Q}$.

The base case, $n=1$ is provided by Proposition~\ref{A chain URTs for minuscules}.

Otherwise, $\mathcal{P}$ is the slant sum of irreducible components. One of these components contains the minimum $\hat{0}_\mathcal{P}$; call this component $\mathcal{M}$. By R.~Proctor's classification of acyclic nodes \cite{Proctor:JACO}, $\mathcal{M}$ has at most two acyclic nodes.
Then,
\[\mathcal{P} = \mathcal{M} \slantsum{L} \{ \mathcal{L}_1, \dots, \mathcal{L}_\ell \} \slantsum{R} \{\mathcal{R}_1, \dots, \mathcal{R}_r\},\]
where $L$ and $R$ are the acyclic nodes of $\mathcal{M}$ (if $L$ or $R$ is not an acyclic node, then we have $\ell = 0$ or $r=0$ respectively), and $\mathcal{L}_1, \dots, \mathcal{L}_\ell$ and $\mathcal{R}_1, \dots, \mathcal{R}_r$ are disjoint $d$-complete posets that are slant sum trees of minuscule components. Note that each $\mathcal{R}_i$ and $\mathcal{L}_j$ is a slant sum of strictly fewer than $n$ irreducible components. 

Suppose $T$ is a minimally-labeled increasing $\mathcal{P}$-tableau of straight shape. Then, $T|_{\mathcal{R}_i}$ is a minimally-labeled $\mathcal{R}_i$-tableau of straight shape (modulo shifting the alphabet), so by the inductive hypothesis, $T|_{\mathcal{R}_i}$ is an $A_{\mathcal{R}_i}$-chain URT in $\mathcal{R}_i$ for all $i$. Similarly, $T|_{\mathcal{L}_i}$ is an $A_{\mathcal{L}_i}$-chain URT in $\mathcal{L}_i$ for all $i$. Finally, $T|_\mathcal{M}$ is a minimally-labeled $\mathcal{M}$-tableau of straight shape, so by the inductive hypothesis it is $A_\mathcal{M}$-chain URT in $\mathcal{M}$. Thus, by Proposition~\ref{slant sums with AB-chain URTs}, we have that $T$ is an $A$-chain URT in $\mathcal{P}$, where $A$ is the set of acyclic nodes in $\mathcal{P}$.
\end{proof}

The following is the precise version of Theorem~\ref{thm:main}.

\begin{corollary}
\label{Cor slant sum of minuscule ideals URTs}
Let $\mathcal{P}$ be a $d$-complete poset. If $\mathcal{P}$ is an iterated slant sum of minuscule posets and $\mathcal{Q} \subseteq \mathcal{P}$ is an order ideal, then all minimally-labeled increasing $\mathcal{Q}$-tableaux of straight shape are unique rectification targets.
\end{corollary}
\begin{proof}
Let $M_\lambda$ be a minimally-labeled  increasing $\mathcal{Q}$-tableau of straight shape. Since $\mathcal{Q}$ is an order ideal of $\mathcal{P}$, $M_\lambda$ is also a minimally labeled increasing $\mathcal{P}$-tableau of straight shape. Hence by Theorem~\ref{Thm slant sum of minuscule URTs}, $M_\lambda$ is a unique rectification target in $\mathcal{P}$, so it is a unique rectification target in $\mathcal{Q}$.
\end{proof}

Finally, we recall the construction necessary to make precise sense of Conjecture~\ref{conj:geometry}. Let $\mathcal{P}$ be any poset satisfying the conclusion of Conjecture~\ref{conj:URT_precise}. Then, as in \cite[\textsection 3.5]{BS16}, we construct a \emph{combinatorial $K$-theory ring} associated to $\mathcal{P}$. Let $K(\mathcal{P})$ be the free abelian group on the set of order ideals of $\mathcal{P}$. Define a product structure on $K(\mathcal{P})$ by setting
\[
\lambda \cdot \mu \coloneqq \sum_\nu t_{\lambda,\mu}^\nu  \; \nu,
\]
where the Greek letters denote order ideals of $\mathcal{P}$ and 
$
t_{\lambda,\mu}^\nu
$
is defined to be $(-1)^{|\nu| - |\lambda| - |\mu|}$ times the number of skew increasing $\mathcal{P}$-tableaux of shape $\nu / \lambda$ that rectify to the minimally-labeled tableau $M_\mu$. (Since $M_\mu$ is by hypothesis a URT in $\mathcal{P}$, this number is well-defined.) By \cite[Proposition~3.17]{BS16}, this product structure makes $K(\mathcal{P})$ into a commutative associative   algebra with the empty order ideal as multiplicative identity. Conjecture~\ref{conj:geometry} claims then that, when $\mathcal{P}$ is $d$-complete, the structure constants of the algebra $K(\mathcal{P})$ coincide with corresponding $\Lambda$-minuscule Schubert structure constants of the $K$-theory ring $K(X)$, where $X = G/P$ is a Kac-Moody homogeneous space, $w \in W^P$ is a $\Lambda$-minuscule Weyl group element for $P$, and $\mathcal{P}$ is the poset of join irreducibles of the distributive lattice $[{\rm id}, w]$.

\section*{Acknowledgements}
The paper derives from a Summer 2017 DIMACS REU project. R.I. and M.Z. participated in this project under the direction of O.P. with funding provided by the Mathematics Department at Rutgers University. We would like to thank Anders Buch, Lazaros Gallos, and Parker Hund for their roles in organizing and running the REU program. 

O.P. is grateful to Bob Proctor and Alexander Yong for inspiring conversations, and to Jake Levinson for helpful comments on exposition.
O.P. was partially supported by an NSF Mathematical Sciences Postdoctoral Research Fellowship \#1703696.
\bibliographystyle{alpha}
\nocite{*}
\bibliography{bib}
\end{document}